\documentclass[12pt,a4paper]{article}
\usepackage[utf8]{inputenc}
\usepackage[T1]{fontenc}
\usepackage{amsmath}
\usepackage{amsfonts}
\usepackage{amssymb}
\usepackage{amsthm}
\usepackage{mathtools}
\usepackage{hyperref}
\usepackage{url}

\usepackage[top=2cm, bottom=2cm, left=2.5cm, right=2.5cm]{geometry}
\usepackage[francais,english]{babel}
\usepackage{fancybox}
\usepackage{tikz}
\usetikzlibrary{trees}
\newtheorem{thm}{Theorem}[section]
\newtheorem{lem}[thm]{Lemma}
\newtheorem{prop}[thm]{Proposition}
\newtheorem{cor}[thm]{Corollary}

\theoremstyle{definition}
\newtheorem{defn}{Definition}[section]

\newtheorem{conv}[defn]{Convention}

\theoremstyle{remark}
\newtheorem{rem}{Remark}

\newcommand\supp{\operatorname{supp}}
\newcommand\Per{\operatorname{Per}}
\newcommand\Tr{\operatorname{Tr}}
\newcommand\Coeff{\operatorname{Coeff}}
\allowdisplaybreaks

\begin{document}

\title{Markov loops in discrete spaces}
\author{Yinshan Chang \ Yves Le Jan\\
Equipe Probabilit\'{e}s et Statistiques\\
Universit\'{e} Paris-Sud\\
B\^{a}timent 425\\
91405 Orsay Cedex, France}
\date{}
\maketitle

\tableofcontents

\section{Introduction}
The main topic of  these notes are Markov loops, studied in the context of continuous time  Markov chains on discrete state spaces. We refer to \cite{loop} and \cite{Sznitman} for the  short "history" of the subject.
In contrast with these references, symmetry is not assumed, and more attention is given to the infinite case. All results are presented in terms of the semigroup generator. In comparison with \cite{loop}, some delicate proofs are given in more details or with a better method. We focus mostly on properties of the (multi)occupation field but also included some results about loop clusters (see \cite{sophie} in the symmetric context) and spanning trees.

\section{Preliminaries}
In this section, we present some basic results about continuous time  Markov chains, including a discrete version of Feynman-Kac formula and the transformation by time change.
\subsection{Notations}
\begin{enumerate}
\item Suppose $E_1,E_2$ are two countable sets, $(A_{j}^{i},i\in E_1, j\in E_2)$ is a matrix. For $F_1\subset E_1$ and $F_2\subset E_2$, let $(A|_{F_1\times F_2}, i\in F_1, j\in F_2)$ be the sub-matrix defined by $(A|_{F_1\times F_2})_{j}^{i}=A_{j}^{i}$. By convention, the absolute value $|A|$ will denote the matrix: $(|A|)^i_j=|A^i_j|$.
\item $\mathcal{E}(\lambda),\lambda\in[0,\infty]$ denotes a random variable, exponentially distributed with parameter $\lambda$ with the convention that $\mathcal{E}(0)=\infty$ and $\mathcal{E}(\infty)=0$.
\item If $k$ is a non-negative finite function on the state space $S$, $M_k$ will denote the matrix, $(M_k)^x_y=k(x)\delta^x_y$ where $\delta^x_y=1$ iff $x=y$.
\item $x\in\mathbb{R}^n$ can be extended to an periodic series, $x^{nm+k}=x^k, m\in\mathbb{Z}, k=0,\ldots,n-1$. Given $x\in\mathbb{R}^n$, each time we write $x_{n+j}$, we extend $x$ to the $n$-periodical series.
\item For any countable set $A$, $\#A$ and $|A|$ will denote the number of elements in $A$.
\item Let $\mathfrak{S}_k$ be the collection of permutations on $\{1,\ldots, k\}$ and $S$  some state space. For a permutation $\sigma\in\mathfrak{S}_{k}$ and $x=(x_1, \ldots, x_k)\in S^k$, define $\sigma(x)=(x_{\sigma^{-1}(1)},\ldots,x_{\sigma^{-1}(n)})$. Accordingly, a permutation $\sigma$ can be viewed as a function from $S^k$ to $S^k$. Define the circular permutation $r_j$ as follows: $r_j(1,\ldots,k)=(j+1, \ldots, k, 1,\ldots, j)$. Define $\mathfrak{R}_{k}$ to be the subset of $\mathfrak{S}_k$ consisting of circular permutations on $\{1, \ldots, k\}$. Note that $\sigma$ plays two roles, a function on $\{1,\ldots, k\}$ mapping an integer to another integer and a function on some $S^k$ mapping a $k$-uple to another $k$-uple (for example, $r_1(2,1,3,4)=(4,2,1,3)\neq (r_1(2),r_1(3),r_1(1),r_1(4))=(3,4,2,1)$). We have $\sigma_1(\sigma_2(x_1,\ldots,x_n))=(x_{(\sigma_1\circ\sigma_2)^{-1}(1)},\ldots,x_{(\sigma_1\circ\sigma_2)^{-1}(n)})$.
\end{enumerate}

\subsection{Minimal continuous-time sub-Markov chain in a countable space}
Let $S$ be a countable set equipped with the discrete topology. Add an additional cemetery point $\partial$ to $S$ and set $\overline{S}=S\bigcup\{\partial\}$  (compactification).

\begin{defn}[Generator]
A matrix $L=(L^x_y,x,y\in S)$ is called a sub-Markovian (Markovian resp.) generator iff
$$\begin{array}{llll}
0\leq-L^x_x<\infty & \quad\text{for all }x\in S,\\
L^x_y\geq 0 & \quad\text{for all }x\neq y,\\
\sum\limits_{j}L^x_y\leq 0\quad  (\sum\limits_{j}L^x_y=0\text{ resp.}) & \quad\text{for all }x\in S.
\end{array}$$
In case $L^x_x<0$, set $Q^x_y=\frac{L^x_y}{-L^x_x}\,\text{ for }\, x\neq y$ and  $Q^x_x=0$. In case $L^x_x=0$, set $ Q^x_y=\delta^x_y$.
\end{defn}

\begin{conv}
A sub-Markovian generator $L$ on $S$ can be extended to a Markovian generator $\overline{L}$ on $\overline{S}$ as follows:
$\overline{L}^x_y=L^x_y$ for $x, y\in S$,  $\overline{L}^x_\partial=-\sum\limits_{y\in S}L^x_y$ for $x\in S$, $\overline{L}^\partial_x=0$ for $x\in \overline{S}$.
\end{conv}

\subsubsection*{Construction of the probability on the space of right continuous\footnote{In a discrete space, any right-continuous Markov chain has left limit in its lifetime $[0,\zeta[$ if the path stays at the cemetery $\partial$ after there has been infinitely many jumps. Besides, on $\zeta<\infty$, the left limit at time $\zeta$ is the cemetery point for the process.} paths}
Let $\mu$, a probability measure on $S$, be the initial distribution. Let $\xi_0$ be a random variable with distribution $\mu$ and  $(\tau_{ix}, i\in\mathbb{N},x\in S)$ be independent random variables, exponentially distributed with parameter $-L^x_x$. Let $(J_{ix},i\in\mathbb{N},x\in S)$ be independent random variables such that for $y\in S$
$$\mathbb{P}(J_{ix}=y)=Q^x_y.$$
Moreover, assume that $\xi_0$, $\tau=(\tau_{ix}, i\in\mathbb{N},x\in S)$ and $J=(J_{ix},i\in\mathbb{N},x\in S)$ are independent. For any configuration of $(\mu,\tau, J)$, recursively define:
$$\begin{array}{ll}
\xi_n=J_{n\xi_{n-1}}\text{ for } n\geq 1 & \text{ (discrete Markov chain)}\\
T_0=0,T_{n+1}=T_{n}+\tau_{n\xi_{n}} & \text{ (jumping time)}\\
\zeta=\lim\limits_{n\rightarrow\infty}T_n & \text{ (explosion time)}.
\end{array}$$
Then define the path as follows:
$$\begin{array}{ll}
X_t=\xi_i & \text{ for } T_i\leq t<T_{i+1}, \\
X_t=\partial & \text{ for } t\geq \zeta.
\end{array}$$

\begin{thm}[\textbf{Markov Property}]
Set $(P_t)^x_y=\mathbb{P}[X_t=y|X_0=x]$. Use $\mathbb{P}^{\mu}$ to stand for the law of the process $(X_t, t\geq 0)$. $(X_t,t\geq 0)$ defined above is a Markov process with initial distribution $\mu$. Its semi-group will be denoted $P_t$ and $(P_t)^x_y$ is right-continuous in $t$.
\end{thm}

The following theorem is taken from the book \cite{NorrisMR1600720}.
\begin{thm}\
\begin{itemize}
\item[a)]\textbf{Backward Equation.}\\
$P_t$ is the minimal non-negative solution of the backward equation:
\begin{align*}
&\frac{dP_t}{dt}=LP_t\\
&P_0=I \text{ (identity)}.
\end{align*}
\item[b)] \textbf{Forward Equation.}\\
$P_t$ is the minimal non-negative solution of the forward equation:
\begin{align*}
&\frac{dP_t}{dt}=P_tL\\
&P_0=I \text{ (identity)}.
\end{align*}
\end{itemize}
(These equations are viewed as an infinite system of differential equations.)
\end{thm}

\begin{rem}
The process we constructed is minimal in the sense of its semi-group as the solution of the forward backward equations. In a more probabilistic language, it is the least conservative process. To be more precise, for any sub-Markovian process with generator $L$, if we kill the process as long as it jumps infinitely many times, we get the minimal sub-Markov process with generator $L$.
\end{rem}

\begin{defn}
The potential $V$ is defined as follows:
$$V^x_y=\mathbb{E}^x[\int\limits_{0}^{\infty}1_{\{X_t=y\}}\,dt]=\int\limits_{0}^{\infty}(P_t)^x_y\,dt.$$
\end{defn}

\begin{thm}[\textbf{Feynman-Kac}]
For a non-negative function $k$ on $S$, define
$$(P_{t,k})^x_y=\mathbb{E}^x\left(e^{-\int\limits_{0}^{t}k(X_s)\,ds}1_{\{X_t=y\}}\right).$$
Then, it is the minimal positive solution of the following equation:
\begin{equation*}
\frac{\partial u}{\partial t}(t,x)=(L-M_k)u(t,x)\quad\footnote{Recall that $(M_kf)(x)=k(x)f(x)$.}
\end{equation*}
with initial condition $u(0,x)=\delta^x_y$. We denote by $V_k$ the associated potential. Denote by $\mathbb{P}_{k}$ the law of the canonical minimal Markov process with generator $L-M_{k}$. Then,
$$\left.\frac{d\mathbb{P}_{k}}{d\mathbb{P}}\right|_{\mathcal{F}_t}=e^{-\int\limits_{0}^{t}k(X_s)\,ds}$$
where $\mathcal{F}_t=\sigma(X_s,s\in [0,t])$.
\end{thm}

\begin{prop}
Suppose $V$ is transient, i.e. $V^x_y<\infty$ for all $x$ and $y$, then $LV=VL=-Id$.
\end{prop}
\begin{thm}[\textbf{Resolvent equation}]\
The following identities hold:

\begin{itemize}
\item[a)] $V_k+VM_kV_k=V$.
\item[b)] $V_k+V_kM_kV=V$.
\item[c)] $V_kM_kV=VM_kV_k$.
\end{itemize}
\end{thm}

\subsection{The time change induced by a non-negative function}
Let $(X_t, t\geq 0)$ be a minimal Markovian process on $S$, with generator $L$ and lifetime $\zeta$.

Given $\lambda: S\rightarrow [0,\infty]$, define
$$A_t=\int\limits_{0}^{t\wedge\zeta}\lambda(X_s)\,ds,\quad\sigma_t=\inf\{s\geq 0, A_s>t\},\quad \hat{\zeta}=\inf\{s\geq 0, \sigma_s=\sigma_{\infty}\}$$
with the convention that $\inf\phi=\infty$. Then, $ \sigma_t$ are stopping times for all $t$ and they are right-continuous with respect to $t\geq 0$. Set $Y_t=X_{\sigma_t}$ for $0\leq t\leq \hat{\zeta}$ and let $Y$ be killed at time $\hat{\zeta}$. By the strong Markov property, $Y_t$ is also a c\`{a}dl\`{a}g sub-Markov process with lifetime $\hat{\zeta}$. It could be constructed directly from its generator $\hat{L}$ as before.

\bigskip

\begin{prop}\label{Expression of the generator for the time changed process}\
\begin{itemize}
\item[a)] If $0<\lambda<\infty$, then $\displaystyle{\hat{L}^x_y=\frac{L^x_y}{\lambda_x}}$ (change of jumping rates).
\item[b)] If $\lambda=1_A+1_{A^c}\cdot\infty$, then
$$\hat{L}^x_y=\left\{\begin{array}{ll}
L^{x}_y & \text{ for } x, y\in A^c\\
0 & \text{ elsewhere.}
\end{array}\right.$$
($Y$ is the restriction of $X$ to  $A$.)
\item[c)] If $\lambda=1_A$, $Y$ is called the trace of $X$ on $A$. The generator $\hat{L}$ of $Y$ will be denoted by $L_{A}$. In this case, $(Y_t, t\geq 0)$ has the same potential as $(X_t,t\geq 0)$.  On $A\times A$:
\begin{equation*}
V^x_y=\mathbb{E}^x[\int\limits_{0}^{\zeta}1_{\{X_s=y\}}\,ds]=\mathbb{E}^x[\int\limits_{0}^{\hat{\zeta}}1_{\{Y_s=y\}}\,ds].
\end{equation*}
Let $T_1$ be the first jumping time and $T_{1,A}=\inf\{s\geq T_1, X_s\in A\}$. \\
Define $(R^A)^x_y=\mathbb{E}^x[X_{T_{1,A}}=y,T_{1,A}<\infty]$ for $y\in S$ and $(R^A)^x_\partial=1-\sum\limits_{y}(R^A)^x_y$. Then, the generator $L_A$ of $Y$ satisfies:
$$(L_A)^x_x=L^x_x(1-(R^A)^x_x) \text{ and } (L_A)^x_y=-L^x_x(R^A)^x_y\;  \text{ for } x\neq y.$$
\end{itemize}
\end{prop}
\begin{proof}
Define $T_A=\inf\{t\geq 0, X_t\in A\}$ and $(H_A)_y^x=\mathbb{E}^x[X_{T_A}=y,T_A<\infty]$. As usual, set $\displaystyle{Q^x_y=-\frac{L^x_y}{L^x_x}}$ for $y\neq x$, $Q^x_x=0$ and $Q^x_y=\delta^x_y$ if $L^x_x=0$. For any subset $B$ of $S$, define $(J_B)^x_y=1_{\{x\in B\}}\delta^{x}_{y}$, $G_B=I+QJ_B+QJ_BQJ_B+\cdots$. Then $$H_A=J_A+J_{A^c}QH_A=J_A+J_{A^c}QJ_A+J_{A^c}QJ_{A^c}QJ_A+\cdots.$$
Next, we see that $(R^A)^x_y=\mathbb{E}^x[X_{T_{1,A}}=y]=Q^{x}_{y}+\sum\limits_{z\in A^c}Q^x_z(H_A)^z_y=(G_{A^c}QJ_A)^x_y$ for $x,y\in A$. Then, $Y$ can be described as follows: from $x$, it waits for an $\mathcal{E}(-L^x_x)$-time, then jumps to $y\in A\cup\{\partial\}$ according to $(R^A)^x_y$ (it does not actually jump if $y=x$). Finally, it is not hard to see that $(L_A)^{x}_{x}=L^x_x(1-(R^A)_x^x)$ and $ (L_A)^x_y=-(R^A)^x_y L^x_x$ for $y\neq x$.
\end{proof}

\begin{defn}\label{defn:the trace of a Markov process and the restriction of a Markov process}
For $A\subset S$, define $V_{A}=V|_{A\times A}$.  $V_A$ is the potential of the trace of the Markov process on $A$ and $L_{A}=-(V_A)^{-1}$ is its generator. Let $L^{A}=L|_{A\times A}$ denote the generator of the Markov process restricted in A (i.e. killed at entering $A^c$) and let $V^{A}=(-L^A)^{-1}$ be its potential.
\end{defn}

\begin{prop}\label{trace and change of time}
Assume that $V$ is transient, $\chi$ is a non-negative function on $S$ and that $F\subset S$ contains the support of $\chi$. Then, $(V_\chi)_F=(V_F)_{\chi}$.
\end{prop}
\section{Loops and Markovian loop measure}
In this section, we introduce the loop measure associated with a continuous time Markov chain. Its properties under various transformations (time change, trace, restriction, Feynman-Kac) are studied as well as the associated occupation and multi-occupation field.
\subsection{Definitions and basic properties}
\begin{defn}[\textbf{Based loops}]\label{defn:based loop}
A based loop $l$ is an element $(\xi_1,\tau_1,\ldots,\xi_p,\tau_p,\xi_{p+1},\tau_{p+1})$ in $\bigcup\limits_{p\in\mathbb{N}}(S\times]0,+\infty[)^{ p+1}$ such that $\xi_{p+1}=\xi_1$ and $\xi_{i+1}\neq\xi_{i}$ for $i=1,\ldots, p$. We call $p$ the number of jumps in $l$ and denote it by $p(l)$. Define $T=\tau_1+\cdots+\tau_{p+1}, T_0=0, T_i=\tau_1+\cdots+\tau_i$.
Then, a based loop can be viewed as a c\`{a}dl\`{a}g piecewise constant path $l$ on $[0,T]$ such that $l(t)=\xi_{i+1}$ for $t\in[T_i,T_{i+1}[,i=1,\ldots,p$ and $l(T)=\xi_{p+1}=\xi_1$. Clearly, we have $l(T)=l(T-)$.
\end{defn}
Let $\mathbb{P}^{x}$ be the law of the minimal sub-Markovian process started from $x$ with semi-group $(P_{t},t\geq 0)$ (or with generator $L$ equivalently). It induces a probability measure on the space of paths $l$ indexed by $[0,t]$, namely $\mathbb{P}^{x}_{t}$. $\mathbb{P}^{x}_{t}$ is carried by the space of paths with finite many jumps such that $l(0)=l(0+)=x$. Define the non-normalized bridge measure $\mathbb{P}_{t,y}^{x}$ from $x$ to $y$ with duration time $t$ as follows:  $\mathbb{P}_{t,y}^{x}(\cdot)=\mathbb{P}_{t}^{x}(\cdot\cap 1_{\{l(t)=y\}})$.
\begin{defn}
The measure on the based loops is defined as
$\mu^b=\sum\limits_{x\in S}\int\limits_{0}^{\infty}\frac{1}{t}\mathbb{P}_{t,x}^{x}\,dt$.
\end{defn}

\begin{prop}[Expression of the based loop measure]\label{expression of the based loop measure}
For $k\geq 2$,
\begin{align*}
&\mu^b(p(l)=k, \xi_1=x_1, \ldots,   \xi_k=x_k, \xi_{k+1}=x_{k+1},\tau_1\in dt^1, \ldots, \tau_{k+1}\in dt^{k+1})\\
&=1_{\{x_1=x_{k+1}\}}Q_{x_2}^{x_1}\cdots Q_{x_1}^{x_k} \frac{1}{t^1+\cdots+t^{k+1}}(-L^{x_1}_{x_1})e^{L_{x_1}^{x_1}t^1}\cdots  (-L^{x_k}_{x_k})e^{L_{x_k}^{x_k}t^k}e^{L_{x_{k+1}}^{x_{k+1}}t^{k+1}}\,dt^1\cdots  dt^{k+1}
\end{align*}
For $k=1$,
$$\mu^b(p(l)=1, \xi=x,\tau\in dt)=\frac{1}{t}e^{L^x_x t}\,dt$$
\end{prop}
\begin{proof}
For $k\geq 2$ and all sequence of positive measurable functions $(f_i,i\geq 1)$,
denote by $(*)$ the value of $\mu^b(p(l)=k, \xi_1=x_1, \ldots, \xi_k=x_k, \xi_{k+1}=x_{k+1}, f_1(\tau_1)\cdots f_{k+1}(\tau_{k+1}))$.
\begin{align*}
(*)=&\int\limits_{0}^{\infty}\frac{dt}{t}\sum\limits_{x\in S}\mathbb{P}^{x}_{t,x}[p(l)=k, \xi_1=x_1, \ldots, \xi_{k+1}=x_{k+1}, f_1(\tau_1)\cdots f_{k}(\tau_{k}) f_{k+1}(t-\sum\limits_{j=1}^{k}\tau_j)]\\
=&\begin{multlined}[t]
\int\limits_{0}^{\infty}\frac{dt}{t}\sum\limits_{x\in S}\mathbb{P}^x_t[p(l)=k, \xi_1=x_1, \ldots, \xi_{k+1}=x_{k+1} ,\\
f_1(\tau_1) \cdots f_{k}(\tau_{k})f_{k+1}(t-\sum\limits_{j=1}^{k}\tau_j) ,l(t)=x]
\end{multlined}\\
=&\begin{multlined}[t]
\int\limits_{0}^{\infty}\frac{dt}{t}\mathbb{P}^{x_1}_t[p(l)=k, \xi_1=x_1, \ldots, \xi_{k+1}=x_{k+1}, \\
f_1(\tau_1) \cdots f_{k}(\tau_{k})f_{k+1}(t-\sum\limits_{j=1}^{k}\tau_j) ,l(t)=x_1].
\end{multlined}
\end{align*}
By definition of $\mathbb{P}^{x}_t$,
\begin{align*}
(*)=&\begin{multlined}[t]
1_{\{x_1=x_{k+1}\}}\int\limits_{0}^{\infty}\frac{1}{t}\,dtQ^{x_1}_{x_2}\cdots Q^{x_{k-1}}_{x_k}Q^{x_k}_{x_1}
\int\limits_{\mathclap{\begin{subarray}{c}s^1,\ldots,s^{k+1}>0\\s^1+\cdots+s^k<t\\s^1+\cdots+s^{k+1}>t\end{subarray}}}f_1(s^1)\cdots f_k(s^k)f_{k+1}(t-\sum\limits_{j=1}^{k}s_j)\\
(\prod\limits_{i=1}^{k+1}(-L^{x_i}_{x_i})e^{L^{x_i}_{x_i}s^i}\,ds^i)
\end{multlined}\\
=&\begin{multlined}[t]
1_{\{x_1=x_{k+1}\}}\int\limits_{0}^{\infty}\frac{dt}{t}Q^{x_1}_{x_2}\cdots Q^{x_{k-1}}_{x_k}Q^{x_k}_{x_1}\int\limits_{\mathclap{\begin{subarray}{c}s^1,\ldots,s^{k}>0\\s^1+\cdots+s^k<t\end{subarray}}}f_1(s^1)\cdots f_k(s^k)f_{k+1}(t-s^1-\cdots-s^k)\\
e^{L^{x_1}_{x_1}(t-s^1-\cdots-s^k)}(\prod\limits_{i=1}^{k}(-L^{x_i}_{x_i})e^{L^{x_i}_{x_i}s^i}\,ds^i).
\end{multlined}
\end{align*}
Now, change the variables as follows:  $t^1=s^1,\ldots,t^k=s^k,t^{k+1}=t-s^1-\cdots-s^k$.
\begin{align*}
&\mu^b(p(l)=k, \xi_1=x_1, \ldots, \xi_k=x_k, \xi_{k+1}=x_1, f_1(\tau_1) \cdots f_k(\tau_k)f_{k+1}(\tau_{k+1}))\\
&=\begin{multlined}[t]
1_{\{x_1=x_{k+1}\}}\int\limits_{\mathclap{t^1,\ldots,t^{k+1}>0}}\frac{1}{t^1+\cdots+t^{k+1}}Q^{x_1}_{x_2}\cdots Q^{x_{k-1}}_{x_k}Q^{x_k}_{x_1}f_1(t^1)\cdots f_k(t^k)f_{k+1}(t^{k+1})e^{L^{x_{1}}_{x_{1}}t^{k+1}}\\
\prod\limits_{i=1}^{k}(-L^{x_i}_{x_i})e^{L^{x_i}_{x_i}t^i}\prod\limits_{i=1}^{k+1}\,dt^i.
\end{multlined}
\end{align*}
Consequently, for $k\geq 2$,
\begin{align*}
&\mu^b(p(l)=k, \xi_1=x_1, \ldots  \xi_k=x_k, \xi_{k+1}=x_{k+1},\tau_1\in dt^1, \ldots , \tau_{k+1}\in dt^{k+1})\\
&=1_{\{x_1=x_{k+1}\}}Q_{x_2}^{x_1}\cdots Q_{x_1}^{x_k} \frac{1}{t^1+\cdots+t^{k+1}}(-L^{x_1}_{x_1})e^{L_{x_1}^{x_1}t^1}\cdots  (-L^{x_k}_{x_k})e^{L_{x_k}^{x_k}t^k}e^{L_{x_{k+1}}^{x_{k+1}}t^{k+1}}\,dt^1\cdots  dt^{k+1}.
\end{align*}
The case $k=1$ is similar and even simpler.
\end{proof}

\begin{defn}[\textbf{Doob's harmonic transform}]
A non-negative function $h$ is said to be excessive iff $-Lh\geq 0$. Define Doob's harmonic transform $((L^{h})^x_y,\; x, y\in \supp(h))$ of $L$ as follows
$$(L^{h})^{x}_{y}=\frac{L^x_yh(y)}{h(x)}.$$
\end{defn}
As in \cite{sophie}, the following proposition is a direct consequence of Proposition \ref{expression of the based loop measure}.
\begin{prop}\label{invariant under Doob harmonic transform}
The based loop measure is invariant under the harmonic transform with respect to any strictly positive excessive function.
\end{prop}

\begin{rem}
Doob's $h$-transform with respect to a strictly positive function does not change the bridge measure.
\end{rem}

\begin{defn}[\textbf{Pointed loops and discrete pointed loops}]
Using the same notation as before, set $\tau_1^{*}=\tau_1+\tau_{p(l)+1}$, $\tau_i^{*}=\tau_i$ for $1<i<p(l)+1$. Then $(\xi_1,\tau^{*}_1,\ldots,\xi_{p(l)},\tau^{*}_{p(l)})\in\bigcup\limits_{p\in\mathbb{N}_{+}}(S\times\mathbb{R}_{+})^{ p}$ is called the pointed loop obtained from the based loop $(\xi_1, \tau_1, \ldots, \xi_{p(l)+1}=\xi_1, \tau_{p(l)+1})$. Clearly, $\xi_1\neq \xi_{p(l)}$ and $\xi_{i}\neq \xi_{i+1}$ for $i=1,\ldots, p-1$. The induced measure on pointed loops is denoted by $\mu^p$. By removing the holding times from the pointed loop, we get a discrete based loop $\xi=(\xi_1,\ldots, \xi_{p(l)})$.
\end{defn}


%

As a direct consequence of Proposition \ref{expression of the based loop measure}, we obtain the following by change of variables:
\begin{prop}[Expression of $\mu^p$]\label{expression of mup}\
For $k\geq 2$,
\begin{align*}
&\mu^p(p(l)=k, \xi_1=x_1, \ldots,  \xi_k=x_k, \tau_1^{*}\in dt^1, \ldots , \tau_k^{*}\in dt^k)\\
=&Q_{x_2}^{x_1}\cdots Q_{x_1}^{x_k} \frac{t^1}{t^1+\cdots+t^k}(-L^{x_1}_{x_1})e^{L_{x_1}^{x_1}t^1}\cdots (-L^{x_k}_{x_k}) e^{L_{x_k}^{x_k}t^k} \,dt^1\cdots  dt^k.
\end{align*}
For $k=1$,
$$\mu^p(p(l)=1, \xi_1=x_1, \tau_1^{*}\in dt^1)=\frac{1}{t^1}e^{L_{x_1}^{x_1}t^1}\,dt^1.$$
\end{prop}

\begin{defn}[\textbf{Loops and loop measure}]
We define an equivalence relation between based loops. Two based loops are called equivalent iff they have the same time length and their periodical extensions are the same under a translation on $\mathbb{R}$. The equivalence class of a based loop $l$ is called a loop and denoted $l^{o}$. Sometimes, for the simplicity of the notations,  if there is no ambiguity, we will omit the superscript  $o$ and use the same notation $l$ for a based loop and the associated  loop. Moreover, the based loop measure induces a measure on loops, namely the loop measure $\mu$. The loop measure is defined by the generator $L$. Sometimes, we will write $\mu(L,dl)$ instead of $\mu$ to stress this point.
\end{defn}

\begin{defn}
For a pointed loop $l$, let $p(l)$ be the number of jumps made by $l$. For any pointed loop $(\xi_1,\tau_1, \ldots,\xi_n,\tau_n)$, define $N^x_y=\sum\limits_{i=1}^{p(l)}1_{\{\xi_i=x, \xi_{i+1}=y\}}$ and $N^x=\sum\limits_{y\in S}N^x_y=\sum\limits_{i=1}^{p(l)}1_{\{\xi_i=x\}}$. $p(l)$, $N^x_y(l)$ and $N^x(l)$ have the same value for equivalent pointed loops. Accordingly, they can be defined on the space of loops and denoted the same.
\end{defn}



\begin{defn}[\textbf{Discrete loops and discrete loop measure}]
We define an equivalence relation $\sim$ on $\bigcup\limits_{k} S^k$ as follows:  $(x_1,\ldots,x_n)\sim(y_1,\ldots,y_m)$ iff $m=n$ and $\exists j\in\mathbb{Z}$ such that $(x_1,\ldots,x_n)=(y_{1+j},\ldots, y_{m+j})$. For any $(x_1,\ldots, x_n)\in\bigcup\limits_{k} S^k$, use $(x_1,\ldots, x_n)^o$ to stand for the equivalent class of $(x_1,\ldots, x_n)$. Then the space of discrete loops is $\{(x_1,\ldots,x_n)^o;(x_1,\ldots,x_n)\in\bigcup\limits_{k} S^k\}$. For any loop $l^o=(x_1,t^1,\ldots,x_k,t^k)^o$, use $l^{o,d}$ to stand for the discrete loop $(x_1,\ldots,x_k)^o$. The mapping from loops to discrete loops and the loop measure induces a measure on the space of discrete loops, namely the discrete loop measure $\mu^d$.
\end{defn}

\begin{defn}[\textbf{Powers}]
Let $l:[0,|l|]\rightarrow S$ be a based loop. Define the $n$-th power of $l^{n}:[0,n|l|]\rightarrow S$ as follows:
for $k=0,\ldots,n-1$ and $t\in[0,|l|]$, $l^n(t+k|l|)=l(t)$.
The $n$-th powers of equivalent based loops are again equivalent. Consequently, the $n$-th powers of the loop is well-defined. The powers of the discrete loops are defined similarly.
\end{defn}

\begin{defn}[\textbf{Multiplicity and primitive of the non-trivial loops}]\label{defn:multiplicity of the loop}
The multiplicity of a discrete loop is defined as follows:
$$n(l^{o,d})=\max\{k\in\mathbb{N}:\exists \tilde{l}^{o,d},l^{o,d}=(\tilde{l}^{o,d})^k\}$$
If $l^{o,d}=(\tilde{l}^{o,d})^{n(l^{o,d})}$, then $\tilde{l}^{o,d}$ is called a primitive of $l^{o,d}$.
For a non-trivial loop $l$, the multiplicity is defined as follows:
$$n(l^o)=\max\{k\in\mathbb{N}:\exists \tilde{l}^{o}, l^o=(\tilde{l}^{o})^k\}$$
For a trivial loop $l$, the multiplicity is defined to be 1.
If $(\tilde{l}^{o})^{n(l^o)}=l^o$, then $\tilde{l}^{o}$ will be called the primitive of $l^{o}$, as it is always unique. And we will use $prime$ to stand for the mapping from a (discrete) loop to its primitive.
\end{defn}

\begin{defn}[\textbf{Primitive (discrete) loops and (discrete) primitive loop measure}]
A (discrete) loop is called primitive iff its multiplicity is one. The mapping $prime$ induces a measure on (discrete) primitive loops, namely the (discrete) primitive loop measure.
\end{defn}

\begin{prop}
We have the following expression for the discrete loop measure:
$$\mu^d((x_1,\ldots,x_k)^o)=\frac{1}{n((x_1, \ldots, x_k)^o)}Q^{x_1}_{x_2}\cdots Q^{x_k}_{x_1}.$$
\end{prop}

\begin{defn}[\textbf{Pointed loop measure}]\label{defn:pointed loop measure}
We can define another measure $\mu^{p*}$ on the pointed loop space as follows:
\begin{itemize}
\item for $k\geq 2$,
\begin{multline*}
\mu^{p*}(p(l)=k, \xi_1=x_1, \tau_1\in dt^1,\ldots, \xi_k=x_k, \tau_k\in dt^k)\\
=\frac{1}{k}Q^{x_1}_{x_2}\cdots Q^{x_k}_{x_1}(-L^{x_1}_{x_1})e^{L^{x_1}_{x_1}t^1}\,dt^1\cdots(-L^{x_k}_{x_k})e^{L^{x_k}_{x_k}t^k}\,dt^k.
\end{multline*}
\item for $k=1$,
$\mu^{p*}(p(\xi)=1, \xi=x, \tau\in dt)=\frac{1}{t}e^{L_{x}^{x}t}\,dt.$
\end{itemize}
We call $\mu^{p*}$ the pointed loop measure.
\end{defn}
\begin{prop}\label{pointed loop measure induces the loop measure}
$\mu^{p*}$ induces the same loop measure as $\mu^b$ and $\mu^p$.
\end{prop}
\begin{proof}
It is obvious for the trivial loops. Let us focus on the non-trivial loops. For a non-trivial pointed loop $l=(\xi_1,\tau_1,\ldots, \xi_n,\tau_n)$, define $\theta(l)=(\xi_2,\tau_2, \ldots, \xi_n,\tau_n,\xi_1,\tau_1)$. Fix $n\geq 2$, $x_1,\ldots, x_n\in S$, $f: \mathbb{R}_{+}^{n}\rightarrow\mathbb{R}_{+}$ measurable, define $$\Phi(l)=1_{\{p(l)=n\}}1_{\{\xi_1=x_1, \ldots, \xi_n=x_n\}}f(\tau_1,\ldots, \tau_n)$$ and $\bar{\Phi}=\frac{1}{n}(\Phi+\Phi\circ\theta+\cdots+\Phi\circ\theta^{n-1})$. By Proposition \ref{expression of mup},
$$\mu^p(\bar{\Phi})=\frac{1}{n}Q^{x_1}_{x_2}\cdots Q^{x_n}_{x_1}\int\limits_{\mathbb{R}_{+}^n}f(t^1,\ldots,t^n)(\prod\limits_{i=1}^{n}(-L^{x_i}_{x_i})e^{L^{x_i}_{x_i}t^i}\,dt^i).$$
From the definition of the pointed loop measure $\mu^{p*}$, $\theta\circ\mu^{p*}=\mu^{p*}$, $$\mu^{p*}(\bar{\Phi})=\mu^{p*}(\Phi)=\frac{1}{n}Q^{x_1}_{x_2}\cdots Q^{x_n}_{x_1}\int\limits_{\mathbb{R}_{+}^n}f(t^1,\ldots,t^n)(\prod\limits_{i=1}^{n}(-L^{x_i}_{x_i})e^{L^{x_i}_{x_i}t^i}\,dt^i).$$
We have $\mu^p(\bar{\Phi})=\mu^{p*}(\bar{\Phi})$. For a positive functional $\Phi$ on the space of pointed loops, we have the following decomposition
$$\Phi=\sum\limits_{n\geq 1}\sum\limits_{x\in S^n}1_{\{p(l)=n\}}1_{\{\xi_1=x_1,\ldots,\xi_n=x_n\}}f^{x}(\tau_1,\ldots,\tau_n)$$
where $f^x(\tau_1,\ldots,\tau_n)=(\Phi|_{\{l:p(l)=n\}})(x_1,\tau_1,\ldots,x_n,\tau_n)$. Define
$$\bar{\Phi}=\sum\limits_{n\geq 1}\sum\limits_{x\in S^n}\overline{1_{\{p(l)=n\}}1_{\{\xi_1=x_1,\ldots,\xi_n=x_n\}}f^{x}(\tau_1,\ldots,\tau_n)}.$$
It is clear that $\bar{}:\Phi\rightarrow\bar{\Phi}$ is a well-defined linear map which preserves the positivity. By monotone convergence, $\mu^{p*}(\bar{\Phi})=\mu^{p}(\bar{\Phi})$ for any positive measurable pointed loop functional. As a consequence, the loop measure induced by $\mu^{p*}$ is exactly $\mu$.
\end{proof}


\begin{defn}\label{defn:pointed loop measure for pointed loops visiting F}
For a pointed loop $l=(\xi_1,\tau_1,\ldots,\xi_{p(l)},\tau_{p(l)})$, $\xi=(\xi_1,\ldots,\xi_{p(l)})$ is the corresponding discrete pointed loop. For any $F\subset S$, define $q(F,l)=\sum\limits_{x\in F}N^x(l)$ the number of times $l$ visits $F$. Recursively define the $i$-th hitting time for $F$ as follows ($i=1,\ldots,q(F,l)$):  $T^{F}_1(l)=T^F_1(\xi)=\inf\{m\leq p(l):\xi_m\in F\}$ and  $T^{F}_{i+1}(l)=T^F_{i+1}(\xi)=\inf\{m>T^{F}_i:m\leq p(l),\xi_m\in F\}$. Define $T=T^{F}_{q(F,l)}$ the last visiting time for $F$. Define $p(F,l)=\#\{i:\xi_{T^F_i}\neq\xi_{T^F_{i+1}},i=1,\ldots,q(F,l)\}$ with the convention that $\xi_{T^F_{q(F,l)+1}}=\xi_{T^F_{1}}$.\\
Define a pointed loop measure $\mu^{p*,F}$ as follows:
\begin{align*}
\mu^{p*,F}1_{\{p(F,l)\neq 0\}} &= 1_{\{\xi_1\in F,\xi_T\neq\xi_1\}}\frac{p(l)}{p(F,l)}\mu^{p*}\\
\mu^{p*,F}1_{\{p(F,l)=0\}} &= 1_{\{\xi_1\in F,p(F,l)=0\}}\frac{p(l)}{q(F,l)}\mu^{p*}.
\end{align*}
\end{defn}
\begin{rem}
$p(F,l)=0$ iff the intersection of the pointed loop $l$ and the subset $F\subset S$ is a single element set: $|l\cap F|=1$ (or $|\bigcup\limits_{i=1}^{q(F,l)}\{\xi_{T^F_i}\}|=1$ equivalently). For a loop $l$ with $p(F,l)\neq 0$ (or $p(F,l)=2,\ldots,\infty$ equivalently), the term $1_{\{\xi_1\in F,\xi_T\neq\xi_1\}}$ in the above expression implies that $\mu^{p*,F}|_{\{l:p(F,l)\neq 0\}}$ is concentrated on the pointed loops satisfying the following two conditions:
\begin{enumerate}
\item the pointed loop starts from a point in $F$.
\item the trace of the pointed loop on $F$ has an endpoint different from the starting point.
\end{enumerate}
By an argument similar to remark \ref{pointed loop measure induces the loop measure}, it can be showed that $\mu^{p*,F}$ induces a loop measure which is exactly the restriction of $\mu$ to the loops visiting $F$.
\end{rem}

\begin{defn}[\textbf{Multi-occupation field}]
Define the circular permutation operator $r_j$ as follows: $r_j(z^1,\ldots,z^p)=(z^{1+j},\ldots,z^n,z^1,\ldots,z^j)$.
For any $f:S^n\rightarrow\mathbb{R}$ measurable, define the multi-occupation field of a based loop $l$ of length $t$ as $$\langle l,f\rangle=\sum\limits_{j=0}^{n-1}\int\limits_{0<s^1<\cdots<s^n<t}f\circ r_j(l(s_1),\ldots,l(s_n))\,ds^1\cdots\,ds^n.$$
If $l_1$ and $l_2$ are two equivalent based loops, they correspond to the same multi-occupation field. Therefore, it is well-defined for loops. When $n=1$, it is called the occupation time. For $x\in\mathbb{R}^m$ for some integer $m$, define $l^{x}=\langle l,\delta_x\rangle$ where $\delta_x(y)=1_{\{x=y\}}$.
\end{defn}

\begin{defn}[\textbf{Another bridge measure }$\mu^{x,y}$]
Another bridge measure $\mu^{x,y}$ can be defined on paths from $x$ to $y$:
$$\mu^{x,y}(d\gamma)=\int\limits_{0}^{\infty}\mathbb{P}^{x,y}_{t}(d\gamma)dt.$$
For a path $\gamma$ from $x$ to $y$, let $p(\gamma)$ be the total number of jumps, $T_i$ the $i$-th jumping time and $T$ the time duration of $\gamma$. Then $\gamma$ can be viewed as $(x,T_1,\gamma(T_1),T_2-T_1,\gamma(T_2),\ldots,T_{p(\gamma)}-T_{p(\gamma)-1},y=\gamma(T_{p(\gamma)}),T-T_{p(\gamma)})$.
\end{defn}
The bridge measure $\mu^{x,y}$ can be expressed as follows:
\begin{prop}
\begin{align*}
&\mu^{x,y}(p(\gamma)=p,\gamma(T_1)=x_1,\ldots,\gamma(T_{p-1})=x_{p-1},\\
&T_1\in dt^1,T_2-T_1\in dt^2,\ldots,T_{p}-T_{p-1}\in dt^p,T-T^p\in dt^{p+1})\\
=&Q^{x}_{x_1}Q^{x_1}_{x_2}\cdots Q^{x_{p-1}}_{y}1_{\{t^1,\ldots,\;t^{p+1}>0\}}(-L^x_x)e^{L^x_x t^1}(-L^{x_1}_{x_1})e^{L^{x_1}_{x_1}t^2}\cdots (-L^{x_{p-1}}_{x_{p-1}})e^{L^{x_{p-1}}_{x_{p-1}}t^p}e^{L^y_yt^{p+1}}\prod\limits_{j=1}^{p+1}\,dt^j
\end{align*}
\end{prop}
In the case $x=y$, $\gamma=(x,T_1,\gamma(T_1),T_2-T_1,\gamma(T_2),\ldots,T_{p(\gamma)}-T_{p(\gamma)-1},y=T_{p(\gamma)},T-T_{p(\gamma)})$ can be viewed as a based loop. Therefore, $\mu^{x,x}$ can be viewed as a measure on the based loop. Moreover, $\mu^{x,x}(dl)=1_{\{l(0)=x\}}|l|\mu^b(dl)$. Consequently, the loop measure induced by $\mu^{x,x}$, which will be denoted by the same notation $\mu^{x,x}$, has the following relation with the loop measure $\mu$.
\begin{prop}\label{relation between the bridge measure and loop measure 1}
$$\mu^{x,x}(dl)=l^x\mu(dl).$$
\end{prop}
In the case $x\neq y$, $\gamma=(x,T_1,\gamma(T_1),T_2-T_1,\gamma(T_2),\ldots,T_{p(\gamma)}-T_{p(\gamma)-1},y=\gamma(T_{p(\gamma)}),T-T_{p(\gamma)})$ can be viewed as a pointed loop. Similarly, $\mu^{x,y}$ can be viewed as a measure on the pointed loop. Moreover, $L^y_x\mu^{x,y}(dl)=1_{\{l \text{ starts at }x \text{ and ends up at } y\}}p(l)\mu^{p*}(dl)$. Consequently, the loop measure induced by $\mu^{x,y}$, which will be denoted by the same notation $\mu^{x,y}$, has the following relation with the loop measure $\mu$.
\begin{prop}
$$L^y_x\mu^{x,y}(dl)=N^y_x\mu(dl).$$
\end{prop}

\subsection{Compatibility of the loop measure with time change}
\begin{prop}\label{compatibility with the time change} Suppose $\lambda:S\rightarrow[0,\infty]$. Given a Markov process $(X_t, t\geq 0)$ in $S$, define $B_t=\int\limits_{0}^{t}\lambda(X_s)\,ds$. Let $(C_t, t\geq 0)$ be the right-continuous inverse of $(B_t, t\geq 0)$. Define $\zeta=\inf\{s\geq 0:C_s=C_{\infty}\}$. Define $Y_t=X_{C_t},t<\zeta$ (it will be called the time-changed process of $X$ with respect to $\lambda$ and denoted $\lambda(X)$).  On the space of based (pointed) loops contained in $\{x\in S:\lambda(x)<\infty\}$, $\lambda$ defines a similar operation. If $l_1$ and $l_2$ are two equivalent based (pointed) loops, $\lambda(l_1)$ and $\lambda(l_2)$ are equivalent again. Consequently, $\lambda$ can be defined on the space of loops with the domain $D(\lambda)=\{\text{loops contained in }\{x\in S:\lambda(x)<\infty\}\}$. There are two Markovian loop measures $\mu_X$, $\mu_Y$ defined by $X$, $Y$ respectively. The following diagram commutes:
$$\begin{array}{c c c}
X & \xrightarrow{\lambda} & Y\\
\downarrow & & \downarrow\\
\mu_{X} & \xrightarrow{\lambda} & \mu_{Y}
\end{array}$$
In particular, the loop measure is compatible with the notion of ``trace on a set" (i.e. $\lambda=1_A$) and ``restriction" (i.e. $\lambda=1_A+\infty\cdot1_{A^c}$).
\end{prop}

\begin{proof}
Let $\lambda\circ\mu$ be the image law of $\mu$ under the mapping $\lambda$. Denote by $\pi^{p\rightarrow o}$ the quotient map from pointed loops to loops. Then, we have to show that $\lambda$ commutes with $\pi^{p\rightarrow o}$.\\
The holding times are almost surely different for $\mu_X,\mu_Y\text{ and }\lambda\circ\mu_X$. So the same is true for the measures on pointed loops $\mu^{p*}_X,\mu^{p*}_{Y}$ and $\lambda\circ\mu^{p*}_X$.\\
Every  change of time can be done in three steps: i) Restriction, ii) trace, iii)  time change with a function $0<\lambda<\infty$. Accordingly, it is enough to deal with these three special cases separately:
\begin{itemize}
\item[i)] $0<\lambda<\infty$\\
Let $L$ and $\hat{L}$ represent the generator of $X$ and $Y$. Then $\hat{L}^x_y=\frac{L^x_y}{\lambda_x}$.\\
By Definition \ref{defn:pointed loop measure} and its following remark,
\begin{align*}
\lambda\circ\mu^{p*}_X( p(\xi)=k,&\;  \xi_1=x_1, \cdots  \xi_k=x_k, \tau_1\in dt^1, \cdots , \tau_k\in dt_k)\\
= &\mu^{p*}_X(p(\xi)=k, \xi_1=x_1, \cdots  \xi_k=x_k, \lambda_{x_1}\tau_1\in dt^1, \cdots , \lambda_{x_k}\tau_k\in dt_k)\\
= &\frac{1}{k}L_{x_2}^{x_1}\cdots L_{x_1}^{x_k} e^{L_{x_1}^{x_1}t^1/\lambda_{x_1}}\cdots  e^{L_{x_k}^{x_k}t^k/\lambda_{x_k}} \,dt^1\cdots  dt^k\\
= &\frac{1}{k}\hat{L}_{x_2}^{x_1}\cdots \hat{L}_{x_1}^{x_k} e^{\hat{L}_{x_1}^{x_1}t^1}\cdots  e^{\hat{L}_{x_k}^{x_k}t^k} \,dt^1\cdots  dt^k\\
= &\mu^{p*}_Y(p(\xi)=k, \xi_1=x_1, \cdots  \xi_k=x_k, \tau_1\in dt^1, \cdots , \tau_k\in dt_k)
\end{align*}
Therefore, $\lambda\circ\mu_X=\lambda\circ\pi^{p\rightarrow o}\circ\mu^{p*}_X=\pi^{p\rightarrow o}\circ\lambda\circ\mu^{p*}_X=\pi^{p\rightarrow o}\mu^{p*}_Y=\mu_Y$

\item[ii)] $\lambda=1_A+\infty\cdot 1_{A^c}$.
In that case, $\lambda\circ\mu_X=\mu_X|_{D(\lambda)}=\mu_Y$.
\item[iii)] $\lambda=1_A+0\cdot 1_{A^c}$.\\
We needs to show that $\lambda\circ\mu_{X}=\mu_Y$. We will only prove this for the non-trivial loops. The trivial loop case can be proved in a similar way.\\
Use $\mathbb{P}^x$ to stand for the law of a minimal Markov process $X$ starting from $x$. Let $T_1$ be the first jumping time, and set $T_{1,A}=\inf\{s\geq T_1, X_s\in A\}$. Let $(R^A)^x_y=\mathbb{P}^x[X_{T_{1,A}}=y]$ for $x,y\in S$. Obviously, $(R^A)^x_y=0$ for $y\in A^c$. By Proposition \ref{Expression of the generator for the time changed process}, the relation between the generator $L$ of $X$ and the generator $\hat{L}$ of $Y$ is stated as follows: $\hat{L}^{x}_{x}=L^x_x(1-(R^A)_x^x), \hat{L}^x_y=-(R^A)^x_y L^x_x$ for $x\neq y$.

Fix a non-trivial discrete pointed loop $(x_1,\ldots,x_n)$ where $x_i\in A$ for $i=1,\ldots,n$. Take $F=\bigcup\limits_{i=1}^{n}\{x_i\}$. Take $n$ positive measurable functions $f_1,\ldots,f_n$ on $S$. By Definition \ref{defn:pointed loop measure for pointed loops visiting F} and its following remark, it is enough to show that
\begin{multline*}
\lambda\circ\mu_X^{p*,F}(p(\xi)=n,\xi_1=x_1,\ldots,\xi_n=x_n,\prod\limits_{i=1}^{n}f_i(\tau_i))\\
=\mu_Y^{p*}(p(\xi)=n,\xi_1=x_1,\ldots,\xi_n=x_n,\prod\limits_{i=1}^{n}f_i(\tau_i)).
\end{multline*}
In order that $\lambda(l)$, the image of the pointed loop $l$, equals $(p(\xi)=n,x_1,\tau_1,\ldots,x_n,\tau_n)$, the pointed loop $l$ has to be of the following form $\mu_X^{p,*F}$-a.s.:
\begin{align*}
(&\xi_{111},\tau_{111},\ldots,\xi_{11M_1^1},\tau_{11M_1^1},\ldots,\xi_{1N_11},\tau_{1N_11},\ldots,\xi_{1N_1M_{N_1}^1},\tau_{1N_1M_{N_1}^1},\\
&\xi_{211},\tau_{211},\ldots,\xi_{21M_1^2},\tau_{21M_1^2},\ldots,\xi_{2N_21},\tau_{2N_21},\ldots,\xi_{2N_2M_{N_2}^2},\tau_{2N_2M_{N_2}^2},\\
&\cdots\\
&\xi_{n11},\tau_{n11},\ldots,\xi_{n1M_1^n},\tau_{n1M_1^n},\ldots,\xi_{nN_n1},\tau_{nN_n2},\ldots,\xi_{nN_nM_{N_n}^n},\tau_{nN_nM_{N_n}^n})
\end{align*}
with
\begin{itemize}
\item $\xi_{ij1}=x_i$ for all $i,j$;
\item $\xi_{ijk}\in A^c$ for $k\neq 1$ and all $i,j$;
\item $\tau_i=\sum\limits_{j}\tau_{ij1}$.
\end{itemize}
Roughly speaking, $\xi_{ij1},\tau_{ij1},\ldots,\xi_{ijM^i_j},\tau_{ijM^i_j}$ can be viewed as an excursion in $A^c$ from $x_i$ to $x_i$ for $j\neq N_i$. And $\xi_{iN_i1},\tau_{iN_i1},\ldots,\xi_{iN_iM^i_{N_i}},\tau_{iN_iM^i_{N_i}}$ can be viewed as an excursion in $A^c$ from $x_i$ to $x_{i+1}$. Accordingly,
\begin{multline*}
\lambda\circ\mu_X^{p*,F}(p(\xi)=n,\xi_1=x_1,\ldots,\xi_n=x_n,\prod\limits_{i=1}^{n}f_i(\tau_i))\\
=\sum\limits_{\xi}\mu_X^{p*,F}(\xi_{ij1}=x_i\text{ and for all }i,j,\xi_{ijn}\in A^c\text{ for }n\neq 1,\prod\limits_{i=1}^{n}f_i(\sum\limits_{j}\tau_{ij1})).
\end{multline*}
Since $Q^x_y+\sum\limits_{p=1}^{\infty}\sum\limits_{z_1,\ldots,z_p\in A^c}Q^{x}_{z_1}Q^{z_1}_{z_2}\cdots Q^{z_{p-1}}_{z_p}Q^{z_p}_{y}=(R^A)^x_y$, the above quantity equals
\begin{align*}
&\frac{1}{n}\sum\limits_{N_1,\ldots,N_n\geq 1}\prod\limits_{i=1}^{n}((R^A)^{x_i}_{x_i})^{N_i-1}(R^A)^{x_i}_{x_{i+1}}\\
&\int f(t^{i11}+\cdots+t^{iN_i1})(-L^{x_i}_{x_i})^{N_i}e^{L^{x_i}_{x_i}(t^{i11}+\cdots+t^{iN_i1})}\,dt^{i11}\cdots\,dt^{iN_i1}\\
=&\frac{1}{n}\sum\limits_{N_1,\ldots,N_k\geq 1}\prod\limits_{i=1}^{n}\int((R^A)^{x_i}_{x_i})^{N_i-1}(R^A)^{x_i}_{x_{i+1}}\frac{(t^i)^{N_i-1}}{(N_i-1)!}(-L^{x_i}_{x_i})^{N_i}e^{L^{x_i}_{x_i}t^i}f(t^i)\,dt^i\\
=&\frac{1}{n}\prod\limits_{i=1}^{n}\int -L^{x_i}_{x_i}(R^A)^{x_i}_{x_{i+1}}e^{L^{x_i}_{x_i}t^i(1-(R^A)^{x_i}_{x_i})}f(t^i)\,dt^i\\
=&\frac{1}{n}\prod\limits_{i=1}^{n}\int (L_A)^{x_i}_{x_{i+1}}e^{(L_A)^{x_i}_{x_i}t^i}f(t^i)\,dt^i\\
=&\mu_Y^{p*}(p(\xi)=n,\xi_1=x_1,\ldots,\xi_n=x_n,\prod\limits_{i=1}^{n}f_i(\tau_i))\text{ for }n\geq 2.
\end{align*}
For $n=1$, it can be proved in a similar way. Finally, we conclude that $\lambda\circ\mu_X=\mu_Y$.
\end{itemize}
\end{proof}
\subsection{Decomposition of the loops and excursion theory}
Fix some set $F\subset S$.
\begin{defn}[\textbf{excursion outside }$F$]
By non-empty excursion outside $F$, we mean a multiplet of the form $((\xi_1,\tau^1,\ldots,\xi_k,\tau^k),A,B)$ for some $k\in\mathbb{N}_{+},\xi_1,\ldots,\xi_k\in F^c,A,B\in F$ and $\tau^1,\ldots,\tau^k\in\mathbb{R}_{+}$. Let $T_0=0$ and $T_m=\tau^1+\cdots+\tau^m$ for $m=1,\ldots,k$. Define $e:[0,T_k[\rightarrow F^c$ such that $e(u)=\xi_m$ for $u\in[T_{m-1},T_m[$. Therefore, the excursion can be viewed as a path $e$ attached to starting point $A$ and ending point $B$ and it will also be denoted by $(e,A,B)$. By empty excursion, we mean $(\phi,A,B)$.
\end{defn}
\begin{defn}[\textbf{excursion measure outside }$F$]\label{defn:excursion measure outside of F}
Define a family of probability measure $\nu_{F,ex}^{x,y}$ indexed by $x,y\in F$ as follows:
\begin{multline*}
\nu_{F,ex}^{x,y}(\xi_1=x_1,\tau^1\in dt^1,\ldots,\xi_k=x_k,\tau^k\in dt^k,A=u,B=v)\\
=\delta^{(x,y)}_{(u,v)}\frac{1}{(R^F)^x_y}1_{\{x_1,\ldots,x_k\in F^c\}}Q^x_{x_1}L^{x_1}_{x_2}\cdots L^{x_{k-1}}_{x_k}L^{x_k}_{y}e^{L^{x_1}_{x_1}t^1}\cdots e^{L^{x_k}_{x_k}t^k}\,dt^1\cdots\,dt^k.
\end{multline*}
and $\nu_{F,ex}^{x,y}[(\phi,A,B)=(\phi,u,v)]=\delta^{(x,y)}_{(u,v)}\frac{Q^x_y}{(R^F)^x_y}$. Recall that
$$(R^{F})^x_y=\left\{\begin{array}{ll}
Q^x_y+\sum\limits_{k\geq 1}\sum\limits_{x_1,\ldots,x_k\in F^c}Q^x_{x_1}Q^{x_1}_{x_2}\cdots Q^{x_{k-1}}_{x_k}Q^{x_k}_{y} & \text{ for }y\in F\\
0 & \text{otherwise}.\end{array}\right.$$
\end{defn}

Define a function $\phi^{br\rightarrow ex}$ from the space of bridges to the space of excursions as follows:
Given a bridge $\gamma$ from $x$ to $y$, which is represented by $$(x,T_1,\gamma(T_1),T_2-T_1,\ldots,\gamma(T_{p(\gamma)-1}),T_{p(\gamma)}-T_{p(\gamma)-1},y=\gamma(T_{p(\gamma)}),T-T_{p(\gamma)}),$$
we represent $\phi^{br\rightarrow ex}(\gamma)$ by $$((\gamma(T_1),T_2-T_1,\ldots,\gamma(T_{p(\gamma)-1}),T_{p(\gamma)}-T_{p(\gamma)-1}),x,y).$$
The image measure of $\mu^{x,y}$ under $\phi^{br\rightarrow ex}$, namely $\phi^{br\rightarrow ex}\circ \mu^{x,y}$, has the following relation with the excursion measure $\nu_{F,ex}^{x,y}$:
\begin{prop}
$$\nu_{F,ex}^{x,y}(d\gamma)=\frac{1}{-L^y_yR^x_y}\phi^{br\rightarrow ex}\circ \mu^{x,y}(d\gamma,\gamma(T_1),\ldots,\gamma(T_{p(\gamma)-1})\in F^c)$$
\end{prop}

Define a function $\varphi^{ex\rightarrow po}_{F}$ from the space of non-empty excursions out of $F$ to the space of pointed loops as follows:
$$\varphi^{ex\rightarrow po}_{F}:((\xi_1,\tau^1,\ldots,\xi_k,\tau^k),A,B)\rightarrow (\xi_1,\tau^1,\ldots,\xi_k,\tau^k)$$
Accordingly, $\nu^{x,y}_{F,ex}$ induces a pointed loop measure on the space of pointed loops outside of $F$, which is denoted by the same notation $\nu^{x,y}_{F,ex}$. The relation with the pointed loop measure is as follows:
\begin{prop}
Let $C=\{(\xi_1,\tau^1,\ldots,\xi_n,\tau^n)\in\{\text{pointed loops}\}:Q^{\xi_n}_{\xi_1}>0\}$. Then, $\varphi^{ex\rightarrow po}_{F}\circ\nu^{x,y}_{F,ex}$ is absolutely continuous with respect to $\mu^{p*}$. Moreover,
$$Q^{\xi_n}_{\xi_1}\frac{d\varphi^{ex\rightarrow po}_{F}\circ\nu^{x,y}_{F,ex}}{d\mu^{p*}}((\xi_1,\tau^1,\ldots,\xi_n,\tau^n))=1_{\{R^x_y>0,\xi_1,\ldots,\xi_n\in F^c\}}\frac{nQ^x_{\xi_1}Q^{\xi_n}_y}{R^x_y}$$
\end{prop}

\begin{defn}[\textbf{Decomposition of a loop}]\label{defn:decomposition of a loop into a pre-trace and some excursions indexed by the edges of the pre-trace}
Let $l=(\xi_1,\tau^1,\ldots,\xi_k,\tau^k)^o$ be a loop visiting $F$. The pre-trace of the loop $l$ on $F$ is obtained by removing all the $\xi_m,\tau^m$ such that $\xi_m\in F^c$ for $m=1,\ldots,k$ . We denote it by $Ptr_{F}(l)$. Suppose the pre-trace on $F$ can be written as $(\mathfrak{x}_1,\mathfrak{s}^1,\ldots,\mathfrak{x}_q,\mathfrak{s}^q)^{o}$. Then we can write the loop $l^{o}$  in the following form: $$(\mathfrak{x}_1,\mathfrak{s}^1,y_{1}^{1},t^{1}_{1},\ldots,y^{1}_{m_1},t^{1}_{m_1},\mathfrak{x}_2,\mathfrak{s}^2,y^{2}_{1},t^{2}_{1},\ldots,y^{2}_{m_2},t^{2}_{m_2},\ldots,\mathfrak{x}_q,\mathfrak{s}^q,y^{q}_{1},t^{q}_{1},\ldots,y^{q}_{m_q},t^{q}_{m_q})^{o}$$
with $\mathfrak{x}_i\in F$ for all $i$ and $y^i_j\in F^c$ for all $i,j$ (with the following convention: if $m_i=0$ for some $i=1,\ldots,q$, $y_{1}^{i},t^{i}_{1},\ldots,y^{i}_{m_i},t^{i}_{m_i}$ does not appear in the above expression). We will use $e_i$ to stand for $(y_{1}^{i},t^{i}_{1},\ldots,y^{i}_{m_i},t^{i}_{m_i})$ with the convention that $e_i=\phi$ if $m_i=0$. Define a point measure $\mathcal{E}_F(l)=\sum\limits_{i}\delta_{(e_i,\mathfrak{x}_i,\mathfrak{x}_{i+1})}$. Define $N^x_y(Ptr_F(l))=\sum\limits_{i=1}^{q}1_{\{\mathfrak{x}_{i}=x,\mathfrak{x}_{i+1}=y\}}$ with the convention that $\mathfrak{x}_{q+1}=\mathfrak{x}_1$. Set $q(Ptr_F(l))=\sum\limits_{x,y}N^x_y(Ptr_F(l))$. In particular, in the case above, we have $q(Ptr_F(l))=q$ if $q\geq 2$ and $q(Ptr_F(l))=0$ if $q=1$.
\end{defn}
\begin{rem}
The pre-trace $(\mathfrak{x}_1,\mathfrak{s}^1,\ldots,\mathfrak{x}_q,\mathfrak{s}^q)^{o}$ of a loop l on $F$ is not necessarily a loop. We allow $\mathfrak{x}_i=\mathfrak{x}_{i+1}$ for some $i=1,\ldots, q$ which is prohibited in the definition we gave of a loop.
\end{rem}

\begin{defn}\label{defn:trace of a loop on F}
The pre-trace of a loop $l$ on $F$ can always be written as follows:
$$(\mathfrak{x_1},\mathfrak{s}^{1}_{1},\ldots,\mathfrak{x_1},\mathfrak{s}^{1}_{m_1},\mathfrak{x}_2,\mathfrak{s}^{2}_{1},\ldots,\mathfrak{x}_2,\mathfrak{s}^{2}_{m_2},\ldots,\mathfrak{x}_k,\mathfrak{s}^{k}_{1},\ldots,\mathfrak{x}_k,\mathfrak{s}^{k}_{m_k})^{o}$$
with $\mathfrak{x}_i\neq \mathfrak{x}_{i+1}$ for $i=1,\ldots,k$ with the usual convention that $\mathfrak{x}_{k+1}=\mathfrak{x}_{1}$.
Then, $l_{F}$, the trace of $l$ on $F$ is defined by
$$(\mathfrak{x_1},\mathfrak{t}^1=\mathfrak{s}^{1}_{1}+\cdots+\mathfrak{s}^{1}_{m_1},\mathfrak{x}_2,\mathfrak{s}^{2}_{1}+\cdots+\mathfrak{s}^{2}_{m_2},\ldots,\mathfrak{x}_k,\mathfrak{t}^k=\mathfrak{s}^{k}_{1}+\cdots+\mathfrak{s}^{k}_{m_k})^{o}.$$
Formally, the trace of $l$ on $F$ is obtained by throwing away the parts out of $F$ and then by gluing the rest in circular order.
\end{defn}
By replacing $\mu$ by $\mu^{p*,F}$ and considering the pointed loops, we have the following propositions.

\begin{prop}\label{joint measure of the pre-trace and the excursions}
Let $f$ be some measurable positive function on the space of excursions and $g$ a positive measurable function on the space of  pre-traces on F. Then,
$$\mu(1_{\{l \text{ visits }F\}}g(Ptr_{F}(l))e^{-\langle\mathcal{E}_F(l),f\rangle})=\mu(1_{\{l \text{ visits }F\}}g(Ptr_F(l))\prod\limits_{x,y\in F}(\nu_{F,ex}^{x,y}(e^{-f}))^{N^x_y(Ptr_F(l))}).$$
\end{prop}

\begin{prop}\label{expression of the pre-trace measure on F}
The image measure $\mu^{p*}_{Ptr,F}$ of the pointed loop measure $\mu^{p*,F}$ under the map of the pre-trace on $F$ can be described as follows:
\begin{itemize}
\item if $x_1,\ldots,x_q$ are not identical, then
\begin{multline*}
\mu_{Ptr,F}^{p*}(q(Ptr_F(l))=q,\mathfrak{x}_1=x_1,\mathfrak{s}^1\in ds^1,\ldots,\mathfrak{x}_q=x_q,\mathfrak{s}^q\in ds^q)\\
=\frac{1}{p(l_F)}\prod\limits_{x,y}((R^F)^x_y)^{N^x_y(Ptr_F(l))}\prod\limits_{i=1}^{q}(-L^{x_i}_{x_i})e^{L^{x_i}_{x_i}s^i}\,ds^{i};
\end{multline*}
\item if $x_1=\ldots=x_q=x$ and $q>1$, then
\begin{multline*}
\mu_{Ptr,F}^{p*}(q(Ptr_F(l))=q,\mathfrak{x}_1=x_1=\cdots=\mathfrak{x}_q=x_q=x,\mathfrak{s}^1\in ds^1,\ldots,\mathfrak{s}^q\in ds^q)\\
=\frac{1}{q}((R^F)^x_x)^q\prod\limits_{i=1}^{q}(-L^{x}_{x})e^{L^{x}_{x}s^i}\,ds^{i}
\end{multline*}
\item if $q=1$ and $x_1=x$, then
\begin{multline*}
\mu_{Ptr,F}^{p*}(q(Ptr_F(l))=1,\mathfrak{x}=x,\mathfrak{s}\in ds)=\underbrace{(R^F)^x_x(-L^{x}_{x})e^{L^{x}_{x}s}\,ds}\limits_{\text{contribution of the non-trivial loops}}\\
+\underbrace{\frac{1}{s}e^{L^x_x s}\,ds}\limits_{\text{contribution of the trivial loops}}
\end{multline*}
\end{itemize}
where $\mu_{Ptr,F}^{p*}$ is the image measure of the pointed loop measure $\mu^{p*,F}|_{\{\text{loops visiting F}\}}$.
\end{prop}

\begin{prop}\label{conditional excursion number given the trace}
Under the same notation as Definition \ref{defn:trace of a loop on F},
\begin{itemize}
\item for $k>1$,
\begin{multline*}
\mu_{Ptr,F}^{p*}(\mathfrak{x_1}=x_1,\ldots,\mathfrak{x}_k=x_k,m_1=q_1,\ldots,m_k=q_k,\mathfrak{t}^{1}\in dt^1,\ldots,\mathfrak{t}^{k}\in dt^k)\\
=\frac{1}{k}(L_F)^{x_1}_{x_2}\cdots(L_F)^{x_k}_{x_1}\prod\limits_{i=1}^{k}e^{(L_F)^{x_i}_{x_i}t^i}\,dt^i\prod\limits_{j=1}^{k}e^{(L^{x_j}_{x_j}-(L_F)^{x_j}_{x_j})t^j}\frac{((-L^{x_j}_{x_j}+(L_F)^{x_j}_{x_j})t^j)^{q_j-1}}{(q_j-1)!}
\end{multline*}
\item for $k=1$, $q_1=q>1$,
\begin{multline*}
\mu_{Ptr,F}^{p*}(\mathfrak{x_1}=x_1,m_1=q_1,\mathfrak{t}^{1}\in dt^1)\\
=\frac{1}{t_1}e^{(L_F)^{x_1}_{x_1}t^1}\,dt^1e^{(L^{x_1}_{x_1}-(L_F)^{x_1}_{x_1})t^1}\frac{((-L^{x_1}_{x_1}+(L_F)^{x_1}_{x_1})t^1)^{q_1}}{q_1!}
\end{multline*}
\item for $k=1$ and $q_1=1$,
\begin{multline*}
\mu_{Ptr,F}^{p*}(\mathfrak{x_1}=x_1,m_1=1,\mathfrak{t}^{1}\in dt^1)\\
=\frac{1}{t_1}e^{(L_F)^{x_1}_{x_1}t^1}\,dt^1e^{(L^{x_1}_{x_1}-(L_F)^{x_1}_{x_1})t^1}((-L^{x_1}_{x_1}+(L_F)^{x_1}_{x_1})t^1)+\frac{1}{t^1}e^{L^{x_1}_{x_1}t^1}\,dt^1
\end{multline*}
\end{itemize}
\end{prop}
\begin{proof}
The result comes from Proposition \ref{expression of the pre-trace measure on F} and Proposition \ref{Expression of the generator for the time changed process}.
\end{proof}
Combining Proposition \ref{joint measure of the pre-trace and the excursions} and Proposition \ref{conditional excursion number given the trace}, we have the following proposition:
\begin{prop}\label{joint measure of the trace and the excursions}
\begin{multline*}
\mu(1_{\{l\text{ visits }F\}}g(l_{F})e^{-\langle\mathcal{E}_F(l),f\rangle})\\
=\mu(1_{\{l\text{ visits }F\}}g(l_F)\prod\limits_{x\neq y\in F}\nu_{F,ex}^{x,y}(e^{-f})^{N^x_y(l_F)}e^{\sum\limits_{x\in F}(L^x_x-(L_F)^x_x)l_F^x\nu_{F,ex}^{x,x}(1-e^{-f})})
\end{multline*}
\end{prop}
\begin{cor}\label{transition kernel from the trace to the point measure of excursions}
We see that $\nu^{x,y}_{F,ex}$ is a probability measure on the space of the excursions from $x$ to $y$ out of $F$. By mapping an excursion $(e,x,y)$ into the Dirac measure $\delta_{(e,x,y)}$, $\nu^{x,y}_{F,ex}$ induces a probability measure on $\mathcal{M}^{p}(\{\text{excursions}\})$, the space of point measures over the space of excursions. We will adopt the same notation $\nu^{x,y}_{F,ex}$. Choose $k$ samples of the excursions according to $\nu^{x,y}_{F,ex}$, namely $ex_1,\ldots,ex_k$, then $\sum\limits_{i}\delta_{ex_i}$ has the law $(\nu^{x,y}_{F,ex})^{\otimes k}$. For any $\beta=(\beta^x,x\in F)\in\mathbb{R}_{+}^{F}$, let $\mathcal{N}_{F}(\beta)$ be a Poisson random measure on the space of excursions with intensity $\sum\limits_{x}(-L^x_x+(L_{F})^x_x)\beta^x\nu^{x,x}_{F,ex}$. Let $l_F\rightarrow K(l_F,\cdot)$ be a transition kernel from \{the trace of the loop on F\} to \{point measure over the space of excursions\} as follows:
$$K(l_F,\cdot)=\bigotimes\limits_{x\neq y\in F}(\nu^{x,y}_{F,ex})^{\otimes N^x_y(l_F)}\bigotimes \mathcal{N}_F((l_F^{x},x\in F))$$
Then the joint measure of $(l_F,\mathcal{E}_F(l))$ is $\mu_F(dl_F) K(l_F,\cdot)$ where $\mu_F$ is the image measure of $\mu$ under $l\rightarrow l_F$. By Proposition \ref{compatibility with the time change}, $\mu_F$ is actually the loop measure associated with the trace of the Markov process on $F$ or with $L_F$ equivalently.
\end{cor}
\begin{rem}
$K(l_F,\cdot)$ can also be viewed as a Poisson random measure on the space of excursions with intensity $\sum\limits_{x}(-L^x_x+(L_{F})^x_x)l_F^x\nu^{x,x}_{F,ex}+\sum\limits_{x\neq y\in F}\nu^{x,y}_{F,ex}$ conditioned to have exactly $N^x_y(l_F)$ excursions from $x$ to $y$ out of $F$ for all $x\neq y\in F$.
\end{rem}
\begin{defn}
Suppose $\chi$ is a non-negative function on $S$ vanishing on $F$. For an excursion $(e,A,B)$, define the real-valued function $\langle\chi,\cdot\rangle$ of the excursion as follows:
$$\langle \chi,(e,A,B)\rangle=\int\limits\chi(e(t))\,dt.$$
\end{defn}
\begin{lem}\label{lem:laplace transform for the excursion measure}
We see that the excursion measure $\nu^{x,y}_{F,ex}$ varies as the generator changes. Let $\nu^{x,y,\chi}_{F,ex}$ be the excursion measure when $L$ is replaced by $L-M_{\chi}$. Define $(R^F_{\chi})^x_y$ as $(R^F)^x_y$ when $L$ is replaced by $L-M_{\chi}$. Then,
$$e^{-\langle\chi,\cdot\rangle}\cdot\nu^{x,y}_{F,ex}=\frac{(R^F_{\chi})^x_y}{(R^F)^x_y}\nu^{x,y,\chi}_{F,ex}.$$
In particular,
$$\nu^{x,y}_{F,ex}[e^{-\langle\chi,\cdot\rangle}]=\frac{(R^F_{\chi})^x_y}{(R^F)^x_y}.$$
\end{lem}

Accordingly, we have the following corollary,
\begin{cor}
\begin{multline*}
\mu(1_{\{\text{l visits F}\}}g(l_{F})e^{-\sum\limits_{\mathcal{E}_{F}(l)}\langle\chi,\cdot\rangle})\\
=\mu\left(1_{\{\text{l visits F}\}}g(l_F)\prod\limits_{x\neq y\in F}\left(\frac{(R^{F}_{\chi})^x_y}{(R^F)^x_y}\right)^{N^x_y(l_F)}e^{\sum\limits_{x\in F}(L^x_x-(L_F)^x_x)l_F^x\left(1-\frac{(R^F_\chi)^x_x}{(R^F)^x_x}\right)}\right).
\end{multline*}
\end{cor}

\subsection{Further properties of the multi-occupation field}
We know the loop measure varies as the generator varies. To emphasize this, we write $\mu(L,dl)$ instead of $\mu(dl)$.
\begin{prop}\label{compatibility with Feynman-Kac}
 $e^{-\langle l,\chi\rangle}\mu(L,dl)=\mu(L-M_{\chi},dl)$ for positive measurable function $\chi$ on $S$.
\end{prop}
\begin{proof}
It is the direct consequence of the Feynman-Kac formula. To be more precise,
\begin{multline*}
e^{-\langle l,\chi\rangle}\mu(L,dl)=\sum\limits_{x\in S}\int\frac{1}{t}\mathbb{P}_{t,x}^x[e^{-\langle l,\chi\rangle},dl]\,dt=\sum\limits_{x\in S}\int\frac{1}{t}\mathbb{P}_{t}^x[e^{-\langle l,\chi\rangle}1_{\{l(t)=x\}},dl]\,dt\\
=\sum\limits_{x\in S}\int\frac{1}{t}\mathbb{P}_{\!(\chi)\,t}^x[1_{\{l(t)=x\}},dl]\,dt=\mu(L-M_{\chi},dl).
\end{multline*}
\end{proof}

\begin{prop}\label{expectation of the multi-occupation field}
Suppose $f:S^n\rightarrow\mathbb{R}$ is positive measurable, then
$$\mu(\langle l,f\rangle)=\sum\limits_{(y_1,\ldots,y_n)\in S^n}\!\!V^{y_1}_{y_2}\cdots V^{y_n}_{y_1}f(y_1,\ldots,y_n).$$
\end{prop}
\begin{proof}
\begin{align*}
\mu(\langle l,f\rangle)=&\int\limits_0^{\infty}\frac{1}{t}\sum\limits_{x\in S}\mathbb{P}_{t,x}^{x}[\sum\limits_{j=0}^{n-1}\int\limits_{0<s^1<\cdots<s^n<t}f\circ r_j(l(s^1),\ldots,l(s^n))\prod\limits_{i=1}^{n}\,ds^i]\,dt\\
=&\int\limits_{0<s^1<\cdots<s^n<t<\infty}\sum\limits_{j=1}^{n}\sum\limits_{(x,x_1,\ldots,x_n)\in S^{n+1}}\,ds^1\cdots\,ds^n\,dt\\
&\frac{1}{t}f\circ r_j(x_1,\ldots,x_n)(P_{s^1})^{x}_{x_1}(P_{s^2-s^1})^{x_2}_{x_1}\cdots(P_{t-s^n})^{x_n}_x\\
=&\int\limits_{0<s^1<\cdots<s^n<t<\infty}\sum\limits_{j=1}^{n}\sum\limits_{(x_1,\ldots,x_n)\in S^n}\,ds^1\cdots\,ds^n\,dt\\
&\frac{1}{t}f\circ r_j(x_1,\ldots,x_n)(P_{s^2-s^1})^{x_1}_{x_2}(P_{s^3-s^2})^{x_2}_{x_3}\cdots(P_{t-s^n+s^1})^{x_n}_{x_1}.
\end{align*}
Performing the change of variables $a^0=s^1,a^1=s^2-s^1,\ldots,a^{n-1}=s^{n}-s^{n-1},a^n=t-s^n+s^1$,
\begin{align*}
\mu(\langle l,f\rangle)=&\int\limits_{a^0,\ldots,a^n>0,a^n>a^0}\,da^0\cdots\,da^n\sum\limits_{(x_1,\ldots,x_n)\in S^n}\sum\limits_{j=0}^{n-1}\frac{1}{a^1+\cdots+a^n}\\
&f\circ r_j(x_1,\ldots,x_n)(P_{a^1})^{x_1}_{x_2}\cdots(P_{a^n})^{x_n}_{x_1}\\
=&\int\limits_{a^1,\ldots,a^n>0}\,da^1\cdots\,da^n\sum\limits_{(x_1,\ldots,x_n)\in S^n}\sum\limits_{j=0}^{n-1}\frac{a_n}{a^1+\cdots+a^n}\\
&f\circ r_j(x_1,\ldots,x_n)(P_{a^1})^{x_1}_{x_2}\cdots(P_{a^n})^{x_n}_{x_1}.
\end{align*}
Changing again variables with  $b^{1+j}=a^1,\ldots,b^n=a^{n-j},b^1=a^{n-j+1},\ldots,b^{j}=a^{n}$ and $y_{1+j}=x_1,\ldots,y_n=x_{n-j},y_1=x_{n-j+1},\ldots,y_{j}=x_{n}$, and summing the integrals for all $j$,
\begin{align*}
\langle l,f\rangle=&\int\limits_{b^1,\ldots,b^n>0}\sum\limits_{(y_1,\ldots,y_n)\in S^n}(P_{b^1})^{y_1}_{y_2}\cdots(P_{b^n})^{y_n}_{y_1}f(y_1,\ldots,y_n)\,db^1\cdots\,db^n\\
=&\sum\limits_{(y_1,\ldots,y_n)\in S^n}V^{y_1}_{y_2}\cdots V^{y_n}_{y_1}f(y_1,\ldots,y_n).
\end{align*}
\end{proof}
 Define $\tilde{\mathfrak{S}}_{n,m} \subset\mathfrak{S}_{n+m}$ to be the collection of permutations $\sigma$ on $\{1,\ldots,n+m\}$ such that the order of $1,\ldots,n$ and $n+1,\ldots,n+m$ is preserved under the permutation $\sigma$ respectively, i.e.
$$\tilde{\mathfrak{S}}_{n,m}=\{\sigma\in\mathfrak{S}_{n+m}:\sigma(1)<\cdots<\sigma(n)\text{ and }\sigma(n+1)<\cdots<\sigma(n+m)\}.$$
Define $\mathfrak{S}^1_{n,m}=\{\sigma\in\mathfrak{S}_{n,m};\sigma(1)=1\}$. Then, we have $\sigma(1)<\cdots<\sigma(n)$ for $\sigma\in\mathfrak{S}^1_{n,m}$.

\begin{prop}[Shuffle product]\label{shuffle product}
Suppose $f:S^n\rightarrow\mathbb{R},g:S^m\rightarrow\mathbb{R}$ bounded or positive and measurable. Then,
$$\langle l,f\rangle \langle l,g\rangle=\sum\limits_{j=0}^{m-1}\sum_{\sigma\in \tilde{\mathfrak{S}}_{n,m}}\langle l,(f\otimes (g\circ r_j))\circ\sigma^{-1}\rangle.$$
\end{prop}
\begin{proof}
Let $t$ be the length of $l$.
\begin{align*}
\langle l,f\rangle\langle l,g\rangle=&\sum\limits_{j=0}^{n-1}\sum\limits_{k=0}^{m-1}\int\limits_{0<u^1<\cdots<u^n<t}f\circ r_j(l(u^1),\ldots,l(u^n))\,du^1\cdots\,du^n\\
&\int\limits_{0<v^1<\cdots<v^m<t}g\circ r_k(l(v^1),\ldots,l(v^m))\,dv^1\cdots\,dv^m\\
=&\sum\limits_{j=0}^{n-1}\sum\limits_{k=0}^{m-1}\int\limits_{\begin{subarray}{l}0<u^1<\cdots<u^n<t\\
0<v^1<\cdots<v^m<t\end{subarray}}f\circ r_j(l(u^1),\ldots,l(u^n))\\
&g\circ r_k(l(v^1),\ldots,l(v^m))\,du^1\cdots\,du^n\,dv^1\cdots\,dv^m.
\end{align*}
Let $w=(u^1,\ldots,u^n,v^1,\ldots,v^m)$. Almost surely, $u^1<\cdots<u^n,v^1<\cdots<v^m$ are different from each other. Let $s=(s^1,\ldots,s^{m+n})$ be the rearrangement of $w$ in increasing order. Almost surely, for each $w$, there exists a unique $\sigma\in \tilde{\mathfrak{S}}_{n,m}$ such that $s=\sigma(w)$. We change $w$ by $\sigma^{-1}(s)$,
\begin{align*}
\langle l,f\rangle\langle l,g\rangle=&\begin{multlined}[t]
\sum\limits_{j=0}^{n-1}\sum\limits_{k=0}^{m-1}\sum\limits_{\sigma\in \tilde{\mathfrak{S}}_{n,m}}\int\limits_{0<s^1<\cdots<s^{n+m}<t}\,ds^1\cdots\,ds^{n+m}\\
(f\circ r_j)\otimes(g\circ r_k)\circ\sigma^{-1}(l(s^1),\ldots,l(s^{m+n}))
\end{multlined}\\
=&\sum\limits_{\sigma\in \mathfrak{S}_{n,m}}\int\limits_{0<s^1<\cdots<s^{n+m}<t}(f\otimes g)\circ\sigma^{-1}(l(s^1),\ldots,l(s^{m+n}))\,ds^1\cdots\,ds^{n+m}\\
=&\sum\limits_{\begin{subarray}{l}
\sigma\in\mathfrak{S}^1_{n,m}\\
r\in R\end{subarray}}\int\limits_{0<s^1<\cdots<s^{n+m}<t}(f\otimes g)\circ r\circ\sigma^{-1}(l(s^1),\ldots,l(s^{m+n}))\,ds^1\cdots\,ds^{n+m}\\
=&\sum\limits_{\sigma\in\mathfrak{S}^1_{n,m}}\langle l,(f\otimes g)\circ\sigma^{-1}\rangle\\
=&\sum\limits_{j=0}^{m-1}\sum_{\sigma\in \tilde{\mathfrak{S}}_{n,m}}\langle l,(f\otimes (g\circ r_j))\circ\sigma^{-1}\rangle.
\end{align*}
\end{proof}
\begin{cor}\label{moments of the 1-occupation field}
\begin{equation*}
\mu(l^{x_1}\cdots l^{x_n})=\frac{1}{n}\sum\limits_{\sigma\in\mathfrak{S}_n}V_{x_{\sigma_2}}^{x_{\sigma_1}}\cdots V_{x_{\sigma_1}}^{x_{\sigma_n}}.
\end{equation*}
\end{cor}
\begin{proof}
$$l^{x_1}\cdots l^{x_n}=\prod\limits_{i=1}^{n}\int\limits_{0}^{|l|}1_{\{l(t_i)=x_{i}\}}\,dt_i=\int\limits_{0}^{|l|}\prod\limits_{i=1}^{n}1_{\{l(t_i)=x_{i}\}}\,dt_i.$$
In the above expression, almost surely, one can write $t_1,\ldots,t_n$ in increasing order, $s_1=t_{\sigma(1)}<\cdots<s_n=t_{\sigma(n)}$ for a unique $\sigma\in\mathfrak{S}_n$. Then,
\begin{align*}
\int\limits_{0}^{|l|}\prod\limits_{i=1}^{n}1_{\{l(t_i)=x_{i}\}}\,dt_i=&\sum\limits_{\sigma\in\mathfrak{S}_n}\int\limits_{0<t_{\sigma(1)}<\cdots<t_{\sigma(n)}<|l|}\prod\limits_{i=1}^{n}1_{\{l(t_{\sigma(i)})=x_{\sigma(i)}\}}\,dt_{\sigma(i)}\\
=&\sum\limits_{\sigma\in\mathfrak{S}_n}\int\limits_{0<s_1<\cdots<s_n<|l|}\prod\limits_{i=1}^{n}1_{\{l(s_i)=x_{\sigma(i)}\}}\,ds_i.
\end{align*}
Since $\mathfrak{S}_nr_{j}=\mathfrak{S}_n$ for all $j=1,\ldots,n$, the above expression equals to
$$\sum\limits_{\sigma\in\mathfrak{S}_n}\int\limits_{0<s_1<\cdots<s_n<|l|}\prod\limits_{i=1}^{n}1_{\{l(s_i)=x_{\sigma(i+j)}\}}\,ds_i.$$
for all $j=1,\ldots,n$. Finally,
\begin{align*}
l^{x_1}\cdots l^{x_n}=&\frac{1}{n}\sum\limits_{j=1}^{n}\sum\limits_{\sigma\in\mathfrak{S}_n}\int\limits_{0<s_1<\cdots<s_n<|l|}\prod\limits_{i=1}^{n}1_{\{l(s_i)=x_{\sigma(i+j)}\}}\,ds_i\\
=&\frac{1}{n}\sum\limits_{\sigma\in\mathfrak{S}_n}l^{x_{\sigma(1)},\ldots,x_{\sigma(n)}}.
\end{align*}
Then, by Proposition \ref{expectation of the multi-occupation field}, we are done.
\end{proof}
\begin{cor}
The linear space generated by all the multi-occupation fields is an algebra.
\end{cor}
\begin{proof}
By shuffle product, the operation of multiplication is closed.
\end{proof}

\begin{thm}[Blackwell's theorem, \cite{meyer}]
Suppose $(E,\mathcal{E})$ is a Blackwell space, $\mathcal{S},\mathcal{F}$ are sub-$\sigma$-field of $\mathcal{E}$ and $\mathcal{S}$ is separable. Then $\mathcal{F}\subset\mathcal{S}$ iff every atom of $\mathcal{F}$ is a union of atoms of $\mathcal{S}$.
\end{thm}

\begin{thm}\label{thm:multi-occupation fields generate the Borel-sigma-field on the loops}
The family of all multi-occupation fields generates the Borel-$\sigma$-field on the loops.
\end{thm}

\begin{lem}\label{lem:a corollary of the theorem of Blackwell}
Suppose $(E,\mathcal{B}(E))$ is a Polish space with the Borel-$\sigma$-field. Let $\{f_i,i\in\mathbb{N}\}$ be measurable functions and denote $\mathcal{F}=\sigma(f_i,i\in\mathbb{N})$. Then, $\mathcal{F}=\mathcal{B}(E)$ iff for all $x\neq y\in E$, there exists $f_i$ such that $f_i(x)\neq f_i(y)$.
\end{lem}
\begin{proof}
Since $E$ is Polish, $\mathcal{B}(E)$ is separable and $(E,\mathcal{B}(E))$ is Blackwell space. The atoms of $\mathcal{B}(E)$ are all the one point sets. Obviously, $\mathcal{F}\subset\mathcal{B}(E)$ and $\mathcal{F}$ is separable. By Blackwell's theorem, $\mathcal{F}=\mathcal{B}(E)$ iff the atoms of $\mathcal{F}$ are all the one point sets which is equivalent to the following: for all $x\neq y\in E$, there exists $f_i$ such that $f_i(x)\neq f_i(y)$.
\end{proof}

\begin{proof}[Proof for Theorem \ref{thm:multi-occupation fields generate the Borel-sigma-field on the loops}]
By Lemma \ref{lem:a corollary of the theorem of Blackwell} and the fact that $$\{l^{x_1,\ldots,x_m}:m\in\mathbb{N}_{+},(x_1,\ldots,x_m)\in S^m\}$$ is countable, it is sufficient to show that given all the multi-occupation fields of the loop $l$, the loop is uniquely determined.\\
Note first that the length of the loop can be recovered from the occupation field as $|l|=\sum\limits_{x\in S}l^x$.\\
Let $J(l)=\max\{n\in\mathbb{N}:\exists (x_1,\ldots,x_n)\in S^n\text{ such that }x_i\neq x_{i+1}\text{ for }i=1,\ldots,n-1,x_1\neq x_n\text{ and }l^{x_1,\ldots,x_n}>0\}$, the total number of the jumps in the loop $l$. Define $D(l)$ to be the set of discrete pointed loop such that $l^{x_1,\ldots,x_{J(l)}}>0$. As a discrete loop is viewed as an equivalent class of discrete pointed loop, it appears that $D(l)$ is actually the discrete loop $l^d$. A loop is defined by the discrete loop with the corresponding holding times.  It remains to show that the corresponding holding times can be recovered from the multi-occupation field. Suppose we know that the multiplicity of the discrete loop $n(l^d)=n$, the length of the discrete loop $J(l)=qn$ and that $(x_1,\ldots,x_q,\ldots,x_1,\ldots,x_q)\in D(l)$ is a pointed loop representing $l^d$. Then the loop $l$ can be written in the following form:
$$(x^1_1,\tau^1_1,\ldots,x^1_{q},\tau^1_{q},\ldots,x^{n}_1,\tau^{n}_1,\ldots,x^{n}_{q},\tau^{n}_{q})^o$$
with $x^i_j=x_j, i=1,\ldots,q$ and $(\tau^1_1,\ldots,\tau^1_{q})\geq\cdots\geq(\tau^{n}_1,\ldots,\tau^{n}_{q})$ in the lexicographical order.  For $k\in M_{n\times q}(\mathbb{N}_{+})$ a $n$ by $q$ matrix, define $y(k)\in S^{\sum\limits_{i,j}k^i_j}$ as follows
$$y(k)=(\underbrace{x^1_1,\ldots,x^1_1}\limits_{k^1_1\text{ times}},\underbrace{x^1_2,\ldots,x^1_2}\limits_{k^1_2\text{ times}},\ldots,\underbrace{x^n_q,\ldots,x^n_q}\limits_{k^n_q\text{ times}}).$$
Define $k!=\prod\limits_{i,j}k^i_j!$. Define $K^i=(k^i_1,\ldots,k^i_q)$ for $i=1,\ldots,n$. Define $\tau^i=(\tau^i_1,\ldots,\tau^i_{q})$ for $i=1,\ldots,n$. For $K\in \mathbb{N}^{q}$ and $t\in  \mathbb{R}^{q}$, define the polynomial $f^K(t)=\prod\limits_{j=1}^{q} (t_i)^{K_j}$.
We have the following expression,
$$l^{y(k)}=\frac{1}{k!}\sum\limits_{i=1}^{n}(f^{K^1}\otimes\cdots\otimes f^{K^n})\circ r_i(\tau^1,\ldots,\tau^n)$$
where $r_i(\tau^1,\ldots,\tau^n)=(\tau^{n-i+1},\ldots,\tau^n,\tau^1,\ldots,\tau^{n-i})$.
All the holding times are bounded by the length $|l|$ of the loop. By the theorem of Weierstrass, for any continuous function $f$ on $(\mathbb{R}^q)^n$, the following quantity is determined by the family of occupation fields:
$$\sum\limits_{i=1}^{n}f\circ r_i(\tau^1,\ldots,\tau^n).$$
As a consequence, $\sum\limits_{i=1}^{n}\delta_{ r_i(\tau^1,\ldots,\tau^n)}$ is uniquely determined. Since we order $\tau^1\geq\cdots\geq\tau^n$ in the lexicographical order, $(\tau^1,\ldots,\tau^n)$ is uniquely determined. Finally, the loop $l$ is determined by the family of the multi-occupation fields of $l$ and we are done.
\end{proof}

\subsection{The occupation field in the transient case}
\textbf{Assumption}: Throughout this section, assume we are in the transient case.
\begin{prop}\label{laplace transform of the occupation field}
Suppose $\chi$ is a non-negative function on $S$ with compact support $F$. Let $\rho(M_{\sqrt{\chi}}VM_{\sqrt{\chi}})$ be the spectral radius of $M_{\sqrt{\chi}}VM_{\sqrt{\chi}}$.
Then, for $z\in D=\{z\in\mathbb{C}:\operatorname{Re}(z)<\frac{1}{\rho(M_{\sqrt{\chi}}VM_{\sqrt{\chi}})}\}$, the following equation holds:
$$\mu(e^{z\langle l,\chi\rangle}-1)=-\ln\det(I-zM_{\sqrt{\chi}}VM_{\sqrt{\chi}}).$$
Outside of $D$, $\mu(|e^{z\langle l,\chi\rangle}-1|)=\infty$.

\end{prop}
\begin{proof}
Suppose $n=|\supp(\chi)|$ and $\lambda_1,\ldots,\lambda_n$ are the eigenvalues of $M_{\sqrt{\chi}}VM_{\sqrt{\chi}}$ ordered in the sense of non-increasing module. Then, $|\lambda_1|=\rho(M_{\sqrt{\chi}}VM_{\sqrt{\chi}})$. By Corollary \ref{moments of the 1-occupation field} and Proposition \ref{expectation of the multi-occupation field},
$$\mu(\langle l,\chi\rangle^m)=(m-1)!\sum\limits_{(x^1,\ldots,x^m)\in S^m}V^{y_1}_{y_2}\cdots V^{y_m}_{y_1}\chi(y_1)\cdots \chi(y_m)=(m-1)!\Tr((M_{\sqrt{\chi}}VM_{\sqrt{\chi}})^m).$$
We have: $$e^{z}=1+z+\cdots+\frac{z^n}{n!}+z^{n+1}\int\limits_{0<s_1<\cdots<s_{n+1}<1}e^{s_1 z}\,ds_1\cdots\,ds_{n+1}.$$
Therefore
$$|e^{z+h}-1-z-\cdots-\frac{z^n}{n!}|\leq e^{\max(\operatorname{Re}(z),0)}\frac{|z|^{n+1}}{(n+1)!}.$$

In particular, $|e^{x}-1|\leq e^{\max(\operatorname{Re}(x),0)}|x|$ and $|e^{x}-1-x|\leq e^{\max(\operatorname{Re}(x),0)}\frac{|x|^2}{2}$.

For $z\in\mathbb{C}$ such that $\operatorname{Re}(z)< 1/\rho(M_{\sqrt{\chi}}VM_{\sqrt{\chi}})=1/|\lambda_1|$, let $b=\max(\operatorname{Re}(z),0)$,
\begin{align*}
\mu(|e^{z\langle l,\chi\rangle}-1|)\leq & \mu(e^{b\langle l,\chi\rangle}|z|\langle l,\chi\rangle)=\mu(\sum\limits_{m=0}^{\infty}\frac{|z|b^m\langle l,\chi\rangle^{m+1}}{m!})\\
=&\sum\limits_{m=0}^{\infty}\frac{|z|b^m\mu(\langle l,\chi\rangle^{m+1})}{m!}=\sum\limits_{m=0}^{\infty}|z|b^m\Tr((M_{\sqrt{\chi}}VM_{\sqrt{\chi}})^{m+1})\\
\leq & |z||\supp(\chi)|\sum\limits_{m=0}^{\infty}b^m|(\rho(M_{\sqrt{\chi}}VM_{\sqrt{\chi}}))^{m+1}<\infty
\end{align*}
Consequently, $\Phi(z)=\mu(e^{z\langle l,\chi\rangle}-1)$ is well-defined for $z\in D$. Next, we will show that $\Phi(z)$ is analytic in D. Fix $z_0\in D$, take $h$ small enough that $z_0+h\in D$. By an argument very similar to the above one, we have that $\mu(e^{z_0\langle l,\chi\rangle}\langle l,\chi\rangle)$ and $\mu(e^{(\operatorname{Re}(z_0)+\max(\operatorname{Re}(h),0))\langle l,\chi\rangle}\frac{\langle l,\chi\rangle^2}{2})$ are well-defined and finite.
\begin{multline*}
|\Phi(z_0+h)-\Phi(z_0)-h\mu(e^{z_0\langle l,\chi\rangle}\langle l,\chi\rangle)|\\
=\mu(|e^{z_0\langle l,\chi\rangle}(e^{h\langle l,\chi\rangle}-1-h\langle l,\chi\rangle)|)\\
\leq\mu\left(e^{\operatorname{Re}(z_0)\langle l,\chi\rangle}e^{\max(\operatorname{Re}(h),0)\langle l,\chi\rangle}\frac{h^2\langle l,\chi\rangle^2}{2}\right)=O(h^2).
\end{multline*}
Finally, by dominated convergence, for $|z|<1/\rho(M_{\sqrt{\chi}}VM_{\sqrt{\chi}})$,
$$\Phi(z)=\sum\limits_{n\geq 1}\frac{1}{n}\Tr((M_{\sqrt{\chi}}VM_{\sqrt{\chi}})^n)=-\ln\det(1-zM_{\sqrt{\chi}}VM_{\sqrt{\chi}}).$$
Since $\Phi(z)$ is analytic in $D=\{z\in\mathbb{C}:\operatorname{Re}(z)<1/\rho(M_{\sqrt{\chi}}VM_{\sqrt{\chi}})\}$, $\Phi(z)$ is the unique analytic continuation of  $-\ln\det(1-zM_{\sqrt{\chi}}VM_{\sqrt{\chi}})$ in $D$.\\

$\ln\det(I-zM_{\sqrt{\chi}}VM_{\sqrt{\chi}})$ cannot be defined on $\mathbb{C}$ as an analytic function. Nevertheless, after cutting down several half lines starting from $1/\lambda_1,\ldots,1/\lambda_n$, it is analytic and equals $-\sum\limits_{i=1}^{n}\ln(1-z\lambda_i)$. Moreover, when $z$ converges to some $\lambda_i$, $|-\ln\det(I-zM_{\sqrt{\chi}}VM_{\sqrt{\chi}})|$ tends to infinity. But we have showed that $\mu(e^{z\langle l,\chi\rangle}-1)=-\ln\det(I-zM_{\sqrt{\chi}}VM_{\sqrt{\chi}})$ is well-defined as an analytic function on $D$. Consequently, $1/\lambda_1,\ldots,1/\lambda_n$ lie in $D^c$, i.e. $\operatorname{Re}(\frac{1}{\lambda_i})\geq \frac{1}{\rho(M_{\sqrt{\chi}}GM_{\sqrt{\chi}})}=\frac{1}{|\lambda_1|}$. In particular, $\lambda_1=\rho(M_{\sqrt{\chi}}GM_{\sqrt{\chi}})$. For $x\geq \frac{1}{\rho(M_{\sqrt{\chi}}GM_{\sqrt{\chi}})}$, $$\mu(|e^{x\langle l,\chi\rangle}-1|)=\mu(e^{x\langle l,\chi\rangle}-1)\geq \mu(e^{\frac{1}{\rho(M_{\sqrt{\chi}}GM_{\sqrt{\chi}})}\langle l,\chi\rangle}-1).$$
By monotone convergence,
\begin{align*}
\mu(e^{\frac{1}{\rho(M_{\sqrt{\chi}}GM_{\sqrt{\chi}})}\langle l,\chi\rangle}-1)=&\lim\limits_{y\uparrow\lambda_1}\mu(e^{y\langle \chi,l\rangle}-1)=\lim\limits_{y\uparrow\lambda_1}-\ln\det(I-yM_{\chi}VM_{\chi})\\
=&\lim\limits_{y\uparrow\lambda_1}|-\ln\det(I-yM_{\chi}VM_{\chi})|=\infty.
\end{align*}
Consequently, for $x\geq \frac{1}{\rho(M_{\sqrt{\chi}}GM_{\sqrt{\chi}})}$, $\mu(|e^{x\langle l,\chi\rangle}-1|)=\infty$. For all $y\in\mathbb{R}$, $\mu(|e^{iy\langle l,\chi\rangle}-1|)<\infty$. Therefore, by the triangular inequality, for $z=x+iy\notin D$,
\begin{align*}
\mu(|e^{z\langle l,\chi\rangle}-1|)\geq & |\mu(|e^{z\langle l,\chi\rangle}-e^{iy\langle l,\chi\rangle}|)-\mu(|e^{iy\langle l,\chi\rangle}-1|)|\\
=&|\mu(|e^{x\langle l,\chi\rangle}-1|)-\mu(|e^{iy\langle l,\chi\rangle}-1|)|=\infty.
\end{align*}
\end{proof}

\begin{lem}\label{lem:linear algebra lemma}
Suppose $\chi$ is a finitely supported non-negative function on $S$ and $F$ contains the support of $\chi$. Then,
\begin{align*}
\frac{\det(V_{F})}{\det((V_{\chi})_F)}=&\det(I+(M_{\chi})_FV_F)=\det(I+M_{\sqrt{\chi}}VM_{\sqrt{\chi}})\\
=&\begin{vmatrix}
I_F & V_F\\
-(M_{\chi})_F & I_F
\end{vmatrix}=1+\sum\limits_{A\subset F,A\neq \phi}\prod\limits_{x\in A}\chi(x)V_A.
\end{align*}
\end{lem}
\begin{proof}
By the resolvent equation, we have $V_F=(V_F)_{\chi}+(V_F)_{\chi}(M_{\chi})_FV_F$. By Proposition \ref{trace and change of time}, we have $(V_\chi)_F=(V_F)_\chi$. Combining these two results, we have $V_F=(V_{\chi})_F+(V_{\chi})_F(M_{\chi})_FV_F$. Consequently,
$$\frac{\det(V_{F})}{\det((V_{\chi})_F)}=\det(I+(M_{\chi})_FV_F).$$
The last equality follows from simple calculations in linear algebra.
\end{proof}
\begin{cor}\label{positive laplace transform of the occupation field}
For non-negative $\chi$ not necessarily finitely supported,
$$e^{\mu(1-e^{-\langle l,\chi\rangle})}=1+\sum\limits_{F\subset S, 0<|F|<\infty}\prod\limits_{x\in F}\chi(x)\det(V_{F})$$
\end{cor}
\begin{proof}
For $\chi$ a non-negative finitely supported function, by Proposition \ref{laplace transform of the occupation field} with Lemma \ref{lem:linear algebra lemma},
\begin{align*}
e^{\mu(1-e^{-\langle l,\chi\rangle})}=&\det(I+M_{\sqrt{\chi}}VM_{\sqrt{\chi}})=\det(I+(M_{\chi})_{\supp(\chi)}V_{\supp(\chi)})\\
=&1+\sum\limits_{F\subset \supp(\chi),F\neq\phi}(\prod\limits_{x\in F}\chi(x))\det(V_{F}).
\end{align*}
The trace of the Markov process on F has the potential $V_F$ and generator $\tilde{L}$. Since $\det(-\tilde{L})>0$ and $(-\tilde{L})V_{F}=Id$, $\det(V_F)>0$. Finally, the result comes from monotone convergence theorem.
\end{proof}

\begin{cor}
For $a\geq 0$, let $\chi=a\delta_x$, then $\mu(1-e^{-al^x})=\ln(1+aV_x^x)$. As a result, $\mu(l^x\in dt)=\frac{1}{t}e^{-t/V^x_x}\,dt$ for $t>0$.
\end{cor}

\begin{prop}\ \label{laplace transform of the occupation field of the trivial or non-trivial loop}
For non-negative function $\chi$,
\begin{align*}
&\mu(1_{\{\text{l is trivial}\}}(1-e^{-\langle l,\chi\rangle}))=\ln(\prod\limits_{x\in S}\frac{\chi(x)-L^x_x}{-L^x_x})\\
&\mu(1_{\{\text{l is non-trivial}\}}(1-e^{-\langle l,\chi\rangle}))=\ln(I+M_{\sqrt{\chi}}VM_{\sqrt{\chi}})+\ln(\prod\limits_{x\in S}\frac{-L^x_x}{\chi(x)-L^x_x}).
\end{align*}
\end{prop}
\begin{proof}
Since $\mu(p(\xi)=1,\xi_1=x,\tau_1\in dt^1)=\frac{e^{L^x_x t^1}}{t^1}dt^1$,
$$\mu(1_{\{l\text{ is trivial}\}}(1-e^{-\langle l,\chi\rangle}))=\sum\limits_{x\in S} \int\limits_{0}^{\infty}\frac{e^{L^x_x t^1}}{t^1}(1-e^{-\chi(x)t^1})dt^1=\ln(\prod\limits_{x\in S}\frac{\chi(x)-L^x_x}{-L^x_x}).$$
Combining with Proposition \ref{laplace transform of the occupation field}, we have
\begin{multline*}
\mu(1_{\{l\text{ is non-trivial}\}}(1-e^{-\langle l,\chi\rangle}))\\
=\mu(1-e^{-\langle l,\chi\rangle})-\mu(1_{\{l\text{ is trivial}\}}(1-e^{-\langle l,\chi\rangle}))\\
=\ln(\det(I+M_{\sqrt{\chi}}VM_{\sqrt{\chi}}))+\ln(\prod\limits_{x\in S}\frac{-L^x_x}{\chi(x)-L^x_x}).
\end{multline*}
\end{proof}

\begin{prop}\label{a variant of the laplace transform for the occupation field}
If $\chi_1,\ldots,\chi_n$ are finitely supported non-negative functions on $S$, and for $A$ a subset of $\{1,\ldots,n\}$ we set  $\chi_{A}=\sum\limits_{i\in A}\chi_i$, then for $n\geq 2$,
$$\mu(\prod\limits_{i=1}^{n}(1-e^{-\langle l,\chi_i\rangle}))=-\sum\limits_{A\subset \{1,\ldots,n\}}(-1)^{|A|}\ln\det(I+M_{\sqrt{\chi_A}}VM_{\sqrt{\chi_A}});$$
$$\mu(1_{\{l\text{ is trivial}\}}\prod\limits_{i=1}^{n}(1-e^{-\langle l,\chi_i\rangle}))=-\sum\limits_{A\subset \{1,\ldots,n\}}(-1)^{|A|}\ln(\prod\limits_{x\in F_A}\frac{-L^x_x+\chi_A(x)}{-L^x_x}).$$
\end{prop}
\begin{proof}
We see that
$$\prod\limits_{i=1}^{n}(1-e^{-\langle l,\chi_i\rangle})=-\sum\limits_{A\subset\{1,\ldots,n\}}(-1)^{|A|}(1-e^{-\langle l,\chi_{A}\rangle}).$$
Therefore,
\begin{align*}
\mu(\prod\limits_{i=1}^{n}(1-e^{-\langle l,\chi_i\rangle}))=&-\sum\limits_{A\subset\{1,\ldots,n\}}(-1)^{|A|}\mu(1-e^{-\langle l,\chi_{A}\rangle})\\
=&-\sum\limits_{A\subset \{1,\ldots,n\}}(-1)^{|A|}\ln\det(I+M_{\sqrt{\chi_A}}VM_{\sqrt{\chi_A}}).
\end{align*}
The last equality is deduced from Proposition \ref{laplace transform of the occupation field}. By a similar method and Proposition \ref{laplace transform of the occupation field of the trivial or non-trivial loop}, we get the following expression for the trivial loops:
$$\mu(1_{\{l\text{ is trivial}\}}\prod\limits_{i=1}^{n}(1-e^{-\langle l,\chi_i\rangle}))=-\sum\limits_{A\subset \{1,\ldots,n\}}(-1)^{|A|}\ln(\prod\limits_{x\in F_A}\frac{-L^x_x+\chi_A(x)}{-L^x_x}).$$
\end{proof}

\begin{prop}\label{non-trivial loop visits F}
For a finite subset $F\subset S$,
$$\mu(l\text{ is non-trivial and }l\text{ visits }F)=\ln(\prod\limits_{x\in F}(-L^x_x)\det(V_F)).$$
\end{prop}
\begin{proof}
By Proposition \ref{laplace transform of the occupation field of the trivial or non-trivial loop},
$$\mu(1_{\{l\text{ is non-trivial}\}}(1-e^{-\langle l,t1_F\rangle}))=\ln\det(I+M_{\sqrt{t1_F}}VM_{\sqrt{t1_F}})+\ln(\prod\limits_{x\in F}\frac{-L^x_x}{t-L^x_x}).$$
By Lemma \ref{lem:linear algebra lemma}, we have
$$\ln\det(I+M_{\sqrt{t1_F}}VM_{\sqrt{t1_F}})=\ln(1+\sum\limits_{A\subset F,A\neq\phi}t^{|A|}V_{A}).$$
Take $t\rightarrow\infty$, we have
$$\mu(1_{\{l\text{ is non-trivial and }l\text{ visits }F\}})=\ln(\prod\limits_{x\in F}(-L^x_x)\det(V_F)).$$
\end{proof}

Similarly, one has the following property.
\begin{prop}\label{Visiting a collection of compact set}
Suppose we are given $n\geq 2$ finite subset $F_1,\ldots,F_n$. For any subset $A\subset\{1,\ldots,n\}$, define $F_A=\bigcup\limits_{i\in A}F_i$. Then,
\begin{multline*}
\mu(l\text{ is not trivial and it visits all }F_i\text{ for }i=1,\ldots,n)\\
=-\sum\limits_{A\subset\{1,\ldots,n\},A\neq\phi}(-1)^{|A|}\ln\det(V_{F_A})+\sum\limits_{x\in\bigcap\limits_{i=1}^{n}F_i}\ln(-L^x_x).
\end{multline*}
\end{prop}
\begin{proof}
By Proposition \ref{a variant of the laplace transform for the occupation field}, take $\chi_i=t1_{F_i}$:
\begin{multline*}
\mu(1_{\{l\text{ is non-trivial}\}}\prod\limits_{i=1}^{n}(1-e^{-\langle l,t1_{F_i}\rangle}))\\
=-\sum\limits_{A\subset \{1,\ldots,n\},A\neq\phi}(-1)^{|A|}\ln\det(I+M_{\sqrt{\chi_A}}VM_{\sqrt{\chi_A}})\\
+\sum\limits_{A\subset \{1,\ldots,n\},A\neq\phi}(-1)^{|A|}\ln(\prod\limits_{x\in F_A}\frac{-L^x_x+\chi_A(x)}{-L^x_x}).
\end{multline*}
where $\chi_A=\sum\limits_{i\in A}t1_{F_i}$. By Lemma \ref{lem:linear algebra lemma}, for $A$ non-empty,
\begin{align*}
\det(I+M_{\sqrt{\chi_A}}VM_{\sqrt{\chi_A}})=&1+\sum\limits_{B\subset F_A,B\neq\phi}t^{|B|}(\prod\limits_{x\in B}(\sum\limits_{i\in A}1_{\{x\in F_i\}}))\det(V_B)\\
\sim & t^{|F_A|}(\prod\limits_{x\in F_A}(\sum\limits_{i\in A}1_{\{x\in F_i\}}))\det(V_{F_A})\text{ as }t\rightarrow\infty.
\end{align*}
And we have that
$$\prod\limits_{x\in F_A}\frac{-L^x_x+\chi_A(x)}{-L^x_x}\sim t^{|F_A|}\prod\limits_{x\in F_A}\frac{\sum\limits_{i\in A}1_{\{x\in F_i\}}}{L^x_x}\text{ as }t\rightarrow\infty.$$
As a result,
\begin{multline*}
\lim\limits_{t\rightarrow\infty}(-1)^A\left(\ln\det(I+M_{\sqrt{\chi_A}}VM_{\sqrt{\chi_A}})+\ln\left(\prod\limits_{x\in F_A}\frac{-L^x_x+\chi_A(x)}{-L^x_x}\right)\right)\\
=-\ln\det(V_{F_A})-\ln(\prod\limits_{x\in F_A}(-L^x_x)).
\end{multline*}
Then,
\begin{multline*}
\mu(l\text{ is not trivial and it visits all } F_i\text{ for }i=1,\ldots,n)\\
=\lim\limits_{t\rightarrow\infty}\mu(1_{\{l\text{ is non-trivial}\}}\prod\limits_{i=1}^{n}(1-e^{-\langle l,t1_{F_i}\rangle}))\\
=-\sum\limits_{A\subset\{1,\ldots,n\},A\neq\phi}(-1)^{|A|}\ln\det(V_{F_A})-\sum\limits_{A\subset\{1,\ldots,n\},A\neq\phi}(-1)^{|A|}\ln(\prod\limits_{x\in F_A}(-L^x_x)).
\end{multline*}
Finally, by inclusion-exclusion principle, we have
\begin{multline*}
\mu(l\text{ is not trivial and it visits all } F_i\text{ for }i=1,\ldots,n)\\
=-\sum\limits_{A\subset\{1,\ldots,n\},A\neq\phi}(-1)^{|A|}\ln\det(V_{F_A})+\sum\limits_{x\in\bigcap\limits_{i=1}^{n}F_i}\ln(-L^x_x).
\end{multline*}
\end{proof}

\begin{cor}
For $n\geq 2$ and $n$ different states $x_1,\ldots,x_n$,
$$\mu(l\text{ visits each state of }\{x_1,\ldots,x_n\})=-\sum\limits_{A\subset\{x_1,\ldots,x_n\},A\neq\phi}(-1)^{|A|}\ln\det(V_A).$$
\end{cor}

\begin{defn}
 For a loop $l$, let $N(l)$ be the number of different points visited by the loop. That is $N(l)=\sum\limits_{x\in S}1_{\{l^x>0\}}$.
\end{defn}

\begin{cor}\label{expectation of number of points covered by non-trivial loop}
 $$\mu(N1_{\{N>1\}})=\sum\limits_{x\in S}\ln(-L^x_x V^x_x).$$
\end{cor}
\begin{proof}
$$\mu(N1_{\{N>1\}})=\sum\limits_{x\in S}\mu(l\text{ is non-trivial and }l\text{ visits }x).$$
\end{proof}

\begin{cor}
 $$\mu(N^21_{\{N>1\}})=\sum\limits_{x,y\in S;x\neq y}\ln(\frac{V^x_xV^y_y}{V^x_xV^y_y-V^x_yV^y_x})+\sum\limits_{x\in S}\ln(-L^x_xV^x_x).$$
\end{cor}
\begin{proof}
$$\mu(N^21_{\{N>1\}})=\sum\limits_{x,y}\mu(l\text{ is non-trivial, }l\text{ visits }x\text{ and }l\text{ visits }y).$$
\end{proof}

Consider the Laguerre-type polynomial $L_k$ with generating function
$$e^{\frac{ut}{1+t}}-1=\sum\limits_{1}^{\infty}t^kL_k(u).$$
\begin{lem}\label{lem:laguerre}
$$\sum\limits_{1}^{\infty}|t^kL_k(u)|\leq e^{\frac{|ut|}{1-|t|}}-1.$$
\end{lem}
\begin{proof}
$$\sum\limits_{1}^{\infty}t^kL_k(u)=e^{\frac{ut}{1+t}}-1=\sum\limits_{k=1}^{\infty}\left(\frac{ut}{1+t}\right)^k=\sum\limits_{k=1}^{\infty}u^kt^k(1-t+t^2\cdots)^k.$$
Therefore,
$$\sum\limits_{1}^{\infty}|t^nL_n(u)|\leq\sum\limits_{n=1}^{\infty}|u|^n|t|^n(1+|t|+|t|^2+\cdots)^n=e^{\frac{|ut|}{1-|t|}}-1.$$
\end{proof}
\begin{prop}
$(\sqrt{k}(V^x_x)^{k/2}L_k\left(\frac{l^x}{V^x_x}\right),k\geq 1)$ are orthonormal in $L^2(\mu)$. More generally,
$$\mu\left(\sqrt{j}(V^x_x)^{j/2}L_j\left(\frac{l^x}{V^x_x}\right)\sqrt{k}(V^y_y)^{k/2}L_k\left(\frac{l^y}{V^y_y}\right)\right)=\delta^j_k.$$
\end{prop}
\begin{proof}
$\forall s,t\leq 0$ with $|s|,|t|$ small enough with $\frac{V^x_x s}{1-sV^x_x}, \frac{V^y_y t}{1-tV_y^y}<1/2$
\begin{multline*}
\mu\left(\left(\sum\limits_{1}^{\infty}(V^x_x s)^kL_k\left(\frac{l^x}{V^x_x}\right)\right)\left(\sum\limits_{1}^{\infty}(V^y_y t)^kL_k\left(\frac{l^y}{V^y_y}\right)\right)\right)=\mu\left((e^{\frac{l^x s}{1+V^x_x s}}-1)(e^{\frac{l^y t}{1+V^y_y t}}-1)\right)\\
=\mu\left(1-e^{\frac{l^x s}{1+V^x_x s}}\right)+\mu\left(1-e^{\frac{l^y t}{1+V^y_y t}}\right)-\mu\left(1-e^{\frac{l^x s}{1+V^x_x s}+\frac{l^y t}{1+V^y_y t}}\right)=-\ln(1-stV^x_yV^y_x).
\end{multline*}
Recall that $\mu((l^x)^n)=(n-1)!(V^x_x)^n$. By Lemma \ref{lem:laguerre},
\begin{align*}
\mu((\sum\limits_{k=1}^{\infty}|(V^x_x s)^kL_k(l^x/V^x_x)|)^2)\leq& \mu\left(\left(e^{\frac{l^x s}{1-V^x_x s}}-1\right)^2\right)=\sum\limits_{n,m=1}^{\infty}\frac{1}{n!}\frac{1}{m!}\mu\left(\left({\frac{l^x s}{1-V^x_x s}}\right)^{n+m}\right)\\
=&\sum\limits_{n,m=1}^{\infty}\frac{1}{n+m}\frac{(n+m)!}{n!m!}\left(\frac{V^x_x s}{1-V^x_x s}\right)^{n+m}\\
\leq& \sum\limits_{k\geq 1}\left(\frac{2V^x_x s}{1-V^x_x s}\right)^k<\infty.
\end{align*}
Therefore, $(\frac{1}{\sqrt{k}}(V^x_x)^kL_k(\frac{l^x}{V^x_x}),k\geq 1)\in L^2 ({\mu})$. Moreover, in the equation
$$\mu\left((\sum\limits_{1}^{\infty}(V^x_x s)^kL_k(l^x/V^x_x))(\sum\limits_{1}^{\infty}(V^y_y t)^kL_k(l^y/V^y_y))\right)=-\ln(1-stV^x_yV^y_x).$$
we can expand both sides as series of $s$ and $t$, compare the coefficients and deduce that
$$\mu\left((V^x_x)^jL_j(\frac{l^x}{V^x_x})(V^y_y)^kL_k(\frac{l^y}{V^y_y})\right)=\delta^j_k(V^x_yV^y_x)^k/k.$$
Therefore,
$$\mu\left(\sqrt{j}(V^x_x)^{j/2}L_j(\frac{l^x}{V^x_x})\sqrt{k}(V^y_y)^{k/2}L_k(\frac{l^y}{V^y_y})\right)=\delta^j_k.$$
\end{proof}
\subsection{The recurrent case}
\begin{prop}
$$\mu(l^x\in ds,l^x>0)=\frac{1}{s}\,ds.$$
\end{prop}
\begin{proof}
$$\mu(l^xe^{-pl^x})=\mu(L-M_{p\delta_x},l^x)=(V_{p\delta_x})^x_x=1/p.$$
Therefore, $\displaystyle{\mu(l^x\in ds,l^x>0)=\frac{1}{s}1_{\{s>0\}}\,ds}$.
\end{proof}

\begin{lem}\label{lem:one-to-one correspondance between the semi-group and the Markovian loop measure}
In the irreducible positive-recurrent case, there is a one-to-one correspondence between the semi-group of the Markov process and the Markovian loop measure.
\end{lem}
\begin{proof}
It is enough to show the loop measure determines the law of the Markov process.\\
Let $\pi$ be the invariant probability of the Markov process. Then it is positive everywhere.\\
Define a based loop functional $\phi_{t,x,y}^b$ as follows: for any based loop $l$ with length $|l|$, extend the function $(l(t),t\in[0,|l|])$ periodically, i.e. by setting $l(s+|l|)=l(s)$, and set:
$$\phi_{t,x,y}^b(l)=1_{\{|l|>t\}}\int\limits_{0}^{|l|}1_{\{l(s)=x\}}1_{\{l(s+t)=y\}}\,ds.$$
This rotation invariant functional defines a loop functional $\phi_{t,x,y}$ on the space of loops.
$$\frac{\mu(|l|\in du,\phi_{t,x,y})}{du}=\frac{1}{u}\sum\limits_{z\in S}\mathbb{P}^z_{u,z}[\int\limits_{0}^{|l|}1_{\{l(s)=x\}}1_{\{l(s+t)=y\}}\,ds]=(P_t)^x_y(P_{u-t})^y_x.$$
Taking $u$ tends to infinity, $$\lim\limits_{u\rightarrow\infty}\frac{\mu(|l|\in du,\phi_{t,x,y})}{du}=\pi_x(P_t)^x_y.$$
Since $\mu(l^xe^{-p|l|})=\mu(L-p,l^x)=(V_p)_x^x$, $\pi_x=\lim\limits_{p\rightarrow 0}p\mu(l^xe^{-p|l|})$. Finally, we are able to determine the semi-group $(P_t)^x_y$ for all $x,t,y$. Accordingly, the law of the Markov process is uniquely determined.
\end{proof}
\begin{rem}
From the argument above, we see that an irreducible positive-recurrent semi-group cannot have the same loop measure as another irreducible transient or null-recurrent semi-group.
\end{rem}
Finally, we can prove the following theorem.
\begin{thm}
In the irreducible recurrent case, there is a 1-1 correspondence between the semi-group of the Markov process and the Markovian loop measure.
\end{thm}
\begin{proof}
Given a minimal semi-group $(P_t, t\geq 0)$, we can always define the corresponding Markovian loop measure. It is left to show that we can recover the semi-group from the loop measure. Let the series of finite subset $F_1\subset F_2\subset\cdots\subset F_n\subset\cdots$ exhaust $S$. By Proposition \ref{compatibility with the time change}, we know that the measure of the trace of the Markovian loop on $F_i$ corresponds to the trace the Markov process on $F_n$. Since $|F_i|<\infty$, the trace of the Markov process on $F_i$ is an irreducible and positive-recurrent Markov process. Let $(P^{(n)}_t,t\geq 0)$ be its semi-group. By Lemma \ref{lem:one-to-one correspondance between the semi-group and the Markovian loop measure}, we can conclude that this trace of the Markov process is determined by the Markovian loop measure. Recall that $Y^{(n)}_t,t\geq 0$, the trace of the Markov process $X_t, t\geq 0$ on $F_n$ is defined as follows:\\
$A^{(n)}_t=\int\limits_{0}^{t}1_{\{X_s\in F_n\}}\,ds$, $\sigma^{(n)}_t$ is the right-continuous inverse of $A^{(n)}_t,t\geq 0$ and $Y^{(n)}_t=X_{\sigma^{(n)}_t},t\geq 0$.\\
As $n$ tends to infinity, $A^{(n)}_t$ increases to $t$ and $\sigma^{(n)}_t$ decreases to $t$. Since $X_t, t\geq 0$ is right-continuous, $\lim\limits_{n\rightarrow\infty}Y^{(n)}_t=X_t$. As a consequence, for any bounded $f$,  $P_tf(x)=\lim\limits_{n\rightarrow\infty}P^{(n)}_tf(x)$. Thus, we recover the semi-group $P_t$ as the limit.
\end{proof}
\section{Poisson process of loops}

In this section, we study the Poisson point processes naturally defined on the set of Markov loops (which are also known as ``loop soups"). We mostly focus on the associated occupation fields and on the partitions defined by loop clusters.
\subsection{Definitions and some basic properties}
\begin{defn}
We denote by $\mathcal{L}$ the Poisson point process on $\mathbb{R}_{+}\times${loops} with intensity $\text{Lebesgue}\otimes\mu$ and by
$\mathcal{L}_{\alpha}$ the Poisson random measure on the space of loops, $\mathcal{L}_{\alpha}(B)=\mathcal{L}([0,\alpha]\times B)$. Its intensity is $\alpha\mu$.
\end{defn}
The following proposition is taken from \cite{kingman}.
\begin{prop}\label{basic properties for Poisson point processes}
Let $\mathcal{P}$ be a Poisson random measure on $S$ with $\sigma$-finite intensity measure $\mu(dl)$.\\
a) Suppose that $\Phi$ is a measurable complex valued function, with $\mu(|\operatorname{Im}(\Phi)|\wedge 1)<\infty$ and $\mu(|e^{\Phi}-1|)<\infty$, then
$$\mathbb{E}[\exp({\sum\limits_{l\in\mathcal{P}}\Phi(l)})]=e^{\alpha\int(e^{\Phi(l)}-1)\mu(dl)}$$

b) The above equation holds if $\Phi$ is non-negative measurable without further assumptions.

c) Suppose $F_1, \cdots, F_k$ are non-negative functions, then the following `Campbell formula' holds
$$\mathbb{E}[\sum\limits_{l_1,\ldots,l_k\in\mathcal{P} \text{ distinct}}\prod\limits_{i=1}^{k} F_i(l_i)]=\prod\limits_{i=1}^{k}\mu(F_i)$$
d) Suppose that $S,T$ are two measurable spaces and $\phi:S\rightarrow T$ is a measurable mapping. Let $\mathcal{P}$ be a Poisson random measure on $S$ with intensity $\mu$. Then $\phi\circ\mathcal{P}$ is the Poisson random measure on $T$ with intensity $\phi\circ\mu$.
\end{prop}
\begin{proof}
See \cite{kingman}.
\end{proof}
From the expression of $\mu$ on trivial loops, we get the following:
\begin{prop}
Let $\mathcal{L}_{\alpha,\text{Trivial},x}=\{l\in\mathcal{L}_{\alpha}: l\text{ is a trivial loop at }x\}$. Then, $\{|l|:l\in\mathcal{L}_{\alpha,\text{Trivial},x}\}$ is a Poisson point measure on $\mathbb{R}_{+}$ with intensity $\frac{\alpha}{t}e^{L^x_x t}\,dt$.
\end{prop}

Recall that a Poisson-Dirichlet distribution has a representation by a Poisson point process, see section 9.4 in \cite{kingman}.

\begin{cor}
$$\{\frac{\ell^x}{\sum\limits_{l\in\mathcal{L}_{\alpha,\text{Trivial},x}}l^x};\ell\in\mathcal{L}_{\alpha,\text{Trivial},x}\}$$
 follows a Poisson-Dirichlet $(0,\alpha)$ distribution. Moreover, it is independent of $\sum\limits_{l\in\mathcal{L}_{\alpha,\text{Trivial},x}}l^x$ which follows the $\Gamma(\alpha,(-L^x_x)^{-1})$ distribution.
\end{cor}
Recall that the density of $\Gamma(\alpha,\beta)$ distribution is $\displaystyle{\frac{\beta^{-\alpha}}{\Gamma(\alpha)} x^{\alpha-1} e^{-\frac{x}{\beta}}}$.
\begin{proof}
By Proposition \ref{basic properties for Poisson point processes},
\begin{multline*}
\mathbb{E}[\exp({-\lambda\sum\limits_{l\in\mathcal{L}_{\alpha,Trivial,x}}l^x})]=\exp({\int\limits_{0}^{\infty}(e^{-\lambda t}-1)\frac{\alpha}{t}e^{L^x_x t}\,dt})\\
=\exp({\int\limits_{0}^{\infty}\int\limits_{0}^{\infty}\alpha(e^{-\lambda t}-1)e^{-ts}e^{L^x_x t}\,ds\,dt})\\
=\exp({\int\limits_{0}^{\infty}\frac{\alpha}{\lambda+s-L^x_x}-\frac{\alpha}{s-L^x_x}\,ds})=(\frac{-L^x_x}{\lambda-L^x_x})^{\alpha}
\end{multline*}
which is exactly the Laplace transform of the $\Gamma(\alpha, (-L^x_x)^{-1})$ distribution. Therefore, $\sum\limits_{l\in\mathcal{L}_{\alpha,Trivial,x}}l^x$ follows the $\Gamma(\alpha,(-L^x_x)^{-1})$ distribution.
\end{proof}
By taking the trace of the loops on $\{x\}$, we get a Poisson ensemble of Markov loops. To be more precise, we get a Poisson ensemble of trivial loops at $x$, but its intensity measure, (i.e. the loop measure), is associated with the generator $(L_{\{x\}})^x_x=-1/V^x_x$. As a consequence, we have the following proposition:
\begin{prop}
$\hat{\mathcal{L}}_{\alpha}^x=\sum\limits_{l\in\mathcal{L}_{\alpha}}l^x$ follows a $\Gamma(\alpha,V^x_x)$ distribution. $\displaystyle{\{\frac{l^x}{\hat{\mathcal{L}}_{\alpha}^x},l\in\mathcal{L}_{\alpha}\}}$ follows a Poisson-Dirichlet distribution $\Gamma(0,\alpha)$ which is independent of $\hat{\mathcal{L}}_{\alpha}^{x}$.
\end{prop}

\begin{defn}
Define $\hat{\mathcal{L}}_{\alpha}^x=\sum\limits_{l\in\mathcal{L}_{\alpha}}l^x$ and $\langle \hat{\mathcal{L}}_{\alpha},\chi \rangle=\sum\limits_{x\in S}\hat{\mathcal{L}}_{\alpha}^x \chi(x)$.
\end{defn}

\begin{prop}\label{laplace transform of the poissonian occupation field}
For any non-negative measurable $\chi$ on $S$,
$$\mathbb{E}[e^{-\langle\hat{\mathcal{L}}_{\alpha},\chi\rangle}]=(1+\sum\limits_{A\subset S, 0<|A|<\infty}\prod\limits_{x\in A}\chi(x)\det(V_{A}))^{-\alpha}.$$
For any non-negative finitely supported $\chi$ on $S$ and $z\in D=\{z\in\mathbb{C}:\operatorname{Re}(z)<\frac{1}{\rho(M_{\sqrt{\chi}}VM_{\sqrt{\chi}})}\}$,
$$\mathbb{E}[e^{z\langle\hat{\mathcal{L}}_{\alpha},\chi\rangle}]=(\det(I-zM_{\sqrt{\chi}}VM_{\sqrt{\chi}}))^{-\alpha}.$$
Outside of D, $\mathbb{E}[|e^{z\langle\hat{\mathcal{L}}_{\alpha},\chi\rangle}|]=\infty$.

\end{prop}

\begin{proof}
 It is a direct consequence of Proposition \ref{laplace transform of the occupation field}, Corollary \ref{positive laplace transform of the occupation field} and Proposition \ref{basic properties for Poisson point processes}.
\end{proof}

\begin{prop}
$\hat{\mathcal{L}}_{1}^{x}$ is exponentially distributed with parameter $1/V^x_x$.
\end{prop}
\begin{proof}
Since $\mathbb{E}[e^{-p\hat{\mathcal{L}}_1^x}]=\frac{1}{1+pV^x_x}$ and $\hat{\mathcal{L}}_1^x\geq 0$, $\hat{\mathcal{L}}_{1}^{x}$ is exponentially distributed with parameter $1/V^x_x$.
\end{proof}

\begin{rem}
$$\mathbb{E}((1-e^{-\frac{\hat{\mathcal{L}}_{\alpha}^{x}}{V^x_x}})^{-1})=\zeta(\alpha),\alpha>1.$$
\end{rem}
\begin{proof}
By Proposition \ref{laplace transform of the poissonian occupation field}
$$\mathbb{E}((1-e^{-\frac{\hat{\mathcal{L}}_{\alpha}^{x}}{V^x_x}})^{-1})=\sum\limits_{k=0}^{\infty}\mathbb{E}(e^{-\frac{k}{V^x_x}\hat{\mathcal{L}}_{\alpha}^x})
=\sum\limits_{k=1}^{\infty}k^{-\alpha}=\zeta(\alpha).$$
 \end{proof}

 \subsection{Moments and polynomials of the occupation field}
 \begin{defn}[\textbf{$\alpha$-permanent}]
  Denote by $m(\sigma)$ the number of cycles in the decomposition of the permutation $\sigma$. For any square matrix $A=(A^i_j,i,j=1,\ldots,n)$, define the $\alpha$-permanent of $A$ as $$\Per_{\alpha}(A)=\sum\limits_{\sigma\in S_n}\alpha^{m(\sigma)}A_{1\sigma(1)}\cdots A_{n\sigma(n)}.$$
 \end{defn}
 Note that $\Per_{-1}(A)=\det(-A)$.
 \begin{prop}\label{moments of the poissonian occupation fields}
 $$\mathbb{E}[\hat{\mathcal{L}}_{\alpha}^{x_1}\cdots\hat{\mathcal{L}}_{\alpha}^{x_n}]=\Per_{\alpha}((V^{x_l}_{x_m})_{1\leq m,l \leq n}).$$
 \end{prop}
\begin{proof}
Let $F_i(l)=l^{x_i}$. By Corollary \ref{moments of the 1-occupation field},
$$\mu(F_1\cdots F_k)=\mu(l^{x_1}\cdots l^{x_n})=\frac{1}{k}\sum\limits_{\sigma\in \mathfrak{S}_k}V_{x_{\sigma(2)}}^{x_{\sigma(1)}}\cdots V_{x_{\sigma(1)}}^{x_{\sigma(k)}}=\sum\limits_{\sigma\in \mathfrak{S}_k, m(\sigma)=1}V^{x_1}_{x_{\sigma(1)}}\cdots V^{x_k}_{x_{\sigma(k)}}.$$
Let $\mathcal{P}(\{1,\ldots,n\})$ be the collection of partitions of $\{1,\ldots,n\}$. For a partition $\pi$, we denote by $\#\pi$ the  number of blocks in  $\pi$, $\pi=(\pi_1,\ldots,\pi_{\#\pi})$.
\begin{eqnarray*}
\mathbb{E}[\hat{\mathcal{L}}_{\alpha}^{x_1}\cdots\hat{\mathcal{L}}_{\alpha}^{x_n}]&=&\mathbb{E}[(\sum\limits_{l\in\mathcal{L}_{\alpha}}l^{x_1})\cdots(\sum\limits_{l\in\mathcal{L}_{\alpha}}l^{x_n}) ]\\
&=&\sum\limits_{\pi\in\mathcal{P}(\{1,\ldots,n\})}\mathbb{E}[\sum\limits_{\begin{subarray}{l}l_{1},\ldots,l_{\#\pi}\in\mathcal{L}_{\alpha}\\ \text{distinct}\end{subarray}}\prod\limits_{i=1}^{\#\pi}(\prod\limits_{j\in\pi_i}F_j)(l_i)].
\end{eqnarray*}
Define $G^{\pi}_i=\prod\limits_{j\in \pi_i}F_j$ for $j=1,\ldots,\#\pi$. Then,
$$\prod\limits_{i=1}^{\#\pi}(\prod\limits_{j\in\pi_i}F_j)(l_i)=\prod\limits_{i=1}^{\#\pi}G_i(l_i).$$
By Campbell's formula,
$$\mathbb{E}[\sum\limits_{\begin{subarray}{l}l_{1},\ldots,l_{\#\pi}\in\mathcal{L}_{\alpha}\\ \text{distinct}\end{subarray}}\prod\limits_{i=1}^{\#\pi}(\prod\limits_{j\in\pi_i}F_j)(l_i)]=\alpha^{\#\pi}\prod\limits_{i=1}^{\#\pi}\mu(G_i).$$
Write $\pi_i$ in decreasing order $p(i,1)<\cdots<p(i,\#\pi_i)$, then $$\mu(G_i)=\sum\limits_{\sigma\in\mathfrak{S}_{\#\pi_i},m(\sigma)=1}\prod\limits_{j=1}^{\#\pi_i}V^{x_{p(i,j)}}_{x_{p(i,\sigma(j))}}=\sum\limits_{\begin{subarray}{l}\sigma:\text{ circular}\\
\text{permutation}\\
\text{on }\pi_i\end{subarray}}\prod\limits_{j\in\pi_i}V^{x_j}_{x_{\sigma(j)}}.$$
Clearly, there is a one-to-one correspondence between a permutation $\eta$ on $\{1,\ldots,n\}$ and an $m(\eta)$-partition $\pi=(\pi_1,\ldots,\pi_{m(\eta)})$ together with these circular permutation on the blocks of $\pi$.
Finally,
\begin{multline*}
\mathbb{E}[\hat{\mathcal{L}}_{\alpha}^{x_1}\cdots\hat{\mathcal{L}}_{\alpha}^{x_n}]=\sum\limits_{\pi\in\mathcal{P}(\{1,\ldots,n\})}\alpha^{\#\pi}\prod\limits_{i=1}^{\#\pi}\mu(G_i)\\
=\sum\limits_{\pi\in\mathcal{P}(\{1,\ldots,n\})}\alpha^{\#\pi}\prod\limits_{i=1}^{\#\pi}\sum\limits_{\begin{subarray}{l}\sigma:\text{ circular}\\
\text{permutation}\\
\text{on }\pi_i\end{subarray}}\prod\limits_{j\in\pi_i}V^{x_j}_{x_{\sigma(j)}}\\
=\sum\limits_{\eta\in\mathfrak{S}_n}\alpha^{m(\eta)}\prod\limits_{i=1}^{n}V^{x_i}_{x_{\eta(i)}}=\Per_{\alpha}(V^{x_i}_{x_j},1\leq i,j\leq n)
\end{multline*}
\end{proof}

\begin{defn}
$\mu(l^{x})=V^x_x$, define $\tilde{\mathcal{L}}_{\alpha}^x=\hat{\mathcal{L}}_{\alpha}^{x}-\alpha V^x_x$.
\end{defn}
Note that $\mathbb{E}[\tilde{\mathcal{L}}_{\alpha}^x]=0$.
\begin{defn}
For $A=(A_{ij})_{1\leq i,j\leq n}$, define
 $$\Per^0_{\alpha}(A)=\sum\limits_{\sigma\in S_n,\sigma(i)\neq i,i=1,\ldots,n}\alpha^{m(\sigma)}A_{1\sigma(1)}\cdots A_{n\sigma(n)}$$
with $m(\sigma)$ the number of cycles in $\sigma$.
 \end{defn}
 \begin{prop}
 $$\mathbb{E}[\tilde{\mathcal{L}}_{\alpha}^{x_1}\cdots\tilde{\mathcal{L}}_{\alpha}^{x_n}]=\Per^0_{\alpha}((V^{x_l}_{x_m})_{1\leq m,l \leq n}).$$
 \end{prop}
 \begin{proof}
 For $\sigma\in S_n$, let $n(k,\sigma)$ be the number of cycles of length $k$ in $\sigma$. According to Proposition \ref{moments of the poissonian occupation fields},
\begin{align*}
&\mathbb{E}[\tilde{\mathcal{L}}_{\alpha}^{x_1}\cdots\tilde{\mathcal{L}}_{\alpha}^{x_n}]=\mathbb{E}[\prod\limits_{i=1}^{n}(\hat{\mathcal{L}}_{\alpha}^{x_i}-\alpha V^{x_i}_{x_i})]=\sum\limits_{A\subset \{1,\ldots,n\}} (-\alpha)^{|A|}(\prod\limits_{j\in A}V^{x_j}_{x_j})\mathbb{E}[\prod\limits_{j\in A^c}\hat{\mathcal{L}}_{\alpha}^{x_j}]\\
&=\sum\limits_{A\subset \{1,\ldots,n\}} (-\alpha)^{|A|}(\prod\limits_{j\in A}V^{x_j}_{x_j})\Per_{\alpha}(V_{A^c})
\end{align*}
where $A^c=\{1,\ldots,n\} \setminus A$. The above quantity equals
\begin{multline*}
\sum\limits_{A\subset\{1,\ldots,n\}}(-1)^{|A|}\sum\limits_{\sigma\in \mathfrak{S}_n,\sigma|_{A}=Id}\alpha^{m(\sigma)}V^{x_1}_{x_{\sigma(1)}}\cdots V^{x_n}_{x_{\sigma(n)}}\\
=\sum\limits_{\sigma\in \mathfrak{S}_n}\alpha^{m(\sigma)}V^{x_1}_{x_{\sigma(1)}}\cdots V^{x_n}_{x_{\sigma(n)}}(\sum\limits_{A\subset\{1,\ldots,n\}}1_{\{\sigma|_{A}=Id\}}(-1)^{|A|})\\
=\sum\limits_{\sigma\in \mathfrak{S}_n}\alpha^{m(\sigma)}V^{x_1}_{x_{\sigma(1)}}\cdots V^{x_n}_{x_{\sigma(n)}}1_{\{n(1,\sigma)=0\}}=\Per^0_{\alpha}((V^{x_l}_{x_m})_{1\leq m,l \leq n}).
\end{multline*}
\end{proof}
It is well-known that the generalized Laguerre polynomials $(L_{k}^{\alpha-1},k\in\mathbb{N}),\alpha>0$ have the following  generating function
$$\frac{e^{-\frac{xt}{1-t}}}{(1-t)^{\alpha}}=\sum\limits_{k=0}^{\infty}t^kL^{\alpha-1}_{k}(x).$$
Moreover,
$$L_k^{\alpha}(x)=\sum\limits_{i=0}^{k}(-1)^i\binom{n+\alpha}{n-i}\frac{x^i}{i!}.$$
\begin{defn}
Define $\displaystyle{P_{k}^{\alpha,\sigma}(x)=(-\sigma)^{k}L_{k}^{\alpha-1}\left(\frac{x}{\sigma}\right)}$ and $Q_{k}^{\alpha,\sigma}(x)=P_{k}^{\alpha,\sigma}(x+\alpha\sigma)$.
\end{defn}
These polynomials of the occupation field are related to Wick renormalisation in the symmetric case, when $\alpha$ is a half integer (see \cite{loop}).
\begin{prop}\
\begin{itemize}
\item[a)]
$\displaystyle{\frac{e^{\frac{xt}{1+t\sigma}}}{(1+t\sigma)^{\alpha}}=\sum\limits_{k=0}^{\infty}t^kP^{\alpha-1,\sigma}_{k}(x)}$ and
$\displaystyle{\frac{e^{\frac{xt+t\alpha\sigma}{1+t\sigma}}}{(1+t\sigma)^{\alpha}}=\sum\limits_{k=0}^{\infty}t^kQ^{\alpha-1,\sigma}_{k}(x)}. $
\item[b)]
$\displaystyle{P_{k}^{\alpha,V^x_x}(\hat{\mathcal{L}}_{\alpha}^{x})=Q_{k}^{\alpha,V^x_x}(\tilde{\mathcal{L}}_{\alpha}^{x})}$.
\item[c)]
$\displaystyle{\mathbb{E}[P_{k}^{\alpha,V^x_x}(\hat{\mathcal{L}}_{\alpha}^{x})P_{l}^{\alpha,V^x_x}(\hat{\mathcal{L}}_{\alpha}^{x})]=\delta^k_l\binom{\alpha+k}{k}(V^x_yV^y_x)^{k}}.$
\end{itemize}
\end{prop}
\begin{proof}[Proof of c)]
By Proposition \ref{laplace transform of the poissonian occupation field}, for $|s|$ and $|t|$ small enough,
$$\mathbb{E}[\frac{e^{\frac{\hat{\mathcal{L}}_{\alpha}^{x}t}{1+tV^x_x}}}{(1+tV^x_x)^{\alpha}}\frac{e^{\frac{s\hat{\mathcal{L}}_{\alpha}^{y}}{1+sV^y_y}}}{(1+sV^y_y)^{\alpha}}]=(1-stV^x_yV^y_x)^{-\alpha}.$$
Since $L_k^{\alpha}(x)=\sum\limits_{i=0}^{k}(-1)^i\binom{n+\alpha}{n-i}\frac{x^i}{i!}$, $|L_k^{\alpha}(x)|\leq \sum\limits_{i=0}^{k}\binom{n+\alpha}{n-i}\frac{|x|^i}{i!}=L_k^{\alpha}(-|x|)$. Therefore,
\begin{align*}
&\sum\limits_{k,l\in\mathbb{N}}|t^kP_k^{\alpha,V^x_x}(\hat{\mathcal{L}}_{\alpha}^{x})s^lP_l^{\alpha,V^y_y}(\hat{\mathcal{L}}_{\alpha}^{y})|\leq \sum\limits_{k,l\in\mathbb{N}}(|t|V^x_x)^kL_k^{\alpha-1}(-\frac{\hat{\mathcal{L}}_{\alpha}^{x}}{V^x_x})(|s|V^y_y)^lL_l^{\alpha-1}(-\frac{\hat{\mathcal{L}}_{\alpha}^{y}}{V^x_x})\\
&=\frac{e^{\frac{\hat{\mathcal{L}}_{\alpha}^{x}|t|}{1-V^x_x|t|}}e^{\frac{\hat{\mathcal{L}}_{\alpha}^{y}|s|}{1-V^y_y|s|}}}{(1-|t|V^x_x)^{\alpha}(1-|s|V^y_y)^{\alpha}} \end{align*}
By Proposition \ref{laplace transform of the poissonian occupation field}, for $|s|$ and $|t|$ small enough, $\displaystyle{\mathbb{E}\left[\frac{e^{\frac{\hat{\mathcal{L}}_{\alpha}^{x}|t|}{1-V^x_x|t|}}e^{\frac{\hat{\mathcal{L}}_{\alpha}^{y}|s|}{1-V^y_y|s|}}}{(1-|t|V^x_x)^{\alpha}(1-|s|V^y_y)^{\alpha}}\right]<\infty}$. Consequently,
\begin{align*}
\sum\limits_{k,l\in\mathbb{N}}t^kP_k^{\alpha,V^x_x}(\hat{\mathcal{L}}_{\alpha}^{x})s^lP_l^{\alpha,V^y_y}(\hat{\mathcal{L}}_{\alpha}^{y})&=\mathbb{E}[\frac{e^{\frac{\hat{\mathcal{L}}_{\alpha}^{x}t}{1+tV^x_x}}}{(1+tV^x_x)^{\alpha}}\frac{e^{\frac{s\hat{\mathcal{L}}_{\alpha}^{y}}{1+sV^y_y}}}{(1+sV^y_y)^{\alpha}}]\\
&=(1-stV^x_yV^y_x)^{-\alpha}=\sum\limits_{k\in\mathbb{N}}\binom{\alpha+k}{k}(stV^x_yV^y_x)^k.
\end{align*}
Finally,  identifying the coefficients of $s^kt^l$, we obtain
$$\mathbb{E}[P_{k}^{\alpha,V^x_x}(\hat{\mathcal{L}}_{\alpha}^{x})P_{l}^{\alpha,V^x_x}(\hat{\mathcal{L}}_{\alpha}^{x})]=\delta^k_l\binom{\alpha+k}{k}(V^x_yV^y_x)^{k}.$$
\end{proof}
\begin{prop}
Fix some $p\geq 1$, for $|t|$ small enough, $\displaystyle{\alpha\rightarrow (1+tV^x_x)^{-\alpha}e^{\frac{\hat{\mathcal{L}}_{\alpha}^{x}t}{1+tV^x_x}}}$ and $\displaystyle{\alpha\rightarrow P_{k}^{\alpha,V^x_x}(\hat{\mathcal{L}}_{\alpha}^{x})}$ are continuous $L^p$-martingales indexed by $\alpha>0$.
\end{prop}

\subsection{Limit behavior of the occupation field}
 \begin{rem}
 $(X_{\alpha}=(\hat{\mathcal{L}}_{\alpha}^{x_1},\ldots, \hat{\mathcal{L}}_{\alpha}^{x_n}),\alpha\geq 0)$ is a multi-subordinator with respect to the increasing family of $\sigma-$fields $\mathcal{F}_{\alpha}=\sigma(\mathcal{L}_{s},s\leq \alpha)$.

  $\mathbb{E}[\hat{\mathcal{L}}_{1}^{x_1},\ldots,\hat{\mathcal{L}}_{1}^{x_n}]=(V^{x_1}_{x_1},\ldots,V^{x_n}_{x_n})$ and  $\displaystyle{\mathbb{E}[e^{-\lambda_1\hat{\mathcal{L}}_{\alpha}^{x_1}-\cdots-\lambda_n\hat{\mathcal{L}}_{\alpha}^{x_n}}]=e^{-\alpha\Phi(\lambda_1,\ldots,\lambda_n)}}$, where $$\Phi(\lambda_1,\ldots,\lambda_n)=\int\limits_{y^1,\ldots,y^n\in\mathbb{R}^{+}}(1-e^{-\sum\limits_{i=1}^{n}\lambda_iy^i})\mu(l^{x_1}\in dy^1,\ldots,l^{x_n}\in dy^n).$$

  So $\lim\limits_{\alpha\rightarrow\infty}\frac{1}{\alpha}(\hat{\mathcal{L}}_{\alpha}^{x_1},\ldots,\hat{\mathcal{L}}_{\alpha}^{x_n})=(V^{x_1}_{x_1},\ldots,V^{x_n}_{x_n})$. And $(\frac{\hat{\mathcal{L}}_{\alpha}^{x_1}-\alpha V^{x_1}_{x_1}}{\sqrt{\alpha}},\ldots,\frac{\hat{\mathcal{L}}_{\alpha}^{x_n}-\alpha V^{x_n}_{x_n}}{\sqrt{\alpha}}
)$ converges in law to a Gaussian variable with mean 0
 and covariance $(C_{ij}=V^{x_i}_{x_j}V^{x_j}_{x_i},i,j=1,\ldots,n)$.
\end{rem}

The following result comes from \cite{acosta}: the rescaled L\'{e}vy process $(\frac{1}{t}X(ts),s\geq 0),t>0$ verifies the strong large deviation principle with a good rate function as $t\rightarrow\infty$ under the exponential integrability condition:
$$\exists\beta>0,\mathbb{E}[e^{\beta||X(1)||}]<\infty$$
This is true for the subordinator $((\hat{\mathcal{L}}_{\alpha}^{x_1},\ldots,\hat{\mathcal{L}}_{\alpha}^{x_n}),\alpha>0)$ by Proposition \ref{laplace transform of the poissonian occupation field}. The proposition below follows by application of the contraction principle.
\begin{prop}
$\frac{1}{\alpha}(\hat{\mathcal{L}}_{\alpha}^{x_1},\ldots,\hat{\mathcal{L}}_{\alpha}^{x_n})\in\mathbb{R}^n$ verifies a strong large derivation principle with good rate function $\Lambda^{*}:\mathbb{R}^n\rightarrow[0,\infty]$ when $\alpha$ tends to $\infty$. Here, $\Lambda(u)=\ln\mathbb{E}[e^{\hat{\mathcal{L}}_{1}^{x_1}u_1+\cdots+\hat{\mathcal{L}}_{1}^{x_n}u_n}]$ and $\Lambda^{*}(y)=\sup\limits_{u\in\mathbb{R}^d}(\langle u,y\rangle-\Lambda(u))$.
\end{prop}
To be more precise, for all open set $O\subset\mathbb{R}^{n}$, $$\liminf\limits_{\alpha\rightarrow\infty}\frac{1}{\alpha}\ln(\mathbb{P}[\frac{1}{\alpha}(\hat{\mathcal{L}}_{\alpha}^{x_1},\ldots,\hat{\mathcal{L}}_{\alpha}^{x_n})\in O])\geq -\inf\limits_{y\in O}\Lambda^{*}(y)$$ and for all closed subset $C$ of $\mathbb{R}^{n}$, $$\limsup\limits_{\alpha\rightarrow\infty}\frac{1}{\alpha}\ln(\mathbb{P}[\frac{1}{\alpha}(\hat{\mathcal{L}}_{\alpha}^{x_1},\ldots,\hat{\mathcal{L}}_{\alpha}^{x_n})\in C])\leq -\inf\limits_{y\in C}\Lambda^{*}(y).$$

\begin{rem}
In particular, for $n=1$,
 $$\Lambda^{*}(y)=\left\{
 \begin{array}{ll}
 \ln(\frac{V^x_x}{y})-1+\frac{y}{V^x_x} & y>0\\
 \infty & y\leq 0
 \end{array}\right.$$
For $n=2$,
$$\Lambda^{*}(y)=\left\{
\begin{array}{ll}
\ln\Bigl(\frac{1+\sqrt{1+\frac{4y^1y^2V^{x_1}_{x_1}V^{x_2}_{x_2}}{\det(V|_{\{x_1,x_2\}})^2}}}{2y^1y^2}\Bigr)+\ln\det(V|_{\{x_1,x_2\}})& \\
 +\frac{y^1V^{x_2}_{x_2}+y^2V^{x_1}_{x_1}}{\det(V|_{\{x_1,x_2\}})}-1-\sqrt{1+\frac{4y^1y^2V^{x_1}_{x_1}V^{x_2}_{x_2}}{\det(V|_{\{x_1,x_2\}})^2}}  & y_1,y_2\geq 0\\
\infty & \text{otherwise.}
\end{array}\right.$$
\end{rem}

\begin{proof}[Proof of the remark]
For $n=1$, by Proposition \ref{laplace transform of the poissonian occupation field}, $\Lambda(u)=-\ln(1-uV^x_x)$ for $\displaystyle{u<\frac{1}{V^x_x}}$ and $\Lambda(u)=\infty$ otherwise. Then, $$\Lambda^{*}(y)=\sup\limits_{u\in\mathbb{R}}(uy-\Lambda(u))=\left\{
 \begin{array}{ll}
 \frac{y-V^x_x}{yV^x_x}y-\Lambda(\frac{y-V^x_x}{yV^x_x}) & y>0\\
 \infty & y\leq 0
 \end{array}\right.=\left\{
 \begin{array}{ll}
 \ln(\frac{V^x_x}{y})-1+\frac{y}{V^x_x} & y>0\\
 \infty & y\leq 0.
 \end{array}\right.$$
Denote by $A(u_1,u_2)$ the matrix $\displaystyle{\begin{bmatrix}
V^{x_1}_{x_1}u_1 & V^{x_1}_{x_2}\sqrt{u_1u_2}\\
V^{x_2}_{x_1}\sqrt{u_1u_2} & V^{x_2}_{x_2}u_2
\end{bmatrix}}$.\\
For $n=2$,  by Proposition \ref{laplace transform of the poissonian occupation field}, for $u_1\geq 0,u_2\geq 0$,
$$\mathbb{E}[e^{\hat{\mathcal{L}}_1^{x_1}u_1+\hat{\mathcal{L}}_1^{x_1}u_2}]=\left\{
\begin{array}{ll}
1/\det(I-A(u_1,u_2)) & \text{ if }1<1/\rho(A(u_1,u_2))\\
\infty & \text{ otherwise}
\end{array}\right.$$
where the spectral radius $$\rho(A(u_1,u_2))=\frac{V^{x_1}_{x_1}u_1+V^{x_2}_{x_2}u_2+\sqrt{(V^{x_1}_{x_1}u_1-V^{x_2}_{x_2}u_2)^2+4V^{x_1}_{x_2}V^{x_2}_{x_1}u_1u_2}}{2}$$ and $$1/\det(I-A(u_1,u_2))=\frac{1}{1-V^{x_1}_{x_1}u_1-V^{x_2}_{x_2}u_2+\det(V|_{\{x_1,x_2\}})u_1u_2}.$$
Finally, for $u_1,u_2\geq 0$, \\
if $1-V^{x_1}_{x_1}u_1-V^{x_2}_{x_2}u_2+\det(V|_{\{x_1,x_2\}})u_1u_2>0$ and $V^{x_1}_{x_1}u_1+V^{x_2}_{x_2}u_2<2$,
$$\Lambda(u_1,u_2)=-\ln(1-V^{x_1}_{x_1}u_1-V^{x_2}_{x_2}u_2+\det(V|_{\{x_1,x_2\}})u_1u_2)$$
and $\Lambda(u_1,u_2)=\infty$ otherwise.\\
For $u_1,u_2\leq 0$, $\Lambda(u_1,u_2)=-\ln(1-V^{x_1}_{x_1}u_1-V^{x_2}_{x_2}u_2+\det(V|_{x_1,x_2})u_1u_2)$.\\
For $u_1>0,u_2<0$,
$$\mathbb{E}[e^{\hat{\mathcal{L}}_1^{x_1}u_1+\hat{\mathcal{L}}_1^{x_1}u_2}]=\exp\left(\int(e^{u_1l^{x_1}}-1)e^{u_2l^{x_2}}\mu(dl)+\int(e^{u_2l^{x_2}}-1)\mu(dl)\right).$$
By Proposition \ref{laplace transform of the occupation field}$, \int(e^{u_2l^{x_2}}-1)\mu(dl)=-\ln\det(1-u^2V^{x_2}_{x_2})$. By Theorem \ref{compatibility with Feynman-Kac} and then Proposition \ref{laplace transform of the occupation field},
\begin{align*}
\int(e^{u_1l^{x_1}}-1)e^{u_2l^{x_2}}\mu(dl)=&\int(e^{u_1l^{x_1}}-1)\mu(L-M_{-u_2\delta_{x_2}},dl)\\
=&\left\{
\begin{array}{ll}
-\ln(1-u_1(V_{-u_2\delta_{x_2}})^{x_1}_{x_1}) & \text{ if }u_1<1/(V_{-u_2\delta_{x_2}})^{x_1}_{x_1}\\
\infty & \text{ otherwise.}
\end{array}
\right.
\end{align*}
By the resolvent equation,
$$V^{x_1}_{x_2}=(V_{-u_2\delta_{x_2}})^{x_1}_{x_2}+(V_{-u_2\delta_{x_2}})^{x_1}_{x_2}(-u_2)V^{x_2}_{x_2}.$$
Therefore, $(V_{-u_2\delta_{x_2}})^{x_1}_{x_2}=\frac{V^{x_1}_{x_2}}{1-u_2V^{x_2}_{x_2}}$. Again, by the resolvent equation,
$$V^{x_1}_{x_1}=(V_{-u_2\delta_{x_2}})^{x_1}_{x_1}+(V_{-u_2\delta_{x_2}})^{x_1}_{x_2}(-u_2)V^{x_2}_{x_1}.$$
We deduce that $(V_{-u_2\delta_{x_2}})^{x_1}_{x_1}=V^{x_1}_{x_1}+\frac{u_2V^{x_1}_{x_2}V^{x_2}_{x_1}}{1-u_2V^{x_2}_{x_2}}$. Therefore, for $u_1>0,u_2<0$,\\
if $1-V^{x_1}_{x_1}u_1-V^{x_2}_{x_2}u_2+\det(V|_{x_1,x_2})u_1u_2>0$, $$\Lambda(u_1,u_2)=-\ln(1-V^{x_1}_{x_1}u_1-V^{x_2}_{x_2}u_2+\det(V|_{x_1,x_2})u_1u_2)$$
and $\Lambda(u_1,u_2)=\infty$ otherwise. It is easy to check that $V^{x_1}_{x_1}u_1+V^{x_2}_{x_2}u_2<2$ is implied by $1-V^{x_1}_{x_1}u_1-V^{x_2}_{x_2}u_2+\det(V|_{x_1,x_2})u_1u_2>0$ for $u_1>0,u_2<0$.
Similar results can be proved for $u_1<0,u_2>0$. In the end, for any $u_1,u_2\in\mathbb{R}$,\\
if $1-V^{x_1}_{x_1}u_1-V^{x_2}_{x_2}u_2+\det(V|_{\{x_1,x_2\}})u_1u_2>0$ and $V^{x_1}_{x_1}u_1+V^{x_2}_{x_2}u_2<2$,
$$\Lambda(u_1,u_2)=-\ln(1-V^{x_1}_{x_1}u_1-V^{x_2}_{x_2}u_2+\det(V|_{\{x_1,x_2\}})u_1u_2)$$
and $\Lambda(u_1,u_2)=\infty$ otherwise.\\
It is obvious that $\Lambda^{*}(y_1,y_2)=\infty$ for $y_1\leq 0$ or $y_2\leq 0$. Fixing $y_1,y_2>0$,
we are able to solve $(\frac{\partial}{\partial u_1}\Lambda(u_1,u_2),\frac{\partial}{\partial u_2}\Lambda(u_1,u_2))=(y_1,y_2)$. We find that the extreme value of $\langle u,y\rangle-\Lambda(u)$ is reached for
\begin{align*}
&u_1=\frac{V^{x_2}_{x_2}}{V^{x_1}_{x_1}V^{x_2}_{x_2}-V^{x_1}_{x_2}V^{x_2}_{x_1}}-\frac{1+\sqrt{1+\frac{4y_1y_2V^{x_1}_{x_2}V^{x_2}_{x_1}}{(V^{x_1}_{x_1}V^{x_2}_{x_2}-V^{x_1}_{x_2}V^{x_2}_{x_1})^2}}}{2y_1}\\
&u_2=\frac{V^{x_1}_{x_1}}{V^{x_1}_{x_1}V^{x_2}_{x_2}-V^{x_1}_{x_2}V^{x_2}_{x_1}}-\frac{1+\sqrt{1+\frac{4y_1y_2V^{x_1}_{x_2}V^{x_2}_{x_1}}{(V^{x_1}_{x_1}V^{x_2}_{x_2}-V^{x_1}_{x_2}V^{x_2}_{x_1})^2}}}{2y_2}
\end{align*}
and then conclude that
\begin{multline*}
\Lambda^{*}(y)=\ln\left(\frac{1+\sqrt{1+\frac{4y^1y^2V^{x_1}_{x_1}V^{x_2}_{x_2}}{\det(V|_{\{x_1,x_2\}})^2}}}{2y^1y^2})+\ln\det(V|_{\{x_1,x_2\}}\right)\\
+\frac{y^1V^{x_2}_{x_2}+y^2V^{x_1}_{x_1}}{\det(V|_{\{x_1,x_2\}})}-1-\sqrt{1+\frac{4y^1y^2V^{x_1}_{x_1}V^{x_2}_{x_2}}{\det(V|_{\{x_1,x_2\}})^2}}.
\end{multline*}
\end{proof}

\subsection{Hitting probabilities}
\begin{defn}
For $D\subset S$, define $loop^{D}=\{l;\langle l,1_{\{S-D\}}\rangle=0\}$, namely loops contained in D. Let  $\mathcal{L}_{\alpha}^D=\mathcal{L}_{\alpha}\cap loop^D$ be the restriction of the Poisson ensemble on $loop^D$.
\end{defn}
Since $\mu(\{l; l\text{ is a trivial loop at }x\})=\infty$, $\mathcal{L}_{\alpha}$ contains infinitely many trivial loops at $x$ $\mu-a.s.$.
\begin{prop}
For a finite subset $F$,
$$\mathbb{P}[\nexists l\in \mathcal{L}_{\alpha}; \text{l is non-trivial and l visits F}]=(\prod\limits_{x\in F}(-L^x_x)\det(V_F))^{-\alpha}.$$
\end{prop}
\begin{proof}
It is a direct consequence of Proposition \ref{non-trivial loop visits F} and the definition of the Poisson random measure.
\end{proof}
\begin{rem}
For any subset $F$, we can find $F_1\subset\cdots\subset F_n\subset\cdots$ a sequence of finite subsets of $F$ increasing to $F$. Then,
\begin{multline*}
\mathbb{P}[\nexists l\in \mathcal{L}_{\alpha}; l\text{ is non-trivial and }l\text{ visits F}]\\
=\lim\limits_{n\rightarrow\infty}\downarrow\mathbb{P}[\nexists l\in \mathcal{L}_{\alpha}; l\text{ is non-trivial and }l\text{ visits }F_n]\\
=\lim\limits_{n\rightarrow\infty}\downarrow(\prod\limits_{x\in {F_n}}(-L^x_x)\det(V_{F_n}))^{-\alpha}\\
=\inf\limits_{A\subset F, |A|<\infty}(\prod\limits_{x\in A}(-L^x_x)\det(V_A))^{-\alpha}.
\end{multline*}
\end{rem}

\subsection{Densities of the occupation field}
A non-symmetric generalization of Dynkin's isomorphism was given in  \cite{nonsymdyn}.
Suppose $L$ is the generator of a transient sub-Markovian process on $\{x_1,\ldots,x_n\}$, $m$ is an excessive measure, and $\chi$ is a non-negative function on $\{x_1,\ldots,x_n\}$, then
\begin{equation*}
\frac{1}{(2\pi)^{n}}\int e^{-\frac{1}{2}\langle z\bar{z}, \chi\rangle_m}e^{\frac{1}{2}\langle Lz, z\rangle_m}\prod\,du^i\,dv^i=\det(-M_m L+M_{\chi m})^{-1}
\end{equation*}
where $z^j=u^j+\sqrt{-1}\cdot v^j$ for $j=1,\ldots,n$ and $L^i_j=L^{x_i}_{x_j}$ for $i,j=1,\ldots,n$.
And it has been proved that  if $ \supp(\chi)\subset F$, then $\mathbb{E}[e^{-\langle\hat{\mathcal{L}}_{1},\chi \rangle}]=\frac{\det(V_F)_{\chi}}{\det(V_F)}=\frac{\det(-L_F)}{\det(-L_F+\chi)}$.
So, we have the following representation.
\begin{prop}
Let $F=\{x_1,\ldots,x_n\}\subset S$ and $L_F=(-V_F)^{-1}$. For any bounded measurable function $G$,
\begin{equation*}
\mathbb{E}[G(\hat{\mathcal{L}}_{1}^{x_1},\ldots,\hat{\mathcal{L}}_{1}^{x_n})]=\frac{\det(-M_m L_F)}{(2\pi)^{n}}\int G(\frac{1}{2}m_1|z^1|^2,\ldots,\frac{1}{2}m_n|z^n|^2)e^{\frac{1}{2}\langle L_Fz, z\rangle_m}\prod\,du^i\,dv^i.
\end{equation*}
\end{prop}
\begin{rem}
Recall that in the symmetric case, if $\phi$ is a Gaussian free field with covariance matrix given by the Green function, $\hat{\mathcal{L}}_{1/2}$ has the same law as $\frac{1}{2}\phi^2$. Moreover, if $\phi_1,\ldots,\phi_k$ are $k$ i.i.d. copies of $\phi$, then $\frac{1}{2}\sum\limits_{j=1}^{k}\phi_k^2$ and $\hat{\mathcal{L}}_{k/2}$ have the same law. For details, see Chapter 5 in \cite{loop} and Chapter 4 in \cite{Sznitman}.
\end{rem}

We can derive from this expression a formula for the joint densities of the occupation field, for $\alpha=1$.
\begin{prop}
Let $F=\{x_1,\ldots,x_n\}\subset S$ and $L_F=(-V_F)^{-1}$.
Then, $f(\rho^1,\ldots,\rho^n)$, the density of $(\hat{\mathcal{L}}_{1}^{x^1},\ldots,\hat{\mathcal{L}}_{1}^{x^n})$ with respect to the Lebesgue measure on $\mathbb{R}_{+}^{n}$ is
$$\det(-L_F)\sum\limits_{n_{ij}\in \mathbb{N},i,j=1,\ldots,n}1_{\{\sum\limits_{j}n_{ij}=\sum\limits_{j}n_{ji},\,i=1,\ldots,n\}} \left(\prod\limits_{i,j=1,\ldots,n} \frac{( (L_F)^{x_i}_{x_j}\sqrt{\rho^i\rho^j})^{n_{ij}}}{n_{ij}!}\right).$$
\end{prop}
\begin{proof}
The above Proposition shows that for any $G$ bounded measurable
\begin{equation*}
\mathbb{E}[G(\hat{\mathcal{L}}_{1}^{x_1},\ldots,\hat{\mathcal{L}}_{1}^{x_n})]=\frac{\det(-M_m L_F)}{(2\pi)^{n}}\int G(\frac{1}{2}m_1|z^1|^2,\ldots,\frac{1}{2}m_n|z^n|^2)e^{\frac{1}{2}\langle L_Fz, z\rangle_m}\prod\,du^i\,dv^i.
\end{equation*}
Using the polar coordinate, let $r^j=|z^j|$, $\theta_j\in[0,2\pi[$, $u^j=r^j\cos(\theta_j)$ and $v^j=r^j\sin(\theta_j)$. Then, $\mathbb{E}[G(\hat{\mathcal{L}}_{1}^{x_1},\ldots,\hat{\mathcal{L}}_{1}^{x_n})]$ equals
$$\frac{\det(-M_m L_F)}{(2\pi)^{n}}\int G(\frac{1}{2}m_1(r^1)^2,\ldots,\frac{1}{2}m_n(r^n)^2)e^{\sum\limits_{i,j}\frac{1}{2}(L_F)^{x_i}_{x_j} r^i r^jm^je^{i\theta_i}e^{-i\theta_j}}\prod\,dr^i\,d\theta_i.$$
Let $\rho^j=m_j(r^j)^2/2$ for $j=1,\ldots,n$, then $\mathbb{E}[G(\hat{\mathcal{L}}_{1}^{x_1},\ldots,\hat{\mathcal{L}}_{1}^{x_n})]$ equals
$$\frac{\det(-L_F)}{(2\pi)^{n}}\int G(\rho^1,\ldots,\rho^n)e^{\sum\limits_{i,j} (L_F)^{x_i}_{x_j}\left(\rho^i\rho^j \frac{m_{i}}{m_{j}}\right)^{1/2}e^{i\theta_i}e^{-i\theta_j}}\prod\,d\rho^i\,d\theta_i.$$
Therefore, the density of $\hat{\mathcal{L}}_{1}^{x_1},\ldots,\hat{\mathcal{L}}_{1}^{x_n}$, is
\begin{align*}
f(\rho^1,\ldots,\rho^n)=&\int\limits_{[0,2\pi]^n}\frac{\det(-L_F)}{(2\pi)^{n}}e^{\sum\limits_{i,j} (L_F)^{x_i}_{x_j}\left(\rho^i\rho^j \frac{m_{i}}{m_{j}}\right)^{1/2}e^{i\theta_i}e^{-i\theta_j}}\prod\,d\theta_i\\
=&\int\limits_{|z^i|=1,i=1,\ldots,n}\frac{\det(-L_F)}{(2\pi i)^{n}}\frac{1}{z^1\cdots z^n}e^{\sum\limits_{i,j} (L_F)^{x_i}_{x_j}\left(\rho^i\rho^j \frac{m_{i}}{m_{j}}\right)^{1/2}\frac{z^i}{z^j}}\,dz^1\cdots\,dz^n.
\end{align*}
We expand $\displaystyle{\exp\left(\sum\limits_{i,j} (L_F)^{x_i}_{x_j}\left(\rho^i\rho^j \frac{m_{i}}{m_{j}}\right)^{1/2}\frac{z^i}{z^j}\right)}$ into series, integrate it term by term and use Cauchy's formula. Only the constant terms in the expansion of $$\exp\left(\sum\limits_{i,j} (L_F)^{x_i}_{x_j}\left(\rho^i\rho^j \frac{m_{i}}{m_{j}}\right)^{1/2}\frac{z^i}{z^j}\right)$$ contribute. Accordingly, we have
$$
f(\rho^1,\ldots,\rho^n)=\det(-L_F)\sum\limits_{
\begin{subarray}{l}
n_{ij}\in\mathbb{N}\text{ for }i,j=1,\ldots,n\\
\sum\limits_{j}n_{ij}=\sum\limits_{j}n_{ji}\text{ for }i=1,\ldots,n
\end{subarray}}\left(\prod\limits_{i,j=1,\ldots,n} \frac{( (L_F)^{x_i}_{x_j}\sqrt{\rho^i\rho^j})^{n_{ij}}}{n_{ij}!}\right).
$$
\end{proof}

Moreover, we have the follow expansions of the density of occupation field for general $\alpha>0$:
\begin{prop}
Denote by $\Coeff\left(\det(M_s+V_F)^{-\alpha},s_1^{M_1}\cdots s_n^{M_n}\right)$\footnote{For $s$ sufficient close to $(0,\ldots,0)$, $\det(M_s+V_F)^{-\alpha}$ is an analytic function.} the coefficient before the term $s_1^{M_1}\cdots s_n^{M_n}$ in the expansion of the function $s\rightarrow \det(M_s+V_F)^{-\alpha}$ for $s$ small enough. Then the density $(f^{\alpha}(\rho_1,\ldots,\rho_n),\rho_1,\ldots,\rho_n>0)$ of the occupation field $(\mathcal{L}_{\alpha}^{x_1},\ldots,\mathcal{L}_{\alpha}^{x_n})$ has the following expression:
$$f^{\alpha}(\rho_1,\ldots,\rho_n)=\sum\limits_{M_1,\ldots,M_n\in\mathbb{N}}\frac{\Coeff\left(\det(M_s+V_F)^{-\alpha},s_1^{M_1}\cdots s_n^{M_n}\right)}{\prod\limits_{i=1}^{n}\Gamma(M_i+\alpha)}\prod\limits_{i=1}^{n}\rho_{i}^{M_i+\alpha-1}.$$
\end{prop}
\begin{proof}
Let's calculate the Laplace transform of the function
$$(\rho_1,\ldots,\rho_n)\rightarrow f^{\alpha}(\rho_1,\ldots,\rho_n)e^{-c(\rho_1+\cdots+\rho_n)}.$$
For $c$ sufficient large, we have
$$\sum\limits_{M_1,\ldots,M_n\in\mathbb{N}}\frac{|\Coeff\left(\det(M_s+V_F)^{-\alpha},s_1^{M_1}\cdots s_n^{M_n}\right)|}{\prod\limits_{i=1}^{n}\Gamma(M_i+\alpha)}\prod\limits_{i=1}^{n}\rho_{i}^{M_i+\alpha-1}e^{-c(\rho_1+\cdots+\rho_n)}<\infty$$
\begin{multline*}
\int d\rho_1\cdots d\rho_n e^{-(\rho_1\lambda_1+\cdots+\rho_n\lambda_n)}f^{\alpha}(\rho_1,\ldots,\rho_n)e^{-c(\rho_1+\cdots+\rho_n)}\\
=\sum\limits_{M_1,\ldots,M_n\in\mathbb{N}}\frac{\Coeff\left(\det(M_s+V_F)^{-\alpha},s_1^{M_1}\cdots s_n^{M_n}\right)}{\prod\limits_{i=1}^{n}\Gamma(M_i+\alpha)}\int \prod\limits_{i=1}^{n}\rho_i^{M_i+\alpha-1}e^{-\rho_i(\lambda_i+c)}d\rho_i\\
=\sum\limits_{M_1,\ldots,M_n\in\mathbb{N}}\frac{\Coeff\left(\det(M_s+V_F)^{-\alpha},s_1^{M_1}\cdots s_n^{M_n}\right)}{\prod\limits_{i=1}^{n}(\lambda_i+c)^{M_i+\alpha}}\\
=\det(M^{-1}_{\lambda+c}+V_F)^{-\alpha}\prod\limits_{i=1}^{n}(\lambda_i+c)^{-\alpha}\\
=\det(I+M_{\lambda+c}V_F)^{-\alpha}=\mathbb{E}[e^{-\sum\limits_{i=1}^{n}\hat{\mathcal{L}_{\alpha}^{x_i}}(\lambda_i+c)}].
\end{multline*}
Clearly, $f^{\alpha}$ is the density of $(\mathcal{L}_{\alpha}^{x_1},\ldots,\mathcal{L}_{\alpha}^{x_n})$.
\end{proof}


\subsection{Conditioned occupation field}
\begin{defn}\label{defn:trace of a Poisson loop soup}
For $F\subset S$, define $\mathcal{L}_{\alpha}|_{F}=\{l_F:l\in\mathcal{L}_{\alpha}\}$ where $l_F$ is the trace of $l$ on $F$, see Definition \ref{defn:trace of a loop on F}.
\end{defn}

\begin{prop}\label{conditional distribution of Poisson random measure}
Let $X,Y$ be two Borel spaces. Let $\mathcal{P}$ be a Poisson random measure on $Z=X\times Y$ with $\sigma-$finite intensity measure $\mu(dx,dy)=m(dx)K(x,dy)$, $K$ being a probability kernel. Let $\pi_X$ and $\pi_Y$ be the projection from $Z=X\times Y$ to $X$ and $Y$ respectively.
Define $\mathcal{P}_{X}=\pi_X\circ\mathcal{P}$ and $\mathcal{P}_{Y}=\pi_Y\circ\mathcal{P}$. For all $\Phi:Y\rightarrow\mathbb{R}$ non-negative measurable, define $\phi:Y\rightarrow\mathbb{R}$ according to $\Phi$ by the following equation
$e^{-\phi(x)}=\int\limits_{Y}e^{-\Phi(y)}K(x,dy)$. Then,
$$\mathbb{E}[e^{-\langle\mathcal{P}_Y,\Phi\rangle}|\mathcal{F}_{X}]=e^{-\langle\mathcal{P}_X,\phi\rangle}.$$
\end{prop}
\begin{rem}
The Poisson random measure $\mathcal{P}$ is the $K$-randomization\footnote{Please refer to Chapter 12 of \cite{kallenberg}} of the Poisson random measure $\pi_X\circ\mathcal{P}$.
\end{rem}
\begin{proof}
Take $\Psi:X\rightarrow\mathbb{R}$ and $\Phi:Y\rightarrow\mathbb{R}$ non-negative measurable. Define $\phi$ by the following equation:
$$e^{-\phi(x)}=\int\limits_{Y}e^{-\Phi(y)}K(x,dy).$$
We have \begin{align*}
\mathbb{E}[e^{-\langle\mathcal{P}_X,\Psi\rangle}e^{-\langle\mathcal{P}_{Y},\Phi\rangle}]&=\mathbb{E}[e^{-\langle\mathcal{P},\Psi\otimes\Phi\rangle}]=e^{\mu(e^{-\Psi\otimes\Phi}-1)}\\
&=\exp({\int\limits_{X\times Y} (e^{-\Psi(x)}e^{-\Phi(y)}-1)m(dx)K(x,dy)})\\
&=\exp({\int\limits_{X}(e^{-\Psi(x)}\int\limits_{Y}e^{-\Phi(y)}K(x,dy)-1)m(dx)})\\
&=\exp({\int\limits_{X}(e^{-\Psi(x)-\phi(x)}-1)m(dx)})=\mathbb{E}[e^{-\langle\mathcal{P}_X,\Psi\rangle}e^{-\langle\mathcal{P}_Y,\phi\rangle}].
\end{align*}
Since $\mathcal{F}_{X}=\sigma(\{e^{-\langle\mathcal{P}_X,\Psi\rangle}:\Psi\text{ is a non-negative measurable function on }X\})$,
$$\mathbb{E}[e^{-\langle\mathcal{P}_Y,\Phi\rangle}|\mathcal{F}_{X}]=e^{-\langle\mathcal{P}_X,\phi\rangle}.$$
\end{proof}

Let $f$ be a positive measurable function on the space of excursions. Recall that $\mathcal{E}_{F}(l)$ is the point measure of the excursions of the loop $l$ outside of $F$ (see Definition \ref{defn:decomposition of a loop into a pre-trace and some excursions indexed by the edges of the pre-trace}).
As a consequence of Proposition \ref{conditional distribution of Poisson random measure} and Proposition \ref{joint measure of the trace and the excursions} or Corollary \ref{transition kernel from the trace to the point measure of excursions}, we have the following proposition.
\begin{prop}\label{conditional excursion random measure distribution}
$$\mathbb{E}[e^{-\sum\limits_{l\in\mathcal{L_{\alpha}}}\langle \mathcal{E}_F(l),f\rangle}|\sigma(\mathcal{L}_{\alpha}|_{F})]=(\prod\limits_{x\neq y\in F}\nu^{x,y}_{F,ex}(e^{-f})^{N^x_y(\mathcal{L}_{\alpha}|_{F})})\times e^{\sum\limits_{x\in F}(L^x_x-(L_F)^x_x)\widehat{(\mathcal{L}_{\alpha}|_{F})}^x\nu^{x,x}_{F,ex}(1-e^{-f})}.$$
\end{prop}

For an excursion $(e,x,y)$ outside of $F$ from $x$ to $y$ and $\chi$ any non-negative measurable function on $S$, set $\langle e,\chi\rangle=\int \chi(e(s))\,ds$. Then we have the following:
\begin{prop}
The conditional expectation $\mathbb{E}[e^{-\langle\hat{\mathcal{L}}_{\alpha},\chi\rangle}|\sigma(\mathcal{L}_{\alpha}|_{F})]$ equals
$$\mathbb{E}[e^{-\langle\widehat{\mathcal{L}_{\alpha}^{F^c}},\chi\rangle}]e^{-\langle\hat{\mathcal{L}}_{\alpha}|_F,\chi\rangle}\exp\left(\sum\limits_{x\in F}L^x_x\widehat{(\mathcal{L}_{\alpha}|_{F})}^x((R^{F})^x_x-(R^F_{\chi})^x_x)\right)\prod\limits_{x\neq y\in F}\left(\frac{(R^F_{\chi})^x_y}{(R^{F})^x_y}\right)^{N^x_y(\mathcal{L}_{\alpha}|_{F})}.$$
\end{prop}
\begin{proof}
The set of loops which do not intersect $F$, $\mathcal{L}_{\alpha}^{F^c}$,  is independent of the set of loops which intersect $F$. Therefore,
\begin{align*}
\mathbb{E}[e^{-\langle\hat{\mathcal{L}}_{\alpha},\chi\rangle}|\sigma(\mathcal{L}_{\alpha}|_{F})]=&\mathbb{E}[e^{-\langle\widehat{\mathcal{L}_{\alpha}^{F^c}},\chi\rangle}]\mathbb{E}[\exp({-\sum\limits_{\begin{subarray}{l}l\in\mathcal{L}_{\alpha}\\ l\text{ visits }F\end{subarray}}\langle l,\chi\rangle})|\sigma(\mathcal{L}_{\alpha}|_{F})]\\
=&\mathbb{E}[e^{-\langle\widehat{\mathcal{L}_{\alpha}^{F^c}},\chi\rangle}]\mathbb{E}[\exp({-\sum\limits_{\begin{subarray}{l}l\in\mathcal{L}_{\alpha}\\ l\text{ visits }F\end{subarray}}(\langle l_F,\chi\rangle+\sum\limits_{e\in\mathcal{E}_F(l)}\langle e,\chi\rangle)})|\sigma(\mathcal{L}_{\alpha}|_{F})]\\
=&\mathbb{E}[e^{-\langle\widehat{\mathcal{L}_{\alpha}^{F^c}},\chi\rangle}]\mathbb{E}[\exp(-\sum\limits_{\begin{subarray}{l}l\in\mathcal{L}_{\alpha}\\ l\text{ visits }F\end{subarray}}\sum\limits_{e\in\mathcal{E}_F(l)}\langle e,\chi\rangle)\exp(-\langle\hat{\mathcal{L}}_{\alpha}|_F,\chi\rangle)|\sigma(\mathcal{L}_{\alpha}|_{F})]\\
=&\mathbb{E}[e^{-\langle\widehat{\mathcal{L}_{\alpha}^{F^c}},\chi\rangle}]e^{-\langle\hat{\mathcal{L}}_{\alpha}|_F,\chi\rangle}\mathbb{E}[\exp({-\sum\limits_{\begin{subarray}{l}l\in\mathcal{L}_{\alpha}\\ l\text{ visits }F\end{subarray}}\sum\limits_{e\in\mathcal{E}_F(l)}\langle e,\chi\rangle})|\sigma(\mathcal{L}_{\alpha}|_{F})]\\
=&\mathbb{E}[e^{-\langle\widehat{\mathcal{L}_{\alpha}^{F^c}},\chi\rangle}]e^{-\langle\hat{\mathcal{L}}_{\alpha}|_F,\chi\rangle}\mathbb{E}[\exp({-\sum\limits_{l\in\mathcal{L}_{\alpha}}\sum\limits_{e\in\mathcal{E}_F(l)}\langle e,\chi\rangle}|\sigma(\mathcal{L}_{\alpha})|_{F})].
\end{align*}
By Proposition \ref{conditional excursion random measure distribution}, taking the positive excursion function $f(\cdot)$ to be $\langle\cdot,\chi\rangle$,
\begin{multline*}
\mathbb{E}[e^{-\sum\limits_{l\in\mathcal{L}_{\alpha}}\sum\limits_{e\in\mathcal{E}_F(l)}\langle e,\chi\rangle}|\sigma(\mathcal{L}_{\alpha}|_{F})]=\mathbb{E}[e^{-\sum\limits_{l\in\mathcal{L}_{\alpha}}\langle \mathcal{E}_F(l),\langle\cdot,\chi\rangle\rangle}|\sigma(\mathcal{L}_{\alpha}|_{F})]\\
=(\prod\limits_{x\neq y\in F}\nu^{x,y}_{F,ex}(e^{-\langle\cdot,\chi\rangle})^{N^x_y(\mathcal{L}_{\alpha}|_{F})})\exp\left(\sum\limits_{x\in F}(L^x_x-(L_F)^x_x)\widehat{(\mathcal{L}_{\alpha}|_{F})}^x\nu^{x,x}_{F,ex}(1-e^{-\langle\cdot,\chi\rangle})\right).
\end{multline*}
By Lemma \ref{lem:laplace transform for the excursion measure},
$$\nu^{x,y}_{F,ex}(e^{-\langle\cdot,\chi\rangle})=\frac{(R^F_{\chi})^x_y}{(R^{F})^x_y}.$$
Then, by Proposition \ref{Expression of the generator for the time changed process}, $L^x_x-(L_F)^x_x=L^x_x(R^{F})^x_x$. Then,
\begin{multline*}
\mathbb{E}[e^{-\sum\limits_{l\in\mathcal{L}_{\alpha}}\sum\limits_{e\in\mathcal{E}_F(l)}\langle e,\chi\rangle}|\sigma(\mathcal{L}_{\alpha}|_{F})]\\
=\left(\prod\limits_{x\neq y\in F}\left(\frac{(R^F_{\chi})^x_y}{(R^{F})^x_y}\right)^{N^x_y(\mathcal{L}_{\alpha}|_{F})}\right)\exp\left(\sum\limits_{x\in F}(L^x_x-(L_F)^x_x)\widehat{(\mathcal{L}_{\alpha}|_{F})}^x\left(1-\frac{(R^F_{\chi})^x_x}{(R^{F})^x_x}\right)\right)\\
=\left(\prod\limits_{x\neq y\in F}\left(\frac{(R^F_{\chi})^x_y}{(R^{F})^x_y}\right)^{N^x_y(\mathcal{L}_{\alpha}|_{F})}\right)\exp\left(\sum\limits_{x\in F}L^x_x\widehat{(\mathcal{L}_{\alpha}|_{F})}^x((R^{F})^x_x-(R^F_{\chi})^x_x)\right).
\end{multline*}
Finally, we get $\mathbb{E}[e^{-\langle\hat{\mathcal{L}}_{\alpha},\chi\rangle}|\sigma(\mathcal{L}_{\alpha}|_{F})]$ equals
$$\mathbb{E}[e^{-\langle\widehat{\mathcal{L}_{\alpha}^{F^c}},\chi\rangle}]e^{-\langle\hat{\mathcal{L}}_{\alpha}|_F,\chi\rangle}\left(\prod\limits_{x\neq y\in F}\left(\frac{(R^F_{\chi})^x_y}{(R^{F})^x_y}\right)^{N^x_y(\mathcal{L}_{\alpha}|_{F})}\right)\exp\left(\sum\limits_{x\in F}L^x_x\widehat{(\mathcal{L}_{\alpha}|_{F})}^x((R^{F})^x_x-(R^F_{\chi})^x_x)\right).$$
\end{proof}
\subsection{Loop clusters}
Consider the space $S$ as a graph $(S,E)$ with $S$ as the set of vertices and $E=\{\{x,y\}:N^x_y(l)>0\text{ or }N^y_x(l)>0\}$ as the set of undirected edges. An edge $\{x,y\}$ is said to be open at time $\alpha$ if it is traversed by at least one loop of $\mathcal{L}_{\alpha}$, i.e. $N^x_y(\mathcal{L}_{\alpha})+N^y_x(\mathcal{L}_{\alpha})>0$. The set
of open edges defines a subgraph $G_{\alpha}$ with vertices $S$. The connected components of $G_{\alpha}$ define
a partition of $S$ denoted by $\mathcal{C}_{\alpha}$, namely the loop clusters at time $\alpha$. 

As in section 2 of \cite{sophie}, we have the following proposition,
\begin{prop}
Given a collection of edges $F=\{e_1=\{x_1,y_1\},\ldots,e_k=\{x_k,y_k\}\}$, let $A=\bigcup\limits_{i=1}^{k}\{x_i,y_i\}$. Then,
$$\mathbb{P}[e_1,\ldots,e_k\text{ are all closed}]=\det(I+(L|_{F})|_{A\times A}V_{A})^{-\alpha}$$
where $(L|_{F})^x_y=\left\{
\begin{array}{ll}
L^x_y & \text{if }\{x,y\}\in F\\
0 & \text{otherwise.}
\end{array}\right.$
\end{prop}
\begin{proof}
Suppose $S$ is finite,
\begin{multline*}
\mathbb{P}[e_1,\ldots,e_k\text{ are all closed}]=\exp(-\alpha\mu(\sum\limits_{i=1}^{k}N^{x_i}_{y_i}(l)+N^{y_i}_{x_i}(l)>0))\\
=\exp(-\alpha\mu(\sum\limits_{i=1}^{k}N^{x_i}_{y_i}(l)+N^{y_i}_{x_i}(l)>0,\;l\text{ is non-trivial}))\\
=\exp(-\alpha\mu(l\text{ is non-trivial})+\alpha\mu(\sum\limits_{i=1}^{k}N^{x_i}_{y_i}(l)+N^{y_i}_{x_i}(l)=0,\;l\text{ is non-trivial}))
\end{multline*}
Define $(L')^x_y=L^x_y$ if $\{x,y\}\notin F$ and $(L')^x_y=0$ if $\{x,y\}\in F$. By Proposition \ref{expression of mup},
$$\mu(\sum\limits_{i=1}^{k}N^{x_i}_{y_i}(l)+N^{y_i}_{x_i}(l)=0,l\text{ is non-trivial})=\mu(L',l\text{ is non-trivial}).$$
(Recall that $\mu(L',dl)$ is the Markovian loop measure associated with the generator $L'$.) By Proposition \ref{non-trivial loop visits F}, $\mu(L',l\text{ is non-trivial})=-\ln(\prod\limits_{x\in S}(-L')^x_x)+\ln\det(-L')=\ln(\prod\limits_{x\in S}(-L)^x_x)-\ln\det(-L')$ and $\mu(l\text{ is non-trivial})=\ln(\prod\limits_{x\in S}(-L)^x_x)-\ln\det(-L)$. Therefore,
$$\mathbb{P}[e_1,\ldots,e_k\text{ are all closed}]=\left(\frac{\det(-L)}{\det(-L')}\right)^{\alpha}=\det(-L'V)^{-\alpha}.$$
Write as $-L'=-L+(L-L')=-L+L|_{F}$. Therefore, $\det(-L'V)=\det(I+(L|_{F})|_{A\times A}V_{A})$. Consequently,
$$\mathbb{P}[e_1,\ldots,e_k\text{ are all closed}]=\det(I+(L|_{F})|_{A\times A}V_{A})^{-\alpha}.$$
For $S$ countable, let $A_1\subset A_2\subset\cdots$ exhausting $S$. Then we have
\begin{align*}
&\mathbb{P}[e_1,\ldots,e_k\text{ are all closed}]=\exp(-\alpha\mu(\sum\limits_{i=1}^{k}N^{x_i}_{y_i}(l)+N^{y_i}_{x_i}(l)>0))\\
=&\exp(-\alpha\lim\limits_{n\rightarrow\infty}\mu(\sum\limits_{i=1}^{k}N^{x_i}_{y_i}(l)+N^{y_i}_{x_i}(l)>0,l\text{ is contained in }A_n)).
\end{align*}
By Proposition \ref{compatibility with the time change},
$$\mu(\sum\limits_{i=1}^{k}N^{x_i}_{y_i}(l)+N^{y_i}_{x_i}(l)>0,l\text{ is contained in }A_n)=\mu(L|_{A_n\times A_n}, \sum\limits_{i=1}^{k}N^{x_i}_{y_i}(l)+N^{y_i}_{x_i}(l)>0).$$
By the calculation for the finite case,
$$\mu(L_{A_n\times A_n}, \sum\limits_{i=1}^{k}N^{x_i}_{y_i}(l)+N^{y_i}_{x_i}(l)>0)=\det(I+(L|_{F})|_{A\times A}(-L|_{A_n\times A_n})^{-1}_{A})^{-\alpha}.$$
It is not hard to check that $\lim\limits_{n\rightarrow\infty}((-L|_{A_n\times A_n})^{-1})^x_y=V^x_y$ for $x,y\in S$. Finally,
$$\mathbb{P}[e_1,\ldots,e_k\text{ are all closed}]=\det(I+(L|_{F})|_{A\times A}V_{A})^{-\alpha}.$$
\end{proof}

As a corollary, we obtain another expression by using the Poisson kernel. For $X\subset S$, define the Poisson kernel $(H^X)^x_y=\mathbb{P}^x[X_{T_{X}}=y]$ the probability of hitting $X$ at the position $y$ for a process starting from $x$.
\begin{prop}
Given a partition $\pi=\{S_1,\ldots,S_k\}$, define $\partial S_i=\{x\in S_i:\exists y\in S_i^c,Q^x_y+Q^y_x>0\}$, $F=\bigcup\limits_{i=1}^{k}\{\{x,y\}:Q^x_y+Q^y_x>0,x\in S_i,y\in S_j\}$ and $A=\bigcup\limits_{i=1}^{k}\partial S_i$. Suppose $|A|<\infty$. Define $H_{i,j}=H^{S_i^c}|_{\partial S_i\times \partial S_j}$ and
$$K=\begin{bmatrix}
0 & H_{1,2} & \cdots & H_{1,k}\\
H_{2,1} & 0 & \ddots & \vdots\\
\vdots & \ddots & \ddots & H_{k-1,k}\\
H_{k,1} & \cdots & H_{k,k-1} & 0
\end{bmatrix}.$$ Then,
$$\mathbb{P}[\mathcal{C}_{\alpha}\text{ is finer than }\pi]=\mathbb{P}[\text{all the edges in F are closed}]=(\det(I-K))^{\alpha}.$$
\end{prop}
\begin{proof}
By taking the trace of the loops on $A$, we can suppose the state space $S$ is finite and $\partial S_i=S_i$ for $i=1,\ldots,k$. By an argument similar to the argument in the above proposition, we see that
$$\mathbb{P}[\mathcal{C}_{\alpha}\text{ is finer than }\pi]=\left(\frac{\det(L')}{\det(L)}\right)^{-\alpha}$$
where $(L')^x_y=L^x_y$ for $\{x,y\}\notin F$ and $(L')^x_y=0$ for $\{x,y\}\in F$. To be more precise,
$$L'=\begin{bmatrix}
L|_{S_1\times S_1} & 0 & \cdots & 0\\
0 & L|_{S_2\times S_2} & \ddots & \vdots\\
\vdots & \ddots & \ddots & 0\\
0 & \cdots & 0 & L|_{S_k\times S_k}
\end{bmatrix}.$$
Therefore,
\begin{multline*}
\mathbb{P}[\mathcal{C}_{\alpha}\text{ is finer than }\pi]=(\frac{\det(L')}{\det(L)})^{-\alpha}=(\det((-L')^{-1}(-L)))^{\alpha}\\
=\left(\det(-\begin{bmatrix}
V^{S_1} & 0 & \cdots & 0\\
0 & V^{S_2} & \ddots & \vdots\\
\vdots & \ddots & \ddots & 0\\
0 & \cdots & 0 & V^{S_k}
\end{bmatrix}L)\right)^{\alpha}=\begin{vmatrix}
I & -H_{1,2} & \cdots & -H_{1,k}\\
-H_{2,1} & I & \ddots & \vdots\\
\vdots & \ddots & \ddots & -H_{k-1,k}\\
-H_{k,1} & \cdots & -H_{k,k-1} & I
\end{vmatrix}^{\alpha}\\
=(\det(I-K))^{\alpha}.
\end{multline*}
(Note that $V^{S^i}L|_{S_i\times S_j}=H^{S_j^c}_{S_i,S_j}=H_{i,j}$ for $i\neq j\in\{1,\ldots,k\}$.)
\end{proof}

\subsection{An example on the discrete circle}
Consider a discrete circle G with $n$ vertices $1,\ldots,n$ and $2n$ oriented edges $$E=\{(1,2),(2,3),\ldots,(n-1,n),(n,1),(2,1),(3,2),\ldots,(n,n-1),(1,n)\}$$
Define the clockwise edges set $E_{+}=\{(1,2),(2,3),\ldots,(n-1,n),(n,1)\}$ and the counter clockwise edges $E_{-}=E-E_{+}$. Consider a Markovian generator $L$ such that for any $e\in E_{+}$, $L^{e-}_{e+}=p,L^{e+}_{e-}=1-p,L^{e-}_{e-}=-(1+c)$ and $L$ is null elsewhere. Then, we have a loop measure and Poissonian ensembles associated with $L$. The rest of this subsection is devoted to study the loop cluster $\mathcal{C}_{\alpha}$ in this example.
\
\begin{lem}
Let $T_{3,n}$ be a $n\times n$ tri-diagonal Toeplitz matrix of the following form:
$$\begin{bmatrix}
a & b & 0 & \cdots & 0\\
c & a & b & \ddots& \vdots\\
0  & \ddots & \ddots & \ddots& 0 \\
 \vdots  & \ddots & c& a& b\\
 0  &\cdots  &  0  & c & a\\
\end{bmatrix}_{n\times n}.$$
Let $S_{n}$ be the following $n\times n$ matrix:
$$\begin{bmatrix}
a & b & 0 &  & c\\
c & a & b & \ddots& \\
0  & \ddots & \ddots & \ddots& 0 \\
   & \ddots & c& a& b\\
 b  & &  0  & c & a\\
\end{bmatrix}_{n\times n}.$$

Let $x_1,x_2$ be the roots of $x^2-ax+bc=0$. Then,
\begin{itemize}
\item $\displaystyle{\det(T_{3,n})=\frac{x_1^{n+1}-x_2^{n+1}}{x_1-x_2}}$,
\item $\det(S_n)=x_1^n+x_2^n+(-1)^{n+1}(b^n+c^n)$.
\end{itemize}
\end{lem}
\begin{prop}
\begin{align*}
\text{Set }x_1=&\frac{1}{2}(1+c+\sqrt{(1+c)^2-4p(1-p)}),\\
x_2=&\frac{1}{2}(1+c-\sqrt{(1+c)^2-4p(1-p)}).
\end{align*}
Then, $$\mathbb{P}[\{1,n\}\text{ is closed.}]=\left(\frac{(x_1^n-x_2^n)^2}{(x_1-x_2)(x_1^{n-1}-x_2^{n-1})(x_1^n+x_2^n-(p^n+(1-p)^n))}\right)^{-\alpha}.$$
\end{prop}
\begin{proof}
By Proposition \ref{compatibility with the time change} and Proposition \ref{Visiting a collection of compact set}
\begin{multline*}
\mathbb{P}[\{1,n\}\text{ is closed}]=e^{-\alpha\mu(N^1_n(l)+N^n_1(l)>0)}=e^{-\alpha\mu(l\text{ visits }1\text{ and }n)}\\
=\left(\frac{\det(V_{\{1,n\}})}{V^1_1V^n_n}\right)^{\alpha}=\left(\frac{\det(-L|_{\{2,\ldots,n\}\times \{2,\ldots,n\}})\det(L|_{\{1,\ldots,n-1\}\times \{1,\ldots,n-1\}})}{\det(-L|_{\{2,\ldots,n-1\}\times \{2,\ldots,n-1\}})\det(-L)}\right)^{-\alpha}\\
=\left(\frac{(x_1^n-x_2^n)^2}{(x_1-x_2)(x_1^{n-1}-x_2^{n-1})(x_1^n+x_2^n-(p^n+(1-p)^n))}\right)^{-\alpha}
\end{multline*}
where $x_1=\frac{1+c+\sqrt{(1+c)^2-4p(1-p)}}{2}$ and $x_2=\frac{1+c-\sqrt{(1+c)^2-4p(1-p)}}{2}$.
\end{proof}
\begin{prop}
Conditionally on $\{1,n\}\text{ being closed}$, $\mathcal{C}_{\alpha}$ is a renewal process conditioned to jump at time $n$. To be more precise, by deleting edges $\{1,n\}$ and adding $\{0,1\},\{n,n+1\}$, we get a discrete segment with vertices $\{0,1,\ldots,n,n+1\}$ and edges $\{\{0,1\},\ldots,\{n,n+1\}\}$. Conditionally to $\{1,n\}\text{ being closed}$, $C_{\alpha}$ induces a partition on $\{1,\ldots,n\}$. The clusters of  $C_{\alpha}$ are the intervals between the edges closed at time $\alpha$ (namely the edges which are not crossed by any loop of $\mathcal{L}_{\alpha}$). Then the left points of these closed edges, together with the left points of $\{0,1\}$ and $\{n,n+1\}$, form a renewal process conditioned to jump at $n$.
\end{prop}
\begin{proof}
Among the Poissonian loop ensembles, the ensemble of loops crossing $\{1,n\}$ and the rest are independent. Therefore, the conditional law $\mathcal{Q}$ of the loops not crossing $\{1,n\}$ conditioned on the event that no loop is crossing $\{1,n\}$ is exactly the same as the unconditioned law. Consider another Poissonian loop ensembles on $\mathbb{Z}$ driven by the following generator:
$$L^{m}_{m}=-(1+c),L^{m}_{m+1}=p,L^{m}_{m-1}=1-p\text{ for all }m\in\mathbb{Z},\text{ and $L$ is null elsewhere}.$$
Then, $\mathcal{Q}$ is the same as the conditional law of the loop ensembles contained in $\{1,\ldots,n\}$ given the condition that $\{0,1\}$ or $\{n,n+1\}$ are closed. By Proposition \ref{invariant under Doob harmonic transform}, after a harmonic transform, $L$ is modified as follows:
$$L^{m}_{m}=-(1+c),L^{m}_{m+1}=L^{m}_{m-1}=\sqrt{p(1-p)}\text{ for all }m\in\mathbb{Z},\text{ and $L$ is null elsewhere}.$$
According to Proposition 3.1 in \cite{sophie}, in the case of $\mathbb{Z}$, conditionally to the event that $\{0,1\}$ is closed, the left points of the closed edges form a renewal process. There is an obvious one-to-one correspondence between the jumps of the renewal process and the closed edges. Finally, in the case of the circle, conditioning on $\{1,n\}$ being closed, we can identify $\mathcal{C}_{\alpha}$ to a renewal process conditioned to jump at time $n$. It is not hard to see the parameter $\kappa$ in \cite{sophie} equals $\frac{1+c-2\sqrt{p(1-p)}}{\sqrt{p(1-p)}}$.
\end{proof}

\section{Loop erasure and spanning tree}
In this section we will show that Poisson processes of loops appear naturally in the construction of random spanning trees.

\subsection{Loop erasure}
Suppose $\omega$ is the path of a minimal transient canonical Markov process, then its path can be expressed as a sequence $(x_0,t_0,x_1,t_1,\ldots)$. The corresponding discrete path $(x_0,x_1,\ldots)$ is the embedded Markov chain. From the transience assumption, $\sum\limits_{n\in\mathbb{N}}1_{\{x_n=x\}}<\infty$ a.s..
\begin{defn}[Loop erasure]\label{loop erasure}
The loop erasure operation which maps a path $\omega$ to its loop erased path $\omega_{BE}$ is defined as: $\omega_{BE}=(y_0,\ldots)$ with $y_0=x_0$. Define $T_0=\inf\{n\in\mathbb{N}:\forall m\geq n, x_m\neq y_0\}$, then set $y_1=x_{T_0}$. Similarly define $T_1=\inf\{n\in\mathbb{N}:\forall m\geq n, x_m\neq y_1\}$, set $y_2=x_{T_1}$ and so on. Let $\mathbb{P}^{\nu}_{BE}$ be the image measure of $\mathbb{P}^{\nu}$ where $\nu$ is the initial distribution of the Markov process.
\end{defn}

Recall that $\partial$ is the cemetery point, that $Q^{x}_{\partial}=1-\sum\limits_{y\neq\partial}Q^{x}_{y}$ for $x\neq\partial$ and $Q^{\partial}_{x}=\delta^{\partial}_{x}$. Set $L^{x}_{\partial}=-\sum\limits_{y\neq\partial}L^x_y$ for $x\neq\partial$, $L^{\partial}_{\partial}=-1$ and $L^{\partial}_{x}=0$ for $x\neq\partial$.
\begin{prop}\label{distribution of the loop-erased random walk}
We have the following finite marginal distribution for the loop-erased random walk:
\begin{multline*}
\mathbb{P}^{\nu}_{BE}[\omega_{BE}=(x_0,x_1,\ldots,x_n,\ldots)]\\
=\nu_{x_0}\det(V_{\{x_0,\ldots,x_{n-1}\}})L^{x_0}_{x_1}\cdots L^{x_{n-1}}_{x_n}\mathbb{P}^{x_n}[T_{\{x_0,\ldots,x_{n-1}\}}=\infty]\\
=\nu_{x_0}L^{x_0}_{x_1}\cdots L^{x_{n-1}}_{x_n}\begin{vmatrix}
 V^{x_0}_{x_0} & \cdots & V^{x_0}_{x_{n-1}} & 1 \\                                                                                                                                                                                                                        \vdots & \ddots & \vdots & \vdots \\
 V^{x_{n-1}}_{x_0} & \cdots & V^{x_{n-1}}_{x_{n-1}} & 1 \\
 V^{x_n}_{x_0} & \cdots & V^{x_n}_{x_{n-1}} & 1
\end{vmatrix}.
\end{multline*}
\end{prop}
\begin{proof}
Starting from $x_n$, the probability that the Markov process never reaches the set $\{x_0,\ldots,x_{n-1}\}$,
$\mathbb{P}^{x_n}[T_{\{x_0,\ldots,x_{n-1}\}}=\infty]$ equals the same probability for the trace of the Markov process on $x_0,\ldots,x_n$, $\mathbb{P}^{x_n}_{\{x_0,\ldots,x_n\}}[T_{\{x_0,\ldots,x_{n-1}\}}=\infty]$. It equals the one step transition probability from $x_n$ to $\partial$ for the trace of the process. Let $L_{\{x_0,\ldots,x_n\}}$ be the generator of the trace of the Markov process on $\{x_0,\ldots,x_n\}$. Then, the one step transition probability from $x_n$ to $\partial$
equals $\displaystyle{\frac{(L_{\{x_0,\ldots,x_n\}})^{x_n}_{\partial}}{-(L_{\{x_0,\ldots,x_n\}})^{x_n}_{x_n}}}$.
Since $(L_{\{x_0,\ldots,x_n\}})^{x_n}_{\partial}=-(L_{\{x_0,\ldots,x_n\}})^{x_n}_{x_n}-\sum\limits_{i=0}^{n-1}(L_{\{x_0,\ldots,x_n\}})^{x_n}_{x_i}$ and $\displaystyle{-(L_{\{x_0,\ldots,x_n\}})^{x_n}_{x_i}=(-1)^{i+1+n+1}\frac{\det(V|_{\{x_0,\ldots,x_n\}\setminus\{x_i\}\times\{x_0,\ldots,x_{n-1}\}})}{\det(V_{\{x_0,\ldots,x_n\}})}}$,
we have
$$(L_{\{x_0,\ldots,x_n\}})^{x_n}_{\partial}=\frac{1}{\det(V_{\{x_0,\ldots,x_n\}})}\begin{vmatrix}
 V^{x_0}_{x_0} & \cdots & V^{x_0}_{x_{n-1}} & 1 \\                                                                                                                                                                                                                        \vdots & \ddots & \vdots & \vdots \\
 V^{x_{n-1}}_{x_0} & \cdots & V^{x_{n-1}}_{x_{n-1}} & 1 \\
 V^{x_n}_{x_0} & \cdots & V^{x_n}_{x_{n-1}} & 1
\end{vmatrix}$$
$$\text{and }-(L_{\{x_0,\ldots,x_n\}})^{x_n}_{x_n}=\begin{vmatrix}
 V^{x_0}_{x_0} & \cdots & V^{x_0}_{x_{n-1}} \\                                                                                                                                                                                                                        \vdots & \ddots & \vdots \\
 V^{x_{n-1}}_{x_0} & \cdots & V^{x_{n-1}}_{x_{n-1}} \\
\end{vmatrix}.$$
Therefore, $\mathbb{P}^{x_n}[T_{\{x_0,\ldots,x_{n-1}\}}=\infty]=\begin{vmatrix}
 V^{x_0}_{x_0} & \cdots & V^{x_0}_{x_{n-1}} & 1 \\                                                                                                                                                                                                                        \vdots & \ddots & \vdots & \vdots \\
 V^{x_{n-1}}_{x_0} & \cdots & V^{x_{n-1}}_{x_{n-1}} & 1 \\
 V^{x_n}_{x_0} & \cdots & V^{x_n}_{x_{n-1}} & 1
\end{vmatrix}
/\begin{vmatrix}
 V^{x_0}_{x_0} & \cdots & V^{x_0}_{x_{n-1}} \\                                                                                                                                                                                                                        \vdots & \ddots & \vdots \\
 V^{x_{n-1}}_{x_0} & \cdots & V^{x_{n-1}}_{x_{n-1}} \\
\end{vmatrix}$.\\
Set $D_0=\phi$ and $D_k=\{x_0,\ldots,x_{k-1}\}$ for $k\in\mathbb{N}_{+}$. Note that $Q|_{D_k^{c}\times D_k^{c}}$ is the transition probability for the process restricted in $D_k^c$. In order for the loop-erased path $\omega_{BE}$ to be $(x_0,x_1,\ldots,x_n,\ldots)$, the random walk must start from $x_0$. After some excursions back to $x_0$, it should jump to $x_1$ and never return to $x_0$. Next, after some excursions from $x_1$ to $x_1$, it jumps to $x_2$ and never returns to $x_0,x_1$, etc. Accordingly,
\begin{multline*}
\mathbb{P}^{\nu}_{BE}[\omega_{BE}=(x_0,x_1,\ldots,x_n,\ldots)]\\
=\nu_{x_0}\prod\limits_{k=0}^{n-1}(\sum\limits_{n\geq 0}((Q|_{D_k^{c}\times D_k^{c}})^n)^{x_k}_{x_k})Q^{x_k}_{x_{k+1}}\mathbb{P}^{x_n}[T_{\{x_0,\ldots,x_{n-1}\}}=\infty]\\
=\nu_{x_0}\prod\limits_{k=0}^{n-1}(V^{D_k^c})^{x_k}_{x_k}L^{x_k}_{x_{k+1}}\mathbb{P}^{x_n}[T_{\{x_0,\ldots,x_{n-1}\}}=\infty]
\end{multline*}
where $L^{D_k^c}=L|_{D_k^c\times D_k^c}$ is the generator of the Markov process restricted in $D_k^c$, and $V^{D_k^c}$ be the corresponding potential, see Definition \ref{defn:the trace of a Markov process and the restriction of a Markov process}.

Let $V_F$ stands for the sub-matrix of $V$ restricted to $F\times F$. It is also the potential of the trace of the Markov process on $F$ and let $\mathbb{P}_F$ stand for its law. Then, for all $D\subset F$, we have $(V^{D^c})_F=(V_F)^{D^c}$. In particular, for $k< n$, we have $(V^{D_k^c})^{x_k}_{x_k}=((V^{D_k^c})_{D_n})^{x_k}_{x_k}=((V_{D_n})^{D_k^c})^{x_k}_{x_k}$. One can apply Jacobi's formula
$$\det(A|_{B\times B})\det(A^{-1})=\det(A^{-1}|_{B^c\times B^c})$$
for $A=(V_{D_n})^{D_k^c}$ and $B=\{x_k\}$. To be more precise, since $((V_{D_n})^{D_k^c})^{-1}=(-L_{D_n})|_{D_k^c\times D_k^c}=(-L_{D_n})|_{(D_n-D_k)\times(D_n-D_k)}$, we have
$$(V^{D_k^c})^{x_k}_{x_k}=((V_{D_n})^{D_k^c})^{x_k}_{x_k}=\frac{\det(-L_{D_n}|_{(D_n-D_{k+1})\times (D_n-D_{k+1})})}{\det(-L_{D_n}|_{(D_n-D_k)\times (D_n-D_k)})}$$
with the convention that $\det(-L_{D_n}|_{\phi})=1$. Then,
\begin{multline*}
\prod\limits_{k=0}^{n-1}(V^{D_k^c})^{x_k}_{x_k}=\prod\limits_{k=0}^{n-1}\frac{\det(-L_{D_n}|_{({D_n}-D_{k+1})\times ({D_n}-D_{k+1})})}{\det(-L_{D_n}|_{({D_n}-D_k)\times ({D_n}-D_k)})}=\frac{\det(-L_{D_n}|_{({D_n}-D_{n})\times ({D_n}-D_{n})})}{\det(-L_{D_n}|_{({D_n}-D_{0})\times ({D_n}-D_{0})})}\\
=\frac{1}{\det(-L_{D_n})}=\det((V_{D_n})_{D_{n}})=\det(V_{\{x_0,\ldots,x_{n-1}\}}).
\end{multline*}

Finally, by combining the results above, we conclude that
\begin{multline*}
\mathbb{P}^{\nu}_{BE}[\omega_{BE}=(x_0,x_1,\ldots,x_n,\ldots)]\\
=\nu_{x_0}\det(V_{\{x_0,\ldots,x_{n-1}\}})L^{x_0}_{x_1}\cdots L^{x_{n-1}}_{x_n}\mathbb{P}^{x_n}[T_{\{x_0,\ldots,x_{n-1}\}}=\infty]\\
=\nu_{x_0}L^{x_0}_{x_1}\cdots L^{x_{n-1}}_{x_n}\begin{vmatrix}
 V^{x_0}_{x_0} & \cdots & V^{x_0}_{x_{n-1}} & 1 \\                                                                                                                                                                                                                        \vdots & \ddots & \vdots & \vdots \\
 V^{x_{n-1}}_{x_0} & \cdots & V^{x_{n-1}}_{x_{n-1}} & 1 \\
 V^{x_n}_{x_0} & \cdots & V^{x_n}_{x_{n-1}} & 1
\end{vmatrix}.
\end{multline*}
\end{proof}
\begin{rem}
Since a Markov chain in a countable space could be viewed as a pure-jump sub-Markov process with jumping rate 1, the above result holds for a sub-Markov chain if we replace $L$ by the transition matrix $Q-Id$ and $V=(Id-Q)^{-1}$.
\end{rem}
The following property was discovered by Omer Angel and Gady Kozma, see Lemma 1.2 in \cite{gady}. Here, we give a different proof as an application of Proposition \ref{distribution of the loop-erased random walk}.
\begin{prop}
Let $(X_m,m\in[0,\zeta[)$ be a discrete Markov chain in a countable space S with time life $\zeta$ and initial point $x_0$. Fix some $w\in S\setminus\{x_0\}$, define $T_1=\inf\{n>0:X_n=w\}$ and $T_N=\inf\{m>T_{N-1}:X_m=w\}$ with the convention that $\inf\phi=\infty$. We can perform loop-erasure for the path $(X_0,\ldots,X_{T_N})$, and let $LE[0,T_N]$ stand for the loop-erased self-avoiding path obtained in that way. If $T_N<\infty$ with positive probability, then the conditional law of $LE[0,T_N]$ given that $\{T_N<\infty\}$ is the same as the conditional law of $LE[0,T_1]$ given that $\{T_1<\infty\}$.
\end{prop}
\begin{proof}
We suppose $T_1<\infty$ with positive probability. By adding a small killing rate $\epsilon$ at all states and taking $\epsilon\downarrow 0$, we could suppose that we have a positive probability to jump to the cemetery point from any state. In particular, the Markov chain is transient.\\

Let $\partial$ be the cemetery point. Let $\tau(p)$ be a geometric variable with mean $1/p$, independent of the Markov chain. Let $(X^{(p)}_m,m\in[0,(\zeta-1)\wedge T_{\tau(p)}]$ be the sub-Markov chain $X$ stopped after $T_{\tau(p)}$ which is again sub-Markov. Let $\mathbb{P}^{x_0}_{p}$ stand for the law of $X^{(p)}$ and let $\mathbb{P}^{x_0}_{p,BE}$ stand for the law of the loop-erased random walk associated to $(X^{(p)}_m,m\in[0,(\zeta-1)\wedge T_{\tau(p)}])$. Let $Q^{(p)}$ be the transition matrix of $X^{(p)}$ and use the notation $Q$ for $Q^{(0)}$. Then, $(Q^{(p)})^w_i=(1-p)Q^w_i$ for $i\in S$ and $(Q^{(p)})^i_j=Q^i_j$ for $i\in S\setminus{\{w\}}$ and $j\in S$. Accordingly, $(Q^{(p)})^w_{\partial}=p+Q^w_\partial-pQ^w_\partial$. Define $V=(I-Q)^{-1}$, $V_{q\delta_w}=(M_{q\delta_w}+I-Q)^{-1}$ for $q\geq 0$ and $V^{(p)}=(I-Q^{(p)})^{-1}=(M_{(1-p)\delta_w}(I+\frac{p}{1-p}-Q))^{-1}=V_{\frac{p}{1-p}\delta_{w}}M_{\frac{1}{1-p}\delta_w}$. Set $C_{\omega}=\{$the loop-erased random walk stopped at $w\}$ Then,
\begin{align*}
C_{\Omega}=&\{\text{the random walk stopped at }w\}\\
=&\bigcup\limits_{n\geq 1}\{\text{the random walk stopped at }w\text{ at time }T_k\}\\
=&\bigcup\limits_{k\geq 1}\{\tau(p)=k,T_k<\zeta\}\bigcup\bigcup\limits_{k\geq1}\{T_k=\zeta-1,\tau(p)>k\}.
\end{align*}
For $x_n=w$,
\begin{multline*}
\mathbb{P}^{x_0}_{p,BE}[\omega_{BE}=(x_0,x_1,\ldots,x_n=w)]\\
=(Q^{(p)})^{x_0}_{x_1}\cdots (Q^{(p)})^{x_{n-1}}_{x_n}(Q^{(p)})^{x_n}_{\partial}\begin{vmatrix}
(V^{(p)})^{x_0}_{x_0} & \cdots & (V^{(p)})^{x_0}_{x_{n}}\\                                                                                                                                                                                                                      \vdots & \ddots & \vdots \\
(V^{(p)})^{x_{n}}_{x_0} & \cdots & (V^{(p)})^{x_{n}}_{x_{n}}
\end{vmatrix}\\
=\frac{p+Q^w_\partial-pQ^w_\partial}{(1-p)Q^w_\partial}Q^{x_0}_{x_1}\cdots Q^{x_{n-1}}_{x_n}Q^{x_n}_{\partial}
\begin{vmatrix}
(V_{\frac{p}{1-p}\delta_w})^{x_0}_{x_0} & \cdots & (V_{\frac{p}{1-p}\delta_w})^{x_0}_{x_{n}}\\                                                                                                                                                                                                                      \vdots & \ddots & \vdots \\
(V_{\frac{p}{1-p}\delta_w})^{x_{n}}_{x_0} & \cdots & (V_{\frac{p}{1-p}\delta_w})^{x_{n}}_{x_{n}}
\end{vmatrix}.
\end{multline*}
By the resolvent equation, $V^i_j=(V_{\frac{p}{1-p}\delta_w})^i_j+\frac{p}{1-p}(V_{\frac{p}{1-p}\delta_w})^i_wV^w_j=(V_{\frac{p}{1-p}\delta_w})^i_j+\frac{p}{1-p}(V_{\frac{p}{1-p}\delta_w})^w_jV^i_w$.
Therefore,
\begin{align*}
\begin{vmatrix}
V^{x_0}_{x_0} & \cdots & V^{x_0}_{x_{n}}\\                                                                                                                                                                                                                      \vdots & \ddots & \vdots \\
V^{x_{n}}_{x_0} & \cdots & V^{x_{n}}_{x_{n}}
\end{vmatrix}=&\begin{vmatrix}
1 & -\frac{p}{1-p}V^{x_n}_{x_0} & \cdots & -\frac{p}{1-p}V^{x_n}_{x_n}\\
(V_{\frac{p}{1-p}\delta_w})^{x_0}_{x_n} & (V_{\frac{p}{1-p}\delta_w})^{x_0}_{x_0} & \cdots & (V_{\frac{p}{1-p}\delta_w})^{x_0}_{x_{n}}\\                                                                                                                                                                                                                      \vdots & \vdots & \ddots & \vdots \\
(V_{\frac{p}{1-p}\delta_w})^{x_0}_{x_n} & (V_{\frac{p}{1-p}\delta_w})^{x_{n}}_{x_0} & \cdots & (V_{\frac{p}{1-p}\delta_w})^{x_{n}}_{x_{n}}
\end{vmatrix}\\
=&\begin{vmatrix}
1+\frac{p}{1-p}V^{x_n}_{x_n} & -\frac{p}{1-p}V^{x_n}_{x_0} & \cdots & -\frac{p}{1-p}V^{x_n}_{x_n}\\
0 & (V_{\frac{p}{1-p}\delta_w})^{x_0}_{x_0} & \cdots & (V_{\frac{p}{1-p}\delta_w})^{x_0}_{x_{n}}\\                                                                                                                                                                                                                      \vdots & \vdots & \ddots & \vdots \\
0 & (V_{\frac{p}{1-p}\delta_w})^{x_{n}}_{x_0} & \cdots & (V_{\frac{p}{1-p}\delta_w})^{x_{n}}_{x_{n}}
\end{vmatrix}\\
=&(1+\frac{p}{1-p}V^{x_n}_{x_n})\begin{vmatrix}
(V_{\frac{p}{1-p}\delta_w})^{x_0}_{x_0} & \cdots & (V_{\frac{p}{1-p}\delta_w})^{x_0}_{x_{n}}\\                                                                                                                                                                                                                      \vdots & \ddots & \vdots \\
(V_{\frac{p}{1-p}\delta_w})^{x_{n}}_{x_0} & \cdots & (V_{\frac{p}{1-p}\delta_w})^{x_{n}}_{x_{n}}
\end{vmatrix}.
\end{align*}
Accordingly,$\displaystyle{\frac{(1-p+pV^w_w)Q^w_\partial}{p+Q^w_\partial-pQ^w_\partial}\mathbb{P}^{x_0}_{p,BE}[\omega_{BE}=(x_0,x_1,\ldots,x_n=w)]}$ does not depend on p. Consequently, it must be equal to $\mathbb{P}^{x_0}_{0,BE}[\omega_{BE}=(x_0,x_1,\ldots,x_n=w)]$. Equivalently,
$$\frac{(1-p+pV^w_w)Q^w_\partial}{p+Q^w_\partial-pQ^w_\partial}\mathbb{P}^{x_0}_{p,BE}[\cdot,C_w]=\mathbb{P}^{x_0}_{0,BE}[\cdot,C_w]$$
Therefore,
$$\mathbb{P}^{x_0}_{0,BE}[C_w]=\frac{(1-p+pV^w_w)Q^w_\partial}{p+Q^w_\partial-pQ^w_\partial}\mathbb{P}^{x_0}_{p,BE}[C_w].$$
Immediately, it implies that conditionally on $C_{w}$, the law of the loop-erased random walk does not depend on $p$:
$$\mathbb{P}^{x_0}_{p,BE}[\cdot|C_w]=\mathbb{P}^{x_0}_{0,BE}[\cdot|C_w].$$
Since
\begin{align*}
\mathbb{P}^{x_0}_{p,BE}[\omega_{BE}\in\cdot,C_w]=&\begin{multlined}[t]
\sum\limits_{k\geq 1}\mathbb{P}^{x_0}[\tau(p)=k,T_k<\zeta,LE[0,T_k]\in\cdot]\\
+\sum\limits_{k\geq 1}\mathbb{P}^{x_0}[\tau(p)>k,T_k=\zeta-1,LE[0,T_k]\in\cdot]
\end{multlined}\\
=&\begin{multlined}[t]
\sum\limits_{k\geq 1}(1-p)^{k-1}p\mathbb{P}^{x_0}[T_k<\infty,LE[0,T_k]\in\cdot]\\
+\sum\limits_{k\geq 1}(1-p)^{k}\mathbb{P}^{x_0}[T_k<\infty,LE[0,T_k]\in\cdot]Q^{w}_{\partial}
\end{multlined}\\
=&\sum\limits_{k\geq 1}(1-p)^{k-1}(p+Q^w_\partial-pQ^w_\partial)\mathbb{P}^{x_0}[T_k<\infty]\mathbb{P}^{x_0}[LE[0,T_k]\in\cdot|T_k<\infty],
\end{align*}
we have
\begin{equation}\label{eqn:clerw}
\mathbb{P}^{x_0}_{p,BE}[\omega_{BE}\in\cdot|C_w]=\frac{\sum\limits_{k\geq 1}(1-p)^{k-1}\mathbb{P}^{x_0}[T_k<\infty]\mathbb{P}^{x_0}[LE[0,T_k]\in\cdot|T_k<\infty]}{{\sum\limits_{k\geq 1}(1-p)^{k-1}\mathbb{P}^{x_0}[T_k<\infty]}}.\tag{*}
\end{equation}
Since $\mathbb{P}^{x_0}_{p,BE}[\omega_{BE}\in\cdot|C_w]$ does not depend on $p\in[0,1]$, we will denote it by $\mathbb{Q}$. Then the equation \eqref{eqn:clerw} can be written as follows:
\begin{multline*}
\mathbb{Q}[\cdot]\sum\limits_{k\geq 1}(1-p)^{k-1}\mathbb{P}^{x_0}[T_k<\infty]\\
=\sum\limits_{k\geq 1}(1-p)^{k-1}\mathbb{P}^{x_0}[T_k<\infty]\mathbb{P}^{x_0}[LE[0,T_k]\in\cdot|T_k<\infty].
\end{multline*}
Finally, by identifying the coefficients, we conclude that $\mathbb{P}^{x_0}[LE[0,T_k]\in\cdot|T_k<\infty]=\mathbb{Q}[\cdot]$ as long as $\mathbb{P}^{x_0}[T_k<\infty]$ for $k\geq 1$ and we are done.
\end{proof}

Consider $(e_t,t\geq 0)$, a Poisson point process of excursions of finite lifetime at $x$ with the intensity $Leb\otimes(-L^x_x-\frac{1}{V^x_x})\nu^{x\rightarrow x}_{\{x\},ex}$. (Recall that $\nu^{x\rightarrow x}_{\{x\},ex}$ is the normalized excursion measure at $x$, see Definition \ref{defn:excursion measure outside of F}.) Let $(\gamma(t),t\geq 0)$ be an independent Gamma subordinator\footnote{See Chapter III of  \cite{BertoinMR1406564}.} with the Laplace exponent
$$\Phi(\lambda)=\int\limits_{0}^{\infty}(1-e^{-\lambda s})s^{-1}e^{-s/V^x_x}\,ds.$$
Let $\mathfrak{R}_{\alpha}$ be the closure of the image of the subordinator $\gamma$ up to time $\alpha$, i.e. $\mathfrak{R}_{\alpha}=\overline{\{\gamma(t):t\in[0,\alpha]\}}$. Then, $[0,\gamma(\alpha)]\setminus \mathfrak{R}_{\alpha}$ is the union of countable disjoint open intervals, $\{]\gamma(t-),\gamma(t)[:t\in[0,\alpha],\gamma(t-)<\gamma(t)\}$. To such an open interval $]g,d[$, one can associate a based loop $l$ as follows:
 During the time interval $]g,d[$, the Poisson point process $(e_t,t\geq 0)$ has  finitely many excursions, namely $e_{t_1},\cdots,e_{t_n},g<t_1<\cdots<t_n<d$. Each excursion $e_{t_i}$ is viewed as a c\`{a}dl\`{a}g path of lifetime $\zeta_{t_i}$: $(e_{t_i}(s),s\in[0,\zeta_{t_i}[)$. Define $l:[0,d-g+\sum\limits_{i}\zeta_{t_i}]\rightarrow S$ as follows:
$$l(s)=\left\{
\begin{array}{ll}
e_{t_i}(s-(\sum\limits_{j<i}\zeta_{t_j}+t_i)) & \text{ if }s\in [\sum\limits_{j<i}\zeta_{t_j}+t_i,\sum\limits_{j\leq i}\zeta_{t_j}+t_i[\\
x & \text{otherwise.}
\end{array}
\right.$$
This mapping between an open interval $]g,d[$ and a based loop $l$ depends on $]g,d[$ and $(e_t,t\in]g,d[)$ and we denote is by $\Psi^{]g,d[}$ ($l=\Psi^{]g,d[}(e)$).
By mapping a based loop into a loop, we get a countable collection of loops for $\alpha\geq 0$, namely $\mathcal{O}_{\alpha}$.
\begin{prop}\label{Poisson point process of loops} $(\mathcal{O}_{\alpha},\alpha\geq 0)$ has the same law as the Poisson point process of loops intersecting $\{x\}$, i.e. $(\{l\in\mathcal{L}_{\alpha}:l^x>0\},\alpha>0)$.
\end{prop}
\begin{proof}
As both sides have independent stationary increment, it is enough to show $\mathcal{O}_{1}=\{l\in\mathcal{L}_{1}:l^x>0\}$. It is well-known that $(\gamma(t)-\gamma(t-),t\in\mathbb{R})$ is a Poisson point process with characteristic measure $\frac{1}{s}e^{-s/V^x_x}\,ds$. Therefore, $\sum\limits_{l\in\mathcal{O}_{\alpha}}\delta_{l^x}$ is Poisson random measure with intensity $\frac{1}{s}e^{-s/V^x_x}\,ds$. On the other hand, for the Poisson ensemble of loops $\mathcal{L}_{\alpha}$, by taking the trace of the loops on $x$ and dropping the empty ones, as a consequence of Proposition \ref{compatibility with the time change}, we get a Poisson ensemble of trivial Markovian loops with intensity measure $\frac{1}{s}e^{s(L_{\{x\}})^x_x}\,ds$ where $(-L_{\{x\}})^x_x=1/V^x_x$. Consequently, we have
$$\{l^x:l\in\mathcal{O}_{1}\}\text{ has the same law as }\{l^x:l\in\mathcal{L}_{1},l^x>0\}$$
In other words, by disregarding the excursions attached to each loop, the set of trivial loops in $x$ obtained from $\mathcal{O}_1$ and $\{l\in\mathcal{L}_1:l^x>0\}$ is the same. In order to restore the loops, we need to insert the excursions into the trivial loops. Then, it remains to show that the excursions are inserted into the trivial loops in the same way. Finally, by using the independence between $(e_t,t\geq 0)$ and $(\gamma(t),t\geq 0)$ and the stationary independent increments  property with respect to time $t$, it ends up in proving the following affirmation:
$\Psi^{]0,T[}(e)$ induces the same probability on the loops with $l^x=T$ as the loop measure conditioned by $\{l^x=T\}$.
 By Proposition \ref{relation between the bridge measure and loop measure 1}, we have $l^x\mu(dl)=\mu^{x,x}(dl)$. Hence, $\mu(dl|l^x=T)=\mu^{x,x}(dl|l^x=T)$ where $\mu^{x,x}$ is considered to be a loop measure. Let $\mathbb{P}^x$ be the law of the Markov process $(X_t,t\in[0,\zeta[)$ associated with the Markovian loop measure $\mu$. Let $(L(x,t),t\in[0,\zeta[)$ be the local time process at $x$ and $L^{-1}(x,t)$ be its right-continuous inverse. Let $\tau$ be an independent exponential variable with parameter 1. Define the process $X^{L^{-1}(x,\tau)}$ with lifetime $L^{-1}(x,\tau)\wedge\zeta$ as follows: $X^{L^{-1}(x,\tau)}(T)=X(T),T\in[0,L^{-1}(x,\tau)\wedge\zeta[$. Denote by $\mathbb{Q}[dl]$ the law of $X^{L^{-1}(x,\tau)}$. Then, $e^{-l^x}\mu^{x,x}(dl)=\mathbb{Q}[dl]$ where $\mu^{x,x}(dl)$ is considered to be a based loop measure. Therefore, $$\mu^{x,x}(dl|l^x=T)=\mathbb{Q}[dl|l^x=T]=\mathbb{Q}[dl|\tau=T]=\text{the law of }\Psi^{]0,T[}(e)$$ in the sense of based loop measures. Then, the equality stills holds for loop measures and we are done.
\end{proof}

Suppose $(X_t,t\in[0,\zeta[)$ is a transient Markov process on $S$. (Assume the process stays at the cemetery point after lifetime $\zeta$.) Define the local time at $x$ $L(x,t)=\int\limits_{0}^{t}1_{\{X_s=x\}}\,ds$. Denote by $L^{-1}(x,t)$ its right-continuous inverse and by $L^{-1}(x,t-)$ its left-continuous inverse. The excursion process $(e_t,t\geq 0)$ is defined by $e_t(s)=X_{s+L^{-1}(x,t-)},s\in[0,L^{-1}(x,t)-L^{-1}(x,t-)[$. Define a measure on the excursion which never returns to $x$ by
$$\tilde{\nu}^{x\rightarrow}[dl]=\sum\limits_{y\in S}Q^x_y\mathbb{P}^{y}[dl,\text{the process never hits }x].$$
We can calculate the total mass of $\tilde{\nu}^{x\rightarrow}$ as follows:
\begin{multline*}
\tilde{\nu}^{x\rightarrow}[1]=\sum\limits_{y\in S}Q^x_y\mathbb{P}^{y}[\text{the process never hits }x]\\
=1-\mathbb{E}^{x}[\{\text{after leaving }x\text{, }\text{the process returns to }x\}]\\
=(1-(R^{\{x\}})^x)=\frac{(L_{\{x\}})^x_x}{L^x_x}=-\frac{1}{V^x_xL^x_x}.
\end{multline*}
After normalization, we get a probability measure $\nu^{x\rightarrow}$. The law of the first excursion is $-\frac{1}{V^x_xL^x_x}\nu^{x\rightarrow}+\left(1+\frac{1}{V^x_xL^x_x}\right)\nu^{x\rightarrow x}_{\{x\},ex}$. In particular, the first excursion is not an excursion from $x$ back to $x$ with probability $-\frac{1}{L^x_xV^x_x}$.
According to the excursion theory, the excursion process is a Poisson point process stopped at the appearing of an excursion of infinite lifetime or an excursion that ends up at the cemetery. The characteristic measure is proportional to the law of the first excursion. By taking the trace of the process on $x$, we know that the total occupation time is an exponential variable with parameter $(-L_{\{x\}})^x_x=\frac{1}{V^x_x}$. According to the excursion theory, it is an exponential variable with parameter $\frac{-d}{V^x_xL^x_x}$, $d$ being the mass of the characteristic measure. Immediately, we get $d=-L^x_x$. If we focus on the process of excursions from $x$ back to $x$, it is a Poisson point process with characteristic measure $(-L^x_x-\frac{1}{V^x_x})\nu^{x\rightarrow x}_{\{x\},ex}$ stopped at an independent exponential time with parameter $\frac{1}{V^x_x}$. Let $(\gamma(t),t\geq 0$) be a Gamma subordinator with Laplace exponent $$\Phi(\lambda)=\int\limits_{0}^{\infty}(1-e^{-\lambda s})s^{-1}e^{-s/V^x_x}\,ds.$$ Then, $\gamma(t)$ follows the $\Gamma(t,\frac{1}{V^x_x})$ distribution with density $\rho(y)=\frac{1}{\Gamma(t)(V^x_x)^t}y^{t-1}e^{-y/V^x_x}$. In particular, $\gamma(1)$ is an exponential variable of the parameter $1/V^x_x$. It is known that  $(\frac{\gamma(t)}{\gamma(1)},t\in[0,1])$ is independent of $\gamma(1)$, and that it is a Dirichlet process. (One can prove this by a direct calculation on the finite marginal distribution.) Moreover, the jumps of the process $(\frac{\gamma(t)}{\gamma(1)},t\in[0,1])$ rearranged in decreasing order follow the Poisson-Dirichlet $(0,1)$ distribution.
For $x\in S$, let $Z_x$ be the last passage time in $x$: $Z_x=\sup\{t\in[0,\zeta[:X(t)=x\}$. Suppose the loop erased path $\omega_{BE}$ equals $(x_1,\ldots)$. Define $S_n=T_{x_n}$ for $n\geq 1$ and $S_0=0$. Let $O_i$ be $(X_s,s\in[S_i,S_{i+1}[)$ i.e. the i-th loop erased from the process $X$. Then $O_1$ can be viewed as a Poisson point process $(e^{(1)}_t,t\in[0,L(x_1,\zeta)[)$ of excursions at $x_1$ killed at the arrival of an excursion with infinite lifetime or an excursion ending up at the cemetery. Conditionally on $\omega_{BE}=(x_1,x_2,\ldots)$, the shifted process $(X(s+T_1),s\in[0,\zeta[)$ is the Markov process restricted in $S\setminus\{x_1\}$ starting from $x_2=X(T_1)$ . Moreover, it is conditionally independent of the Poisson point process $e^{(1)}$. Once again, we can view $O_2$ as an killed Poisson point process of excursions at $x_2$ and denote it by $e^{(2)}$. Clearly, we have the independence between $e^{(1)}$ and $e^{(2)}$ conditionally on $\omega_{BE}=(x_1,x_2,\ldots)$. Repeating this procedure, we get a sequence of point process of excursions $e^{(1)},\ldots$. Conditionally on $\omega_{BE}=(x_1,\ldots,x_n,\ldots)$, they are independent, and $e^{(n)}$ has the same law as the killed excursion process for the Markov process restricted in $D_n=S\setminus\{x_1,\ldots,x_{n-1}\}$. Let $O_i^{x_i}$ be the occupation time at $x_i$ for the based loop $O_i$. Let $(\gamma^{(i)}_t,t\geq 0),i\geq 1$ be a sequence of independent Gamma subordinators with Laplace exponent
$$\Phi(\lambda)=\int\limits_{0}^{\infty}(1-e^{-\lambda s})s^{-1}\exp\left(-\frac{s}{(V^{D_n})^{x_i}_{x_i}}\right)\,ds.$$ We suppose they are independent of the Markov process.
Then, $O_i^{x_i},i\geq 1$ has the same law as $\gamma^{(i)}(1),i\geq 1$ conditionally on $\omega_{BE}$.
 In the spirit of Proposition \ref{Poisson point process of loops} by cutting the excursion process according to the range of subordinator, if at time $\alpha\in[0,1]$, we cut the loop $O_i$ according to the range of $(\frac{\gamma^{(i)}(s)O_i^{x_i}}{\gamma^{(i)}(1)},s\in[0,\alpha])$, we get a point process of loops $(\mathcal{O}^{(i)}_{\alpha},\alpha\in[0,1])$. Conditionally on $\omega_{BE}$, it has the same law as the Poisson point process $(\mathcal{L}_\alpha^{D_i}\setminus\mathcal{L}_{\alpha}^{D_{i+1}},\alpha\in[0,1])$. Moreover, conditionally on $\omega_{BE}$, $(\mathcal{O}^{(i)}_{\alpha},\alpha\in[0,1]),i\geq 1$ are independent. The definition of the Poisson random measure ensures independence among $(\mathcal{L}_\alpha^{D_i}\setminus\mathcal{L}_{\alpha}^{D_{i+1}},\alpha\in[0,1]),i=1,\ldots$. Consequently, we have the following proposition.
\begin{prop}\label{reconstruct the Poisson ensemble of loops}
Conditionally on $\omega_{BE}$, $(\mathcal{O}_{\alpha},\alpha\in[0,1])$ has the same law as $(\{l\in\mathcal{L}_{\alpha}:l\text{ intersects }\omega_{BE}\},\alpha\in[0,1])$.
\end{prop}
\begin{rem}\label{rem:from subordinator to poisson-dirichlet}
The jumps of the process $\displaystyle{\frac{\gamma(t)}{\gamma(1)}}$ rearranged in decreasing order follow the Poisson-Dirichlet $(0,1)$ distribution. Since a Poisson point process is always homogeneous in time, the following two cutting method gives the same loop ensemble in law:
\begin{itemize}
\item Cutting the loop according to the range of $\displaystyle{\left(\frac{\gamma(t)}{\gamma(1)},t\in[0,1]\right)}$,
\item Cutting the loop according to the Poisson-Dirichlet $(0,1)$ distribution.
\end{itemize}
As a result, a similar result holds for $\alpha=1$ if we cut the loops according to the Poisson-Dirichlet $(0,1)$ distribution.
\end{rem}
\subsection{Random spanning tree}
Throughout this section, we consider a finite state space $S$ with a transient Markov process $(X_t,t\geq 0)$ on it. Denote by $\Delta$ the cemetery point for $X$. As usual, denote by $L$ the generator of $X$ and by $Q$ the transition matrix of the embedded Markov chain.

By the following algorithm, one can construct a random spanning tree of $S\cup\{\Delta\}$ rooted\footnote{By a random spanning tree rooted at $\Delta$, we mean a random spanning tree with a special mark on the vertex $\Delta$.} at $\Delta$. We give an orientation on the tree: each edge is directed towards the root.
\begin{defn}[Wilson's algorithm]
Choose an arbitrary order on $S$:  $S$=$\{v_1,\ldots,v_n\}$. Define $S_0=\{\Delta\}$. Let $T_0$ be the tree with single vertex $\Delta$. We recurrently construct a series of growing random trees $T_k,k\in\mathbb{N}$ as follows:\\
Suppose $T_k$ is well-constructed with set of vertices $S_k$. If $S\cup\{\Delta\}\setminus S_k=\phi$, then we stop the procedure and set $\mathcal{T}=T_k$. Otherwise, there is a unique vertex in $S\cup\{\Delta\}\setminus S_k$ with the smallest sub-index and we denote it by $y_{k+1}$. Run a Markov chain from $y_{k+1}$ with transition matrix $Q$. It will hit $S_{k}$ in finitely many steps. We stop the Markov chain after it reaches $S_{k}$ and erase progressively the loops according to the Definition \ref{loop erasure}. In this way, we get a loop-erased path $\eta_{k+1}$ joining $y_{k+1}$ to $T_k$. By adding this loop erased path $\eta_{k+1}$ to $T_k$, we construct the random tree $T_{k+1}$. The procedure will stop after a finite number of steps and it produces a random spanning tree $\mathcal{T}$.
\end{defn}

\begin{prop}\label{spanning tree measure with a fixed root for a finite graph}
Denote by $\mu_{ST,\Delta}$ the distribution of the random spanning tree rooted at $\Delta$ given by Wilson's algorithm. Then,
$$\mu_{ST,\Delta}(\mathcal{T}=T)=\det(V)1_{\{T\text{ is a spanning tree rooted at $\Delta$}\}}\prod\limits_{\begin{subarray}{l}
(x,y)\text{ is an edge in }T\\
\text{ directed towards the root }\Delta
\end{subarray}}L^{x}_{y}\quad\footnote{Recall that $L^{x}_{\delta}=-\sum\limits_{y\in S}L^x_y$ for $x\in S$.}$$
where $V$ is the potential of the process $X$ \footnote{Wilson's algorithm use the embedded Markov chain of $X$.}.
\end{prop}
\begin{proof}
Suppose $|S|=n$. Choose an arbitrary order on $S$: $S$=$\{v_1,\ldots,v_n\}$ and use Wilson's algorithm to construct a random spanning tree $\mathcal{T}$ rooted at $\Delta$. Set $v_0=\Delta$. For a rooted spanning tree $T$, let $A_m(T)$ be the set of vertices in $T_{\{v_0,\ldots,v_m\}}$\footnote{Here, $T_{\{v_0,\ldots,v_m\}}$ is the smallest sub-tree of $T$ containing the same root with the set of vertices $v_0,\ldots,v_m$.} for $m=1,\ldots,n$. Set $B_0(T)=\phi$. For $m=1,\ldots,n$, set $B_m(T)=\phi$ if $v_m$ belongs $A_{m-1}(T)$. Otherwise, let $B_m(T)$ be the unique path joining $v_m$ to $A_{m-1}(T)$ in $T$.
We will calculate the conditional distribution of $B_m(\mathcal{T})$ given $A_{m-1}(\mathcal{T})$ for $m\geq 1$. Suppose that $v_m\notin A_{m-1}$. Let $(Y_t, t\geq 0)$ be the process $(X_t,t\geq 0)$ killed at the first jumping time after the process reaches the $A_{m-1}$. Then, $Y$ is a transient Markov with generator
$$(L_Y)^x_y=\left\{\begin{array}{ll}
L^x_y & \text{ for }x\text{ not contained in }\mathcal{T}_{\{v_0,\ldots,v_{m-1}\}}\\
\delta^x_yL^x_x & \text{ otherwise}
\end{array}\right.
$$
and potential $V_Y$ such that
\begin{itemize}
\item $V_Y|_{A_{m-1}^c\times A_{m-1}^c}=V^{A_{m-1}^c}$;
\item $(V_Y)^x_y=\sum\limits_{z\in A_{m-1}^c}(V^{A_{m-1}^c})^x_zL^z_y$ for $x\in A_{m-1}^c,y\in A_{m-1}$;
\item $V_Y|_{A_{m-1}\times A_{m-1}^c}=0$;
\item $(V_Y)^x_y=\delta^x_y\frac{1}{-L^x_x}$ for $x,y\in A_{m-1}$.
\end{itemize}
 Let $\partial_Y$ stand for the cemetery point of $Y$. Then conditionally on $\mathcal{T}_{v_0,\ldots,v_{m-1}}$, the probability $B_m=((z_0,z_1),(z_1,z_2)\ldots,(z_p,z_{p+1}))$ with $z_0=v_m,z_{p+1}\in A_{m-1}$ and $z_0,\ldots,z_p\in A_{m-1}^c$ equals the probability that the loop-erased path obtained by $Y$ is $(z_0,z_1,\ldots,z_p,z_{p+1},\partial_Y)$. According to Proposition \ref{distribution of the loop-erased random walk}, that conditional probability equals

$$\det((V_Y)_{A_m\setminus A_{m-1}})\prod\limits_{\mathclap{(x,y)\text{ is contained in }B_m}} L^x_y=\det((V^{A_{m-1}^c})_{A_m\setminus A_{m-1}})\prod\limits_{\mathclap{(x,y)\text{ is contained in }B_m}}L^x_y.$$
By Jacobi's formula,
$$\det(-L|_{A_{m-1}^c\times A_{m-1}^c})\det((V^{A_{m-1}^c})_{A_{m}\setminus A_{m-1}})=\det(-L|_{A_m^c\times A_m^c}).$$
Accordingly,
$$\det((V^{A_{m-1}^c})_{A_{m}\setminus A_{m-1}})=\frac{\det(V^{{A_{m-1}^c}})}{\det(V^{A_m^c})}.$$
Therefore, if $v_m\notin A_{m-1}$, i.e. $A_{m-1}\neq A_m$,
$$\mathbb{P}[B_m=((z_0,z_1),\ldots,(z_p,z_{p+1}))|A_{m-1}]=\frac{\det(V^{{A_{m-1}^c}})}{\det(V^{A_m^c})}\prod\limits_{(x,y)\text{ is contained in }{B_m}}L^x_y.$$
Trivially, if $v_m\in A_{m-1}$,
$$\mathbb{P}[B_m=\phi|A_{m-1}]=1=\frac{\det(V^{{A_{m-1}^c}})}{\det(V^{A_m^c})}$$
Finally, by multiplying all the conditional probability above, we find that
$$\mu_{ST,\Delta}(\mathcal{T}=T)=\det(V)1_{\{T\text{ is a spanning tree rooted at }\Delta\}}\prod\limits_{\begin{subarray}{l}
(x,y)\text{ is an edge in }T\\
\text{ directed towards the root }\Delta
\end{subarray}}L^x_y.$$
\end{proof}
\begin{rem}
In Wilson's algorithm, the spanning tree is constructed by progressively adding new branches. For $k\in\mathbb{N}$, conditionally on the tree $T_k$ that has been constructed at step $k$, the law of the new branch $\eta_{k+1}$ to be added is associated with the Markov process $X$ stopped at the next jump after reaching $T_k$. At the same time, we remove $\#T_{k+1}-\#T_{k}$ loops based on each vertex of $\eta_{k+1}$. If we partition these loops according to some independent Poisson-Dirichlet (0,1) distribution as in Proposition \ref{reconstruct the Poisson ensemble of loops} and Remark \ref{rem:from subordinator to poisson-dirichlet} and we get an ensemble of loops $\mathcal{O}_{\eta_{k+1}}$. Conditionally on $\mathcal{T}$, $\mathcal{O}_k$ is equal in law to $\mathcal{L}^{\{T_k\}^c}_{1}\setminus\mathcal{L}^{\{T_{k+1}\}^c}_{1}$. (By $\mathcal{L}^{\{T_k\}^c}_{1}$, we mean those loops in $\mathcal{L}_{1}$ that avoid the vertices of the tree $T_k$.) These $(\mathcal{O}_{\eta_{k}},k\geq 1)$ are independent and it is the same for $(\mathcal{L}^{\{T_k\}^c}_{1}\setminus\mathcal{L}^{\{T_{k+1}\}^c}_{1},k\geq 1)$. It implies that $\bigcup\limits_{k\geq 1}\mathcal{O}_{\eta_{k}}$ has the same law as $\mathcal{L}_{1}=\bigcup\limits_{k\geq 1}\mathcal{L}^{\{T_{k-1}\}^c}_{1}\setminus\mathcal{L}^{\{T_{k}\}^c}_{1}$. To summarize, in the Wilson's algorithm, we have removed $\#S$ loops based at each vertex of $S$ . By partitioning all these loops according to independent Poisson-Dirichlet (0,1) distributions as in Proposition \ref{reconstruct the Poisson ensemble of loops} and Remark \ref{rem:from subordinator to poisson-dirichlet}, we recover the Poisson ensemble of loops $\mathcal{L}_{1}$ .
\end{rem}

\nocite{markovianbridge}
\bibliographystyle{amsalpha}
\bibliography{reference}
\end{document}